\documentclass[11pt,reqno]{article}   
\usepackage{fullpage}

\usepackage{microtype}
\usepackage[unicode=true]{hyperref}
\usepackage{amsmath}
\usepackage{cite}
\usepackage{amsfonts}
\usepackage{amssymb}
\usepackage{amsthm}
\usepackage{authblk}
\usepackage{mypack} 
\usepackage[pdftex]{color,graphicx}
\usepackage{adjustbox}
\usepackage{soul}
\usepackage{comment}
\usepackage{bbold}
\usepackage{tikz} 
\usetikzlibrary{decorations.pathreplacing,angles,quotes}
\usetikzlibrary{matrix,arrows,decorations.pathmorphing}
\usepackage{tikz-cd}
\usetikzlibrary{arrows}
\usetikzlibrary{decorations.markings}

\usepackage{bbold}
\usepackage{diagbox}
\usepackage{caption}
\usepackage{subcaption}
\usepackage{pdfpages} 
\usepackage{blkarray}
\usepackage{centernot}
\usepackage{mathtools}
\usepackage{stmaryrd}
\usepackage{soul}  

\definecolor{bluegray}{rgb}{0.4, 0.6, 0.8}
\definecolor{turquoise}{rgb}{0.2, 0.7, 0.6}
\newcommand{\Image}{\operatorname{Im}}
\newcommand{\catCat}{\mathbf{Cat}}

\newcommand{\bSect}{\mathsf{bSect}}
\newcommand{\Sect}{\mathsf{Sect}}

\title{A bundle perspective on contextuality
\\
{\Large{Empirical models and simplicial distributions on bundle scenarios}}}

\author[1]{Rui Soares Barbosa\footnote{rui.soaresbarbosa@inl.int}}
\author[2]{Aziz Kharoof\footnote{aziz.kharoof@bilkent.edu.tr}}
\author[2]{Cihan Okay\footnote{cihan.okay@bilkent.edu.tr}}
\affil[1]{{\small{INL -- International Iberian Nanotechnology Laboratory, Braga, Portugal}}}
\affil[2]{{\small{Department of Mathematics, Bilkent University, Ankara, Turkey}}}

\date{\today}

\begin{document}
  \maketitle

\begin{abstract}
This paper provides a bundle perspective to contextuality by introducing new categories of contextuality scenarios based on bundles of simplicial complexes and simplicial sets.
The former approach generalizes earlier work on the sheaf-theoretic perspective on contextuality, and the latter extends simplicial distributions, a more recent approach to contextuality formulated in the language of simplicial sets.
After constructing our bundle categories, we also construct functors that relate them and natural isomorphisms that allow us to compare the notions of contextuality formulated in two languages.
We are motivated by applications to the resource theory of contextuality, captured by the morphisms in these categories.
In this paper, we develop the main formalism and leave applications to future work.
\end{abstract}

\tableofcontents

\section{Introduction}

This paper develops a bundle perspective on contextuality that generalizes the sheaf-theoretic~\cite{abramsky2011sheaf} and the simplicial~\cite{okay2022simplicial} approaches.
In \cite{karvonen2018categories,barbosa2023closing}, the former approach is used to construct the category $\catScen$ of scenarios, whose morphisms capture a notion of simulation corresponding to the free transformations in a resource theory of contextuality.
In this paper, we introduce two categories: (1) the category $\catbScen$ of bundle scenarios based on simplicial complexes, and
(2) the category $\catsScen$ of simplicial bundle scenarios based on simplicial sets. 
The use of bundles of simplicial complexes to represent measurements and outcomes in contextuality scenarios already appears in the literature; see \cite{abramsky2015contextuality,beer2018contextuality,terra2019measures}. 
However, the notion of ``bundle" is used somewhat informally and not fully formalized.
We provide an abstract definition of a bundle scenario and define morphisms between such bundle scenarios extending the category of scenarios $\catScen$. 
The other category $\catsScen$ is based on simplicial sets, combinatorial objects representing spaces more expressive than simplicial complexes.
Our motivation for such extensions is to translate the resource-theoretic analysis developed in \cite{karvonen2018categories,abramsky2019comonadic,barbosa2023closing,karvonen2021neither} to bundle scenarios of simplicial complexes and ultimately to bundle scenarios of simplicial sets.  
The extension to simplicial sets carries many advantages for studying contextuality; see~\cite{okay2022simplicial,kharoof2022simplicial,
okay2022mermin,kharoof2023topological}. It offers promising applications in extending cohomological frameworks for measurement-based quantum computation developed in~\cite{raussendorf2016cohomological,raussendorf2023putting}.

This paper focuses on setting the formal stage, relating the three categories of scenarios and distributions on them.
Resource-theoretic implications will be studied in upcoming work.
More formally, we construct fully faithful functors
\[\eE \colon \catScen \to \catbScen \qquad \text{ and } \qquad N \colon \catbScen \to \catsScen\]
relating the three categories of scenarios and natural isomorphisms $\eta$ and $\zeta$ relating the empirical models and simplicial distributions on them:
\begin{equation}\label{dia:-eta-zeta}
\begin{tikzcd}
\catScen \arrow[rd,"\eE"] \arrow[dd,"\Emp"',""{name=A,right}] &  \\
&  \catbScen \arrow[Rightarrow, from=A,"\eta"] \arrow[dl,"\bEmp"]  \\
\catSet&
\end{tikzcd}
\;\;\;\;\;\;\;\;\;
\;\;\;\;\;\;\;\;\;
\begin{tikzcd}
\catbScen \arrow[rd,"N"] \arrow[dd,"\bEmp"',""{name=A,right}] &  \\
&  \catsScen \arrow[Rightarrow, from=A,"\zeta"] \arrow[dl,"\sDist"]  \\
\catSet &
\end{tikzcd}
\end{equation}

In the sheaf-theoretic framework, a scenario is described by a pair $(\Sigma,O)$ consisting of a simplicial complex $\Sigma$ and a family $O=\set{O_x}_{x\in \Sigma_0}$ of sets indexed by the vertices of the simplicial complex.
Each vertex $x$ represents a measurement with outcomes in the set $O_x$, and each simplex $\sigma\in \Sigma$ represents a measurement context, a set of measurements that can be jointly performed.
It is convenient to represent the simplicial complex $\Sigma$ as a category $\catC_\Sigma$
whose objects are the simplices of $\Sigma$ and the morphisms are given by inclusions.
Then both measurements and the corresponding outcomes can be described as a functor
$$
\eE_O:\catC_\Sigma \to \catSet
$$
that sends a measurement context $\sigma$ to the set $\eE_O(\sigma)=\prod_{x\in \sigma} O_x$ of all possible outcome assignments. 
A morphism $(\Sigma,O) \to (\Sigma',O')$ between two scenarios
consists of two parts: a functor $\pi:\catC_{\Sigma'}\to \catC_\Sigma$ induced by a simplicial relation that relates the measurement contexts, and a natural transformation $\alpha:\eE_O\circ \pi \to \eE_{O'}$ that specifies a mapping of outcomes over each context.
Since measurements in a context can be simultaneously performed, outcome statistics are represented by a probability distribution $e_\sigma$ on the set $\eE_O(\sigma)$ of joint outcomes. The family $e=\set{e_\sigma}_{\sigma\in \Sigma}$ of distributions, called an empirical model, is required to satisfy an additional condition known as the non-signaling (or no-disturbance) condition: the measurement statistics over two contexts should 
match on their intersection.
Morphisms of the category of scenarios are defined so that an empirical model $e$ on $(\Sigma,O)$ can be pushed forward along a morphism to produce a new empirical model on $(\Sigma',O')$. Operationally this coincides with the notion of simulating the empirical model $e$ with the new one. By the push-forward construction, empirical models on scenarios can be turned into a functor
$$
\Emp: \catScen \to \catSet
$$
Since we can consider probabilistic mixtures of empirical models, the target can be replaced by the category $\catConv_R$ of convex sets for our choice of a semiring $R$, which is usually taken to be non-negative reals.
 
\begin{figure}[h!]
\centering
\includegraphics[width=.7\linewidth]{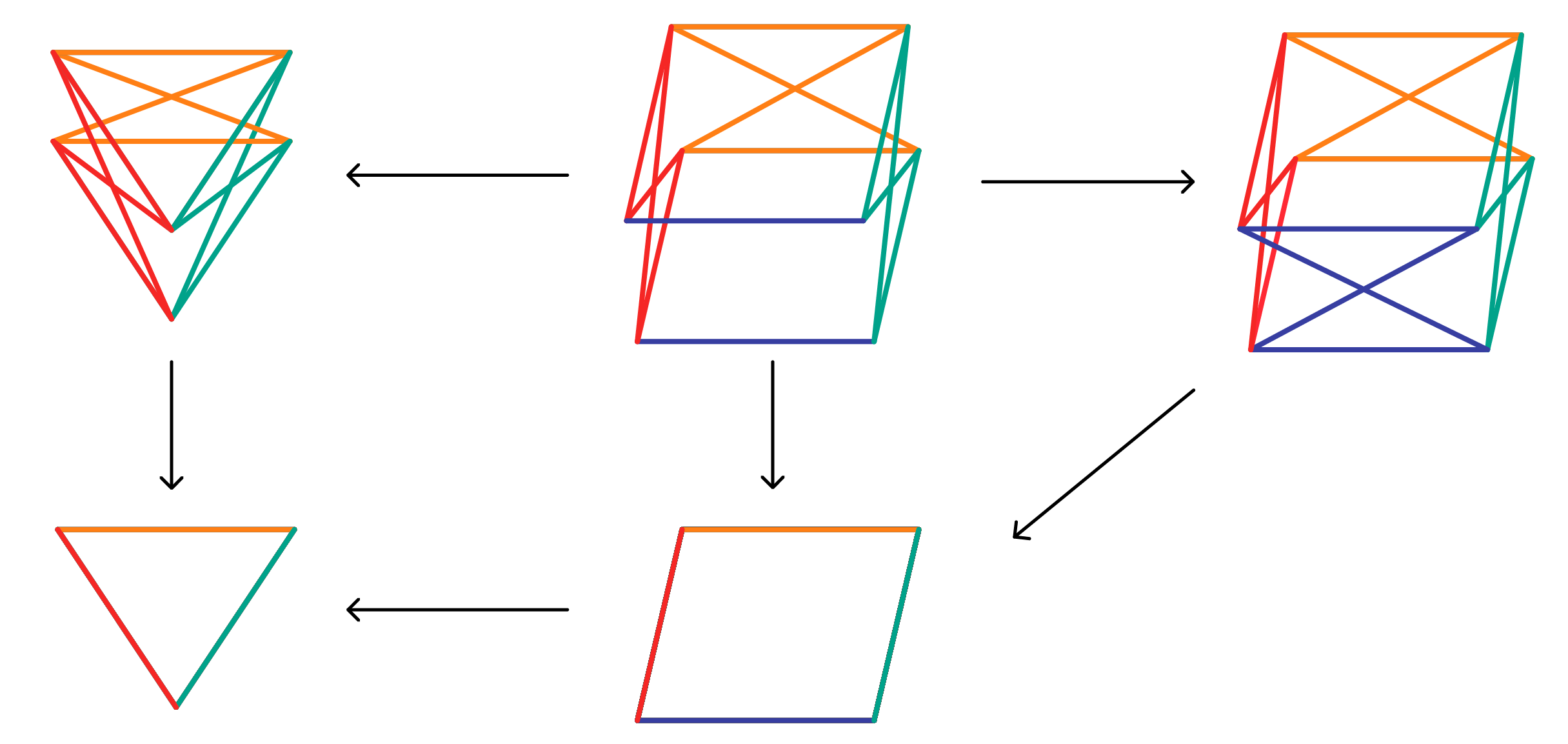}
\caption{{A morphism  between two bundle scenarios on a triangle and a square.}
}
\label{fig:bundle}
\end{figure} 
 
Our first contribution in this paper is the notion of bundle scenario: A map $f:\Gamma\to \Sigma$ of simplicial complexes is a bundle scenario if it is surjective and satisfies the following two properties:
\begin{itemize}
\item  Locally surjectivity: The restriction of $f$ to the star of each simplex of $\Sigma$ is surjective.
\item Discrete over vertices: Every two vertex of $\Gamma$ that map to the same vertex {under $f$ does not form an edge}.
\end{itemize}
To a scenario $(\Sigma,O)$ one can associate a bundle scenario
$$\eE_O(\Sigma)\to \Sigma$$
by assembling the outcomes over each simplex into a simplicial complex that maps down to the simplicial complex of measurement contexts. 
Morphisms in bundle scenarios are described using diagrams of simplicial complexes; see Figure (\ref{fig:bundle}). {Empirical models can be generalized to bundle scenarios and can be assembled into a functor
$$
\bEmp: \catbScen \to \catSet .
$$
}

These constructions in the category of simplicial complexes can be carried over to simplicial sets, which are suitable for generating spaces with more elaborate identifications.
A simplicial set is specified by a sequence of sets $\set{X_n}_{n\geq 0}$ representing the simplicies together with two kinds of maps for gluing and collapsing simplices.
The notions of local surjectivity and being discrete over vertices can be generalized to simplicial sets giving rise to the notion of simplicial bundle scenarios: A simplicial set map $f:E\to X$ is called a simplicial bundle scenario if it is surjective, locally surjective and discrete over vertices.
A source of examples is obtained by applying the nerve construction, the functor $N$ in Diagram (\ref{dia:-eta-zeta}), to a bundle scenario.  
One of the benefits of extending the category of scenarios first to bundle scenarios and then the simplicial scenarios is that the description of morphisms simplifies. A morphism $f\to f'$ between two simplicial bundle scenarios is a commutative diagram
$$
\begin{tikzcd}[column sep=huge,row sep=large]
E \arrow[d,"f"] &  \pi^*(E)  \arrow[l] \arrow[d] \arrow[r,"\alpha"] & E' \arrow[d,"f'"] \\
X & \arrow[l,"\pi"] X' \arrow[r,equal] & X' 
\end{tikzcd}
$$
where the left square is a pull-back. This diagram is similar to Figure (\ref{fig:bundle}). However, in the simplicial complex case, an additional complication is added by the use of the nerve complex functor $\hat N$. This functor is analogous to the nerve functor $N$, but produces a simplicial complex instead of a simplicial set. Intuitively its role can be understood by observing that simplicial complex maps $\Sigma' \to \hat N(\Sigma)$ coincide with simplicial relations.
{We introduce simplicial distributions on simplicial bundle scenarios generalizing empirical models even further giving us a functor
$$
\sDist: \catsScen \to \catSet
$$
The definition of a simplicial distribution is formulated in a diagrammatic way (see Definition \ref{def:simp-dist-revisited}) rather than using the language of sheaves.
}

Our contributions can be summarized as follows:
\begin{itemize}

\item We introduce bundle scenarios in Definition \ref{def:S-LS-DV-BS}. For any scenario $(\Sigma,O)$ we show that $\eE_O(\Sigma)\to \Sigma$ is a bundle scenario (Proposition \ref{pro:canonical-bundle}). Morphisms between bundle scenarios are introduced in Definition \ref{def:bScen}. An important perspective here is the use of the nerve complex functor $\hat N$ given in Definition \ref{def:hatN} to interpret simplicial relations.

\item Empirical models for bundle scenarios are introduced in Definition \ref{def:emprical-bundle}. Non-signaling conditions are formulated using the property of being discrete over vertices. We construct the push-forward distribution along a morphism in Definition \ref{def:push-forward-bundle-mor}. Our main result in this section, {Theorem \ref{thm:main-bundle}},  states that the functor $\eE$ in Diagram (\ref{dia:-eta-zeta}) is fully faithful and $\eta$ is a natural isomorphism.

\item Simplicial bundle scenarios are introduced in Definition \ref{def:simplicial-DV-LS-BS}. We show that $Nf:N\Gamma \to N\Sigma$ is a simplicial bundle scenario whenever $f:\Gamma\to \Sigma$ is a bundle scenario (Proposition \ref{pro:nerve-pres-bundle-scen}).
Morphisms of simplicial bundle scenarios are easier to define than morphisms of bundle scenarios (Definition \ref{def:sScen}).

\item Empirical models can be generalized to the simplicial set framework to give rise to simplicial distributions (Definition \ref{def:simp-dist-revisited}). 
Theorem \ref{thm:main-simplicial} is our main result on simplicial bundle scenarios stating that the functor $N$ in Diagram (\ref{dia:-eta-zeta}) is fully faithful and $\zeta$ is a natural isomorphism.

\item For the discussion of contextuality, we use the theory of convex categories developed in \cite{kharoof2022simplicial}. 
The key ingredient is the observation that in Diagram (\ref{dia:-eta-zeta}), the  categories of scenarios can be replaced by the corresponding free convex categories:
\begin{equation}\label{dia:upgraded}
\begin{tikzcd}
D_R(\catScen) \arrow[rd,"D_R(\eE)"] \arrow[dd,"\widetilde\Emp"',""{name=A,right}] &  \\
&  D_R(\catbScen) \arrow[Rightarrow, from=A,"\tilde \eta"] \arrow[dl,"\bEmp"]  \\
\catConv_R&
\end{tikzcd}
\;\;\;\;\;\;\;\;\;
\;\;\;\;\;\;\;\;\;
\begin{tikzcd}
D_R(\catbScen) \arrow[rd,"D_R(N)"] \arrow[dd,"\widetilde\bEmp"',""{name=A,right}] &  \\
&  D_R(\catsScen) \arrow[Rightarrow, from=A,"\tilde\zeta"] \arrow[dl,"\widetilde\sDist"]  \\
\catConv_R &
\end{tikzcd}
\end{equation}
In Corollary \ref{cor:upgrade-convex} we show that $D_R(\eE)$ and $D_R(N)$ are fully faithful. Moreover, $\tilde \eta$ and $\tilde \zeta$ are natural isomorphisms.

\item We introduce contextuality for the three categories of scenarios in Definition \ref{def:Context} and Definition \ref{def:Contextbbbb}. 
Theorem \ref{thm:contextuality} shows that these three notions coincide when specialized to the target of the natural isomorphisms $\tilde \eta$ and $\tilde \zeta$. 
\end{itemize}

Most of the technical results are proved in the Appendix. Section \ref{sec:simplicial-complexes} begins with a description of the conditions of local surjectivity (Proposition \ref{pro:equivalent-LS}) and being discrete over vertices (Proposition \ref{pro:equivalent-DV}) in terms of liftings of diagrams. In Section \ref{subsec: Pullback}, we describe pull-backs of 
{bundle scenarios}
and show that they behave particularly well (Proposition \ref{pro:pull-back-bundle}). In Proposition \ref{pro:N-bundle-sce}, we show that the $\hat N$ functor preserves bundle scenarios.
Section \ref{sec:lem-category-bundle-scen} contains technical results on the interaction of pull-backs with the $\hat N$ functor and Lemma \ref{lem:bScen-is-well-defined}, which shows that the composition rule defined for the category of bundle scenarios is associative. Section \ref{sec:lem-embedding-into-bundle-scen} contains technical results that are used for proving that $\eE$ is fully faithful (Proposition \ref{pro:fully-faithful-bundle}). Section \ref{sec:lem-empirical-bundle-scen} is concerned with constructing empirical models on $\hat N$ of a bundle scenario and how to push forward empirical models along morphisms.
In Section \ref{sec:simplicial-sets}, we begin by discussing the basic properties of the nerve functor $N$. In Section \ref{sec:lem-categoryof-simplicial-scen}, we prove that $N$ preserves local surjectivity (Lemma \ref{lem:nerve-S-LS}) and being discrete over vertices (Lemma \ref{lem:nerve-pre-dv}). Section \ref{sec:lem-embedding-into-simplicial-scen} and Section \ref{sec:lem-simplicial-dist} contain technical results for proving that $N$ is a fully faithful functor (Proposition \ref{pro:fully-faithful-simplicial}) and $\zeta$ a natural isomorphism (Proposition \ref{pro:natural-iso-simplicial}).
In Section \ref{sec:ConvCat} we recall some facts from \cite{kharoof2022simplicial} concerning convex categories.
A basic construction is the free convex category given in Definition \ref{def:Free CatFun}.
An important result, which is also of independent interest, is Proposition \ref{pro:ConvisConv} stating $\catConv_R$ is an $R$-convex category.
 Diagram (\ref{dia:upgraded}) is obtained by a general observation (Proposition \ref{pro:Nattrans}) on liftings of natural transformations to convex categories.

\paragraph{Acknowledgments.}
RSB and CO are supported by the Digital Horizon Europe project FoQaCiA, GA no. 101070558, AK and CO by the US Air Force Office of Scientific Research under award number FA9550-21-1-0002, and RSB by FCT -- Funda\c{c}\~ao para a Ci\^encia e a Tecnologia through CEECINST/00062/2018.
 
\section{The category of scenarios} \label{sec:category-of-scenarios}

In this section we introduce the category $\catScen$ of scenarios and the functor $\Emp:\catScen \to \catSet$ of empirical models following the presentation in \cite{barbosa2023closing}.

A {\it scenario}  consists of a pair $(\Sigma,O)$, where
\begin{itemize}
\item $\Sigma$ is a simplicial complex whose vertex set $\Sigma_0$ represents measurements and its simplices represent contexts,
\item $O$ is a family $\set{O_x}_{x\in \Sigma_0}$ of nonempty sets representing the outcomes of each measurement. 
\end{itemize} 
We can think of $\Sigma$ as a category $\catC_\Sigma$ whose objects are the simplices $\sigma \in \Sigma$ and morphisms are inclusions $\sigma' \hookrightarrow \sigma$. 
On this category we can define the event presheaf, that is, the contravariant functor 
$$\eE_{O}: \catC_\Sigma \to \catSet$$ 
defined by
$
\eE_{O}(\sigma) = \prod_{x\in \sigma} O_x
$.
{For $s\in \eE_{O}(\sigma)$  we write $s|_{\sigma'}$ for the tuple in $\eE_{O}(\sigma')$ obtained by restricting to the indices from the subset.}
Sometimes for we will write $\eE_{(\Sigma,O)}$ {for this functor} to indicate the underlying simplicial complex.

To define morphisms {between scenarios}, 
we need the notion of simplicial relation.
For a relation $\pi\subset \Sigma_0'\times \Sigma_0$ we will write 
$$
\pi(x')=\set{x\in \Sigma_0:\, (x',x)\in \pi} \;\;\text{ and }\;\;
\pi(\sigma')=\cup_{x'\in \sigma'} \pi(x').
$$
A simplicial relation $\pi:\Sigma'\to \Sigma$ is a relation $\pi\subset \Sigma'_0\times \Sigma_0$ such that  
$\pi(\sigma')\in \Sigma$ for all $\sigma'\in \Sigma'$. 
A simplicial relation induces a functor 
$\pi:\catC_{\Sigma'}\to \catC_\Sigma$.  

\begin{defn}\label{def:scen}
{\rm
The {\it category $\catScen$ of scenarios} consists of objects given by scenarios $(\Sigma,O)$ and a morphism $(\Sigma,O) \to (\Sigma',O')$ between two scenarios is given by a pair $(\pi,\alpha)$ where
\begin{itemize}
\item $\pi:\Sigma' \to \Sigma$ is a simplicial relation.
\item $\alpha:\eE_{O}\circ \pi \to \eE_{O'}$ is a natural transformation of functors.
\end{itemize} 
}
\end{defn}

The composition of the morphisms $(\pi,\alpha):(\Sigma,O)\to (\Sigma',O')$ and $(\pi',\alpha'):(\Sigma',O')\to (\Sigma'',O'')$ 
is given by  
\begin{equation}\label{eq-scen-composition}
(\pi',\alpha')\circ (\pi,\alpha)= (\pi \circ \pi' , \alpha' \circ (\alpha \ast \Id_{\pi'})).
\end{equation}
where
\begin{itemize}
\item $\pi \circ \pi':\Sigma''\to \Sigma$ denotes the simplicial relation obtained {by} the composition of the corresponding functors, and
\item $\alpha \ast \Id_{\pi'}:\eE_{O}\circ (\pi\circ \pi')\to {\eE_{O'}  \circ \pi ' }$ denotes the horizontal composition of natural transformations \cite[Section 1.7]{riehl2017category}.
\end{itemize}

\begin{figure}[h!]
\centering
\begin{subfigure}{.49\textwidth}
  \centering
  \includegraphics[width=.5\linewidth]{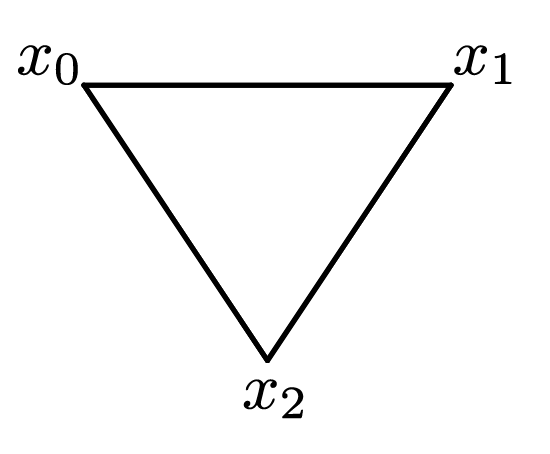}
  \caption{}
  \label{fig:triangle}
\end{subfigure}%
\begin{subfigure}{.49\textwidth}
  \centering
  \includegraphics[width=.5\linewidth]{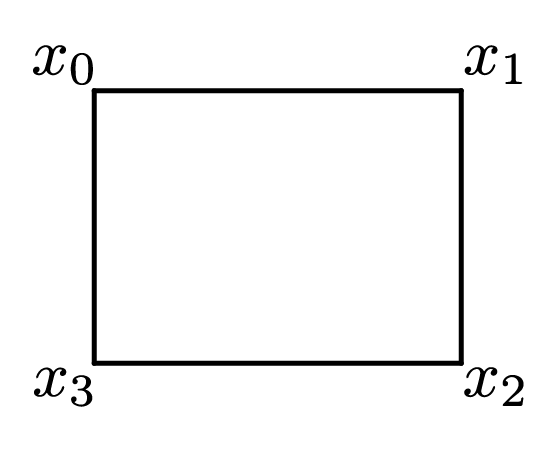}
  \caption{}
  \label{fig:square}
\end{subfigure}%
\caption{
}
\label{fig:example}
\end{figure}

\Ex{\label{ex:triagle-square}
Let $\ZZ_2=\set{0,1}$ denote the two-outcome set.
We will consider a morphism $(\pi,\alpha): (\triangle,O) \to (\square,O)$ where
\begin{itemize}
\item $\triangle$ is the simplicial complex with vertex set $\set{x_0,x_1,x_2}$ {(Figure \ref{fig:triangle})} and maximal simplices
$$
\set{x_0,x_1},\;\; \set{x_1,x_2},\;\; \set{x_2,x_0};
$$ 
\item $\square$ is the simplicial complex with vertex set $\set{x_0,x_1,x_2,x_3}$ {(Figure \ref{fig:square})}  and maximal simplices
\begin{equation}\label{eq:Sq-simplices}
\set{x_0,x_1},\;\; \set{x_3,x_0},\;\; \set{x_1,x_2},\;\; \set{x_2,x_3}.
\end{equation}
\end{itemize}
The outcome sets $O_x$ are given by $\ZZ_2$ for the vertices of $\triangle$ and $\square$.  
The simpicial relation $\pi$ is a simplicial complex map defined by  
$$
\pi(x_i) = \left\lbrace
\begin{array}{ll}
x_i & i=0,1 \\
x_2 & i=2,3,
\end{array}
\right.
$$
and $\alpha_x:\eE_O(\pi(x)) \to \eE_O(x)$ is the identity map on $\ZZ_2$.  
}

\subsection{Empirical models}
{Throughout the paper $R$ will denote a zero-sum-free commutative semiring. The zero-sum-free condition means that $a+b=0$ implies $a=b=0$ for $a,b\in R$.}
Let $D_R:\catSet\to \catSet$ denote the distribution monad \cite{jacobs2010convexity}. Under this functor a set $U$ is sent to the set $D_R(U)$ of functions $p:U\to R$ of finite support such that $\sum_{u\in U} p(u)=1$. For $f:U\to V$ the function $D_R(f): D_R(U)\to D_R(V)$ is defined by 
$$
D_R(f)(p)(v) = \sum_{u\in f^{-1}(v)} p(u).
$$
An {\it empirical model} $e$ on the scenario $(\Sigma,O)$ is a family $\set{e_\sigma \in 
D_R \circ \eE_O(\sigma) }_{\sigma\in \Sigma}$ of distributions that satisfies the compatibility condition: for each inclusion $i:\sigma' \hookrightarrow \sigma$ of simplices we have
$$
e_{\sigma'} = e_\sigma|_{\sigma'}
$$
where $e_\sigma|_{\sigma'}$ stands for $(D_R \circ \eE_O)(i)(e_\sigma)$. 
We denote the set of empirical models on $(\Sigma,O)$ by $\Emp(\Sigma,O)$.

\Def{\label{def:emp}
 Let $e$ be an empirical model on the scenario $(\Sigma,O)$.
Given a morphism $(\pi,\alpha):(\Sigma,O)\to (\Sigma',O')$
the {\it push-forward empirical model} $(\pi,\alpha)_*e$ on $(\Sigma',O')$ is defined by
$$
((\pi,\alpha)_*e)_{\sigma'} = D_R(\alpha_{\sigma'})(e_{\pi(\sigma')})
$$
where $\alpha_{\sigma'}: \eE_{O}(\pi(\sigma')) \to \eE_{O'}(\sigma')$ is the component of the natural transformation $\alpha$ at the simplex $\sigma'{\in \Sigma'}$.
}

The push-forward construction gives a function $(\pi,\alpha)_*:\Emp(\Sigma,O)\to \Emp(\Sigma',O')$.

\Pro{[\!\cite{barbosa2023closing}]
\label{pro:Emp} 
The assignments $(\Sigma,O)\mapsto \Emp(\Sigma,O)$ and $(\pi,\alpha) \mapsto (\pi,\alpha)_*$ define a functor $\Emp:\catScen \to \catSet$.
} 

\Ex{\label{ex:empirical}
{Let us consider the morphism $(\pi,\alpha):(\triangle,O)\to (\square,O)$ in Example \ref{ex:triagle-square}.}
Given an empirical model on $(\triangle,O)$, the push-forward empirical model $(\pi,\alpha)_*e$ {is given by
$((\pi,\alpha)_*e)_{\sigma'}=e_{\pi(\sigma')}$} when $\sigma'$ is one of the first three simplices in Equation (\ref{eq:Sq-simplices}), and for the last one it is given by
$$
((\pi,\alpha)_*e)_{\set{x_2,x_3}}(a,b) = \left\lbrace
\begin{array}{ll}
e_{\set{x_2}}(a) & a=b \\
0 & \text{otherwise.} 
\end{array}
\right.
$$
}

\section{The category of bundle scenarios} \label{sec:bScen}


{
In this section we introduce bundle scenarios of simplicial complexes and morphisms between them.
The resulting category $\catbScen$ is the category of bundle scenarios. We extend empirical models to bundle scenarios and assemble them into a functor $\Emp:\catbScen \to \catSet$. Our main result is Theorem \ref{thm:main-bundle} proving that in the left-hand diagram in (\ref{dia:-eta-zeta}) the functor $\eE$ is fully faithful and $\eta$ is a natural isomorphism.
Most of the technical results concerning simplicial complexes are presented in Section \ref{sec:simplicial-complexes}.
}

In the context of simplicial complexes   locality is  captured by the star construction:
The {\it star of a simplex} $\sigma\in \Sigma$ is defined by
$$
\St(\sigma) = \set{\sigma'\in \Sigma:\, \sigma\subset \sigma'}.
$$
 
\Def{\label{def:S-LS-DV-BS}
A map of $f:\Gamma\to \Sigma$ simplicial complexes  is called
\begin{itemize}
\item {\it surjective} if it is surjective as a function between the set of simplices,
\item {\it locally surjective} if $f|_\gamma:\St(\gamma) \to \St(f(\gamma))$ is surjective for all $\gamma\in \Gamma$,
\item {\it discrete over vertices} 
if $\{x,y\} \notin \Gamma$ for every distinct $x,y \in \Gamma_0$ with 
$f(x)=f(y)$.
\end{itemize}
A {\it bundle scenario} is a surjective, locally surjective map $f: \Gamma \to \Sigma$ of simplicial complexes that is discrete over vertices.
}
The {two properties, local surjectivity and being discrete over vertices,}
can be characterized using lifting conditions (Propositions \ref{pro:equivalent-LS} and \ref{pro:equivalent-DV}). {Therefore bundle scenarios can also be described in terms of lifting conditions (Corollary \ref{cor:equivalent-bundle}).  
This point of view will be fruitful when extending these notions to simplicial sets in Section \ref{sec:category-simplicial-scenarios}.}

Our canonical example of a bundle scenario comes from a scenario $(\Sigma,O)$.  
Associated to this scenario we can construct a bundle scenario
$$f_{(\Sigma,O)}:\eE_O\Sigma \to \Sigma$$
where   $\eE_O\Sigma$ is the simplicial complex with vertex set $(\eE_O\Sigma)_0= \sqcup_{x\in \Sigma_0} O_x$ and the set of simplices 
$$
\eE_O\Sigma=\set{(\sigma,s):\, \sigma\in \Sigma,\; s\in \prod_{x\in \sigma} O_x}.
$$ 
Each simplex $(\sigma,s)$ is regarded as the subset $\sqcup_{x\in \sigma} s(x)\subset (\eE_O\Sigma)_0$.
The simplicial complex map $f_{(\Sigma,O)}$ is given by  projection onto the first coordinate.

\Pro{\label{pro:canonical-bundle} The simplicial complex map
$f_{(\Sigma,O)}:\eE_O\Sigma \to \Sigma$ is a bundle scenario.
}
\Proof{ 
The outcome sets $O_x$ being nonempty implies that $f=f_{(\Sigma,O)}$  is surjective.
For a simplex $(\sigma,s)$ consider the map between the stars
$$
f_{(\sigma,s)}: \St(\sigma,s) \to \St(\sigma).
$$
Given an element of $\St(\sigma)$, that is a simplex $\sigma'$ containing $\sigma$, consider $s'\in \prod_{x\in \sigma'}O_{x}$ such that $s'|_\sigma =s$. Then $(\sigma',s') \in \St(\sigma,s)$ and {it} maps to $\sigma'$ under $f_{(\sigma,s)}$. {This} implies that $f$ is locally surjective.
Finally,   $f$ is discrete over vertices since for $x\in \Sigma_0$ and distinct $o,o' \in O_x$ we have $\{(x,o),(x,o')\} \notin \eE_O\Sigma$.
}

There is a category-theoretic description of $\eE_O\Sigma$ that uses the notion of the category of elements of a functor.

\Def{\label{def:category-of-elements}
Given a contravariant functor $F:\catC\to \catSet$ the {\it category of elements}, denoted by $\catEl(F)$, has objects given by pairs $(c,x)$ where $c$ is an object of $\catC$ and $x\in F(c)$. 
A morphism $(c',x') \to (c,x)$ is given by a morphism $f:c'\to c$ of $\catC$ such that $F(f)(x)=x'$. 
}

When this construction is applied to the functor $\eE_O:\catC_\Sigma \to \catSet$ associated to  a scenario $(\Sigma,O)$ the category $\catEl(\eE_O)$ of elements  {coincides with $\catC_{\eE_O\Sigma}$.}

\subsection{Morphisms of bundle scenarios}\label{sec:morphisms-bundle-scen}

{
To introduce morphisms between bundle scenarios we need an alternative point of view on the notion of simplicial relation introduced in Section \ref{sec:category-of-scenarios}.
The key construction is the nerve complex functor 
$$\hat N:s\catComp \to s\catComp$$ 
defined on the category of simplicial complexes:

\Def{\label{def:hatN}
{\rm
Given a simplicial complex $\Sigma$ the {\it nerve complex} $\hat N \Sigma$ is the simplicial complex defined as follows:
\begin{itemize}
\item The vertex set is given by $(\hat N\Sigma)_0=\Sigma$.
\item The set of simplices is $\hat N\Sigma = \set{ \set{\sigma_1,\sigma_2,\cdots,\sigma_k}:\, \cup_{i=1}^k\sigma_i \in \Sigma }$.
\end{itemize}
A simplicial complex map $f:\Gamma\to \Sigma$ induces a simplicial complex map $\hat Nf:\hat N\Gamma \to \hat N\Sigma$ between the nerve complexes defined on vertices by
$(\hat Nf)(\gamma)=f(\gamma)$.
}}


The functor $\hat N$ is in fact a monad, i.e., comes with two natural transformations: 
\begin{itemize}
\item  $\delta:\idy \to \hat N$ whose component $\delta_\Sigma:\Sigma \to \hat N\Sigma$ is defined by $\delta_{\Sigma}(x)=\{x\}$, and
\item $\mu:\hat N\circ \hat N \to \hat N$ whose component $\mu_\Sigma: \hat N^2\Sigma \to \hat N\Sigma$ is defined by $\mu_{\Sigma}(\{\sigma_1,\dots,\sigma_n\})=\cup_{i=1}^n\sigma_i$.
\end{itemize}
Associated to a monad one can consider the Kleisli category \cite{mac2013categories}. For the nerve complex functor the Kleisli category $\catsComp_{\hat N}$ consists of simplicial complexes as objects and its morphisms are given by
$$
\catsComp_{\hat N}(\Sigma',\Sigma) = \catsComp(\Sigma',\hat N\Sigma).
$$

\Pro{\label{pro:simplicial-rel-Kleisli}
The set of simplicial relations $\pi:\Sigma'\to \Sigma$ is in bijection with the set of morphisms $\Sigma'\to \Sigma$ in the Kleisli category $\catsComp_{\hat N}$.
}
\Proof{
A simplicial complex map $\Sigma'\to \hat N\Sigma$ is given by a function $\pi:\Sigma'_0\to \Sigma$ such that $\pi(\sigma')=\cup_{x'\in \sigma'} \pi(x')$ is a simplex of $\Sigma$ for all $\sigma'\in \Sigma'$. This is precisely a simplicial relation.
}

We will regard simplicial relations as morphisms in the Kleisli category.  
Given  simplicial complex maps   $\pi:\Sigma'\to \hat N \Sigma$ and  $\pi':\Sigma''\to \hat N\Sigma'$ we write $\pi\kleisli \pi':\Sigma''\to \hat N\Sigma$ for the Kleisli composition, i.e., the composition in the category $\catsComp_{\hat N}$.

%
%
%
The final ingredient for the definition of morphisms between bundle scenarios is pull-backs of bundle scenarios. In the category of simplicial complexes the pull-back of a bundle scenario has a particularly nice description (Proposition \ref{pro:pull-back}).
}

The morphisms of  
the category $\catbScen$ 
are constructed from two types of morphisms:
\begin{itemize}
\item A {\it type I morphism} $(\pi,\idy):f\to f^\pi$ is given by a pull-back square 
\begin{equation}\label{dia:typeI}
\begin{tikzcd}[column sep=large,row sep=large]
\hat N  \Gamma \arrow[d,"\hat N f"'] & \Gamma' \arrow[d," f^\pi"] \arrow[l,"l"'] \\
\hat N \Sigma & \Sigma' \arrow[l,"\pi"]
\end{tikzcd} 
\end{equation}
To indicate that this is a pull-back we will write 
$\Gamma'=\pi^* (\hat N \Gamma)$, and sometimes we will use the notation $l_f$ or even $l_{f,\pi}$ instead of $l$.
\item A {\it type II morphism} $\alpha:f\to f'$ is given by a commutative diagram
\begin{equation}\label{dia:typeII}
\begin{tikzcd}
\Gamma  
 \arrow[rd,"f"'] \arrow[rr,"\alpha"]&& \Gamma'
 \arrow[dl,"f'"] \\
& \Sigma  & 
\end{tikzcd}
\end{equation}
\end{itemize}
The pull-back of a bundle scenario is also a bundle scenario (Proposition \ref{pro:pull-back-bundle}). The simplicial complex map $\hat Nf$ in Diagram (\ref{dia:typeI}) is a bundle scenario when $f$ is a bundle scenario (Proposition \ref{pro:N-bundle-sce}). Therefore $f^\pi$ is {also a} bundle scenario.

The composition of type II morphisms is clear, {whereas for type I morphisms we have the following definition.}

\begin{defn}\label{def:composition-T1}
{\rm
The composition of the following type I morphisms 
\begin{equation*} 
\begin{tikzcd}[column sep=large,row sep=large]
\hat N \Gamma  
 \arrow[d,"\hat N f"'] &  \Gamma' 
 \arrow[d,"f'"] \arrow[l,"l"']& \hat N\Gamma'   \arrow[d,"\hat N f'"'] & \Gamma''
 \arrow[d,"f''"] \arrow[l,"l'"']\\
\hat N\Sigma  & \Sigma' \arrow[l,"\pi"] & \hat N \Sigma' & \Sigma'' \arrow[l,"\pi'"]
\end{tikzcd}
\end{equation*}
is defined to be the composition of the following squares
\begin{equation}\label{diag:composition-T1}
\begin{tikzcd}[column sep=large,row sep=large]
\hat N \Gamma  
 \arrow[d,"\hat N f"'] &  \hat N^2\Gamma 
 \arrow[d,"\hat N^2f"'] \arrow[l,"\mu_{\Gamma}"']& \hat N\Gamma'   \arrow[l,"\hat N l"'] \arrow[d,"\hat N f'"'] & \Gamma''
 \arrow[d,"f''"] \arrow[l,"l'"']\\
\hat N\Sigma  & \hat N^2\Sigma \arrow[l,"\mu_{\Sigma}"] & \hat N \Sigma' \arrow[l,"\hat N \pi"]& \Sigma'' \arrow[l,"\pi'"]
\end{tikzcd}
\end{equation}
We will denote this composition by $(\pi\kleisli \pi',\idy)$.
}
\end{defn} 
{The composition of the squares in Diagram (\ref{diag:composition-T1}) is a valid type I morphism since the composition of the left and the middle square is a pull-back diagram (Lemma \ref{lem:pull-hat-N}). Composing two pull-back squares gives a pull-back square hence the composition of the resulting square with the third square is also a pull-back square.}

\begin{defn}\label{def:bScen}
{\rm
The {\it category $\catbScen$ of bundle scenarios} consists of objects given by bundle scenarios $f:\Gamma\to \Sigma$. A morphism $f\to f'$ is given by  
a pair $(\pi,\alpha)$ of simplicial complex maps making the diagram commute 
\begin{equation}\label{dia:bundle-morphisms}
\begin{tikzcd}[column sep=large,row sep=large]
\hat N\Gamma \arrow[d,"\hat Nf"] &  \pi^*(\hat N\Gamma)  \arrow[l] \arrow[d,"f^\pi"] \arrow[r,"\alpha"] & \Gamma' \arrow[d,"f'"] \\
\hat N\Sigma & \arrow[l,"\pi"] \Sigma' \arrow[r,equal] & \Sigma' 
\end{tikzcd}
\end{equation}
where $\pi^*(\hat N\Gamma)$ is the pull-back along $\pi$. The composition of $(\pi_1,\alpha_1):f\to f'$ and $(\pi_2,\alpha_2):f'\to f''$ is defined as follows:
\begin{equation}\label{eq:composition-bundle}
(\pi_2,\alpha_2)\circ (\pi_1,\alpha_1) = (\pi_1\kleisli \pi_2, \alpha_2\circ \pi_2^*(\hat N \alpha_1))
\end{equation}
where the map 
$$\pi_2^*(\hat N \alpha_1) : \pi_2^*(\hat N \pi_1^*(\hat N\Gamma)) \to \pi_2^*(\hat N \Gamma')$$ 
is the pull-back of $\hat N\alpha_1$ in the diagram
$$
\begin{tikzcd} 
\hat N\pi_1^*(\hat N\Gamma) \arrow[rr,"\hat N\alpha_1"] \arrow[dr,"\hat Nf^{\pi_1}"'] && \hat N\Gamma' \arrow[dl,"\hat N f'"] \\
& \hat N\Sigma' &
\end{tikzcd}
$$
along {the simplicial complex map $\pi_2:\Sigma''\to \hat N\Sigma'$}.
}
\end{defn}
The identity morphism of the object $f:\Gamma \to \Sigma$ is given by 
$$
\begin{tikzcd}[column sep=large,row sep=large]
\hat N\Gamma \arrow[d,"\hat N f"] & \Gamma \arrow[l,"\delta_\Gamma"'] \arrow[d,"f"] \arrow[r,"\idy_{\Gamma}"] & \Gamma \arrow[d,"f"] \\
\hat N\Sigma & \Sigma \arrow[l,"\delta_\Sigma"] \arrow[r,equal] & \Sigma
\end{tikzcd}
$$
This is a valid morphism since the left square {is a pull-back diagram (Lemma \ref{DDDeltsqu}).} 
Definition \ref{def:bScen} gives a well-defined category since the composition rule is associative (Lemma \ref{lem:bScen-is-well-defined}).

\begin{ex}\label{ex:Theexample}
The scenarios $(\triangle,O)$ and $(\square,O)$ in Example \ref{ex:triagle-square} can be realized as bundle scenarios $f_\triangle:\eE_O \triangle \to \triangle$ and $f_\square:\eE_O \square \to \square$. 
{Recall that the simplicial relation $\pi:\square \to \triangle$ is in fact a simplicial complex map.
In effect the simplicial relation in {Diagram (\ref{dia:bundle-morphisms})} can be replaced by this simplicial complex map when computing the pull-back.}
The morphism in this example gives a morphism $(\pi,\alpha): f_\triangle \to f_\square$ of bundle scenarios 
$$
\begin{tikzcd}[column sep=large,row sep=large]
\eE_O \triangle \arrow[d] & \arrow[l] \pi^* \eE_O \triangle \arrow[d] \arrow[r,"\alpha"] &  \eE_O \square \arrow[dl] \\
\triangle & \arrow[l,"\pi"] \square & 
\end{tikzcd}
$$
{This is depicted in Figure \ref{fig:bundle}.}
\end{ex}

{A simplicial complex map $\pi: \Sigma' \to \hat N \Sigma$ induces a functor 
$$\overline{\pi}: \catC_{\Sigma'} \to \catC_{\Sigma}$$ 
defined by $\overline{\pi}(\sigma')=\cup_{
\sigma \in \pi(\sigma')} \sigma$. Then the Kleisli composition can be seen as composition of functors
\begin{equation}\label{eq:pibarComp}
\overline{\pi \kleisli \pi'}=
\overline{\pi} \circ \overline{\pi'}.
\end{equation}}
Given a scenario $(\Sigma,O)$ 
the category of elements construction (Definition \ref{def:category-of-elements})
for the functor $\mathcal{E}_O \circ \overline{\pi}$ specifies a   simplicial complex:
$$
\catC_{(\eE_O\circ \pi)\Sigma'}= \catEl(\eE_O\circ \overline{\pi}).
$$
Explicitly, the simplicial complex $(\eE_O\circ \pi)\Sigma'$  has
\begin{itemize}
\item vertices $(x',s)$, where $x' \in \Sigma'_0$ and $s \in \mathcal{E}_O(\pi(x'))$, and 
\item simplices $(\sigma',s)$, where $\sigma'\in \Sigma'$ and $s\in \eE_O(\overline{\pi}(\sigma'))$.
\end{itemize}
{Note that $\pi(x')=\overline{\pi}(\set{x'})$. Projection onto the first coordinate gives a simplicial complex map}
$$f_{\eE_O\circ \pi}:(\eE_O\circ \pi)\Sigma' \to \Sigma'$$
Observe that for $t \in \mathcal{E}_O(\pi(x'))$,  there exists a morphism   $(\set{x'},t)\to (\sigma',s)$ in  $\catC_{(\eE_O\circ \pi)\Sigma'}$
if and only if $x' \in \sigma'$ and $s|_{\pi(x')}=t$. Therefore we can identify $(\sigma',s)$ with the set
$\set{(x',s|_{\pi(x')}):\, x'\in \sigma'}$.

{Given a morphism $(\pi,\alpha):(\Sigma',O')\to (\Sigma,O)$ of $\catScen$ we define a morphism of $\catbScen$:}
\begin{equation}\label{diag:pi-r-alpha}
\begin{tikzcd}[column sep=large,row sep=large]
\hat N(\eE_O\Sigma) \arrow[d,"\hat N f_{(\Sigma,O)}"'] & (\eE_O\circ \pi) \Sigma' \arrow[d,"f_{\eE_O\circ \pi}"] \arrow[l,"l"'] \arrow[r,"r(\alpha)"] & \eE_{O'}\Sigma' \arrow[d,"f'"] \\
\hat N \Sigma & \Sigma' \arrow[l,"\pi"] \arrow[r,equal] & \Sigma'
\end{tikzcd}
\end{equation}
where 
\begin{itemize}
\item $l(x',s)= (\pi(x'),s)$, and
\item $r(\alpha)(\sigma',s)=(\sigma',\alpha_{\sigma'}(s))$.
\end{itemize}
{The left square is a pull-back diagram  (Lemma \ref{lem:pullback}) hence it is a valid type I morphism and 
together  with Propositions \ref{pro:canonical-bundle}, \ref{pro:pull-back-bundle}, and \ref{pro:N-bundle-sce} this implies that  $f_{\eE_O\circ \pi}$ is a bundle scenario.
Observe also that naturality of $\alpha$ implies that the map $r(\alpha)$ is a morphism of $\catsComp$.}


\begin{pro}\label{pro:fully-faithful-bundle}
The assignments $(\Sigma,O) \mapsto f_{(\Sigma,O)}$  and $(\pi,\alpha)\mapsto (\pi,r(\alpha))$, {where $(\pi,r(\alpha))$ is given as in Diagram (\ref{diag:pi-r-alpha}),}  specify a fully faithful functor  
$
\eE:\catScen \to \catbScen.
$ 
\end{pro}  
\Proof{ 
The simplicial complex map $f_{(\Sigma,O)}$ is a bundle scenario by Proposition \ref{pro:canonical-bundle}. By definition it is clear that $\eE(
\Id_{(\Sigma,O)}) = \Id_{\eE(\Sigma,O)}$.
Given the morphisms
$(\pi,\alpha): (\Sigma,O) \to (\Sigma',O')$ and 
$(\pi',\alpha'): (\Sigma',O') \to (\Sigma'',O'')$ of $\catScen$ we have
$$
\begin{aligned}
\eE((\pi',\alpha')\circ(\pi,\alpha))&=
\eE(\pi \diamond \pi', \alpha' \circ (\alpha \ast \Id_{\overline{\pi'}})) \\
&=
(\pi \diamond \pi', r(\alpha' \circ (\alpha \ast \Id_{\overline{\pi'}})))
\\ &=(\pi \diamond \pi',r(\alpha') \circ (\pi')^\ast
(\hat N r(\alpha))) \\
&= (\pi',r(\alpha'))\circ 
(\pi,r(\alpha))\\
&= \eE(\pi',\alpha')\circ \eE(\pi,\alpha).
\end{aligned}
$$
{where in the third line we used Lemma \ref{lem:ralpr}.}
This proves that $\eE$ is a functor. Finally, Lemma \ref{lem:assemble-alpha} and \ref{lem:assemble-typeII} implies that this functor is fully faithful.
}

\subsection{Empirical models on bundle scenarios} 

{
A simplicial complex map that is discrete over vertices satisfies a very useful property which we will refer to through out the paper:
} 
 
\begin{rem}\label{Improppp} 
Let $f: \Gamma \to \Sigma$ be a surjective simplicial complex map that is discrete over vertices.
For simplices $\sigma'\subset \sigma \in \Sigma$ and $\gamma\in \Gamma$ with $f(\gamma)=\sigma$, there exists a unique simplex $\gamma'\in \Gamma$ such that $f(\gamma')=\sigma'$ since $f$ is discrete over vertices.
\end{rem}

We will write
$$
r_{\sigma,\sigma'} : f^{-1}(\sigma) \to f^{-1}(\sigma')
$$
for the map sending $\gamma$ to the corresponding unique simplex $\gamma'$ {specified by Remark \ref{Improppp}.}
For a distribution $p\in D_R(f^{-1}(\sigma))$ we define the restriction $p|_{\sigma'}=D_R(r_{\sigma,\sigma'})(p)$.
With this definition we have
\begin{equation}\label{eq:Improppp}
p|_{\sigma'}(\gamma')=
\sum_{\gamma' \subset \gamma \in f^{-1}(\sigma)}p(\gamma).
\end{equation}

 
\begin{defn}\label{def:emprical-bundle}
{\rm
An {\it empirical model on a bundle scenario} $f: \Gamma \to \Sigma$ is a family $p=(p_{\sigma})_{\sigma \in \Sigma}$ of distributions, where $p_{\sigma}\in D_R(f^{-1}(\sigma))$, such that $p_{\sigma}|_{\sigma'}=p_{\sigma'}$ for every $\sigma' \subset \sigma$ in $ \Sigma$.
We will denote the set of empirical models on $f$ by $\bEmp(f)$.
}
\end{defn}

\subsubsection{Push-forward empirical models for bundles}

\Def{\label{def:hatNp}
Let $f:\Gamma \to \Sigma$ be a bundle scenario and $p \in \bEmp(f)$ be an empirical model.
We define an empirical model $\hat Np$ on the bundle scenario $\hat N f:\hat N\Gamma \to \hat N \Sigma$ by the formula
\begin{equation}\label{eq:hatNp}
(\hat N p)_{\{\sigma_1,\cdots,\sigma
_n\}}(\{\gamma_1,\cdots,\gamma_n\})=p_{\cup_{i}\sigma_i}(\cup_{i}\gamma_i)
\end{equation}
where
$\{\sigma_1,\cdots,\sigma
_n\}\in \hat N \Sigma$, and $\{\gamma_1,\cdots,\gamma
_n\}\in (\hat N f)^{-1}(\{\sigma_1,\cdots,\sigma
_n\}) $.
}

Equation (\ref{eq:hatNp}) gives a well-defined empirical model (Lemma \ref{lem:hat-N-p}).

\begin{defn}\label{def:push-along-T1}
{\rm
For $p \in \bEmp(f)$ we define the push-forward of $p$ along a type I morphism
$$
\begin{tikzcd}[column sep=huge,row sep=large]
\hat N  \Gamma \arrow[d,"\hat N f"] & \pi^* (\hat N \Gamma) \arrow[d," f^\pi"] \arrow[l,"l"'] \\
\hat N \Sigma & \Sigma' \arrow[l,"\pi"]
\end{tikzcd} 
$$
by the formula
$$
(\pi_{\ast}p)_{\sigma'}(\gamma')= (\hat N p)_{\pi(\sigma')}(l(\gamma'))
$$
where $\sigma' \in \Sigma'$ and $\gamma' \in (f^{\pi})^{-1}(\sigma')$. 
When we regard simplicial  relations as functors between the corresponding categories, we {can} write $(\pi_{\ast}p)_{\sigma'}(\gamma')= p_{\overline{\pi}(\sigma')}(\overline{l}(\gamma'))$.
%
}
\end{defn}

The push-forward $\pi_*p$ is an empirical model on the bundle scenario $f^{\pi}$ {as we verify in Lemma \ref{lem:push-typeI}.}

\begin{defn}\label{def:push-along-T2}
{\rm 
For $p \in \bEmp(f)$ and 
$\sigma \in \Sigma$ we define the push-forward of $p$ along a type II morphism
\begin{equation}\label{AAlpDiag}
\begin{tikzcd}
\Gamma  
 \arrow[rd,"f"'] \arrow[rr,"\alpha"]&& \Gamma'
 \arrow[dl,"g"] \\
& \Sigma  & 
\end{tikzcd}
\end{equation}
by the formula
$$
(\alpha_{\ast}p)_{\sigma}=D_R(\alpha|_{f^{-1}(\sigma)})(p_\sigma).
$$
}
\end{defn}

{In Lemma \ref{lem:push-typeII} we verify that}
$\alpha_*p$ is an empirical model on the bundle scenario $g$.
Now we are ready to define the push-forward along {an arbitrary} morphism.

\begin{defn}\label{def:push-forward-bundle-mor}
For $p \in \bEmp(f)$ we define the {\it push-forward} of  $p$ along a morphism  $(\pi,\alpha):f\to f'$ to be $(\pi,\alpha)_{\ast}p=\alpha_{\ast}(\pi_{\ast}p)$.
\end{defn}

Next, we show that with this definition assigning the set of empirical models to a bundle scenario gives a functor.

\Pro{\label{cor:push-forward-empirical}
The assignments $ f \mapsto \bEmp(f)$ and $(\pi,\alpha) \mapsto (\pi,\alpha)_*$ define a functor $\bEmp:\catbScen \to \catSet$.
}

\Proof{
Given $f:\Gamma \to \Sigma$ and $p \in \bEmp(f)$, it is clear that  
$(\delta_{\Sigma},\Id_{\Gamma})_\ast p=p$. 
Using Lemma \ref{lem:push-typeI}, \ref{lem:push-typeII}, and \ref{ReplaabEmp}
for the morphisms
$(\pi,\alpha): f \to f'$ and 
$(\pi',\alpha'): f' \to f''$ of $\catbScen$  we obtain
$$
\begin{aligned}
(\pi',\alpha')_\ast\left((\pi,\alpha)_{\ast}p\right)
&=\alpha'_{\ast}\left(\pi'_\ast
(\alpha_\ast(\pi_\ast p))\right)\\
&=\alpha'_{\ast}(\left((\pi')^\ast(\hat N \alpha)\right)_\ast
\left(\pi'_\ast(\pi_\ast p)\right))
\\
&=
(\alpha'_{\ast}\circ (\pi')^\ast(\hat N \alpha))_\ast
\left((\pi \kleisli \pi')_\ast p\right)\\
&=(\pi \kleisli \pi',\alpha'_{\ast}\circ (\pi')^\ast(\hat N \alpha))_\ast p\\
&=((\pi',\alpha')\circ 
(\pi,\alpha))_\ast p.
\end{aligned}
$$
}
\Pro{\label{pro:natural-iso-bundle}
For an object $(\Sigma,O)$ of $\catScen$ and an empirical model $e \in \Emp(\Sigma,O)$, defining $$\eta_{(\Sigma,O)}(e)_{\sigma}(\sigma,s)=e_{\sigma}(s),$$
where $\sigma \in \Sigma$ and $s \in \mathcal{E}_O(\sigma)$, gives a natural isomorphism
$\eta:\Emp \to \bEmp \circ \eE$.
}
\begin{proof}
Observe that  $(f_{(\Sigma,O)})^{-1}(\sigma)=\{\sigma\}\times \mathcal{E}_O(\sigma)$. That is, one can consider $e$ as a family $\{e_{\sigma} \in D_R(f_{(\Sigma,O)}^{-1}(\sigma))\}_{\sigma \in \Sigma}$ of distributions, and the compatibility in both cases is the same. 
Now, we prove the naturality. Given $(\pi,\alpha):(\Sigma,O)\to (
\Sigma',O')$ a morphism of $\catScen$, we need to prove that the following diagram commutes
\begin{equation}\label{piEe}
\begin{tikzcd}[column sep=huge,row sep=large]
\Emp(\Sigma,O) \arrow[r,"(\pi{,}\alpha)_\ast"] \arrow[d,"\eta_{(\Sigma,O)}"'] & \Emp(\Sigma',O') 
\arrow[d,"\eta_{(\Sigma',O')}"] \\
 \bEmp(f_{(\Sigma,O)}) \arrow[r]  
 \arrow[r,"\eE(\pi{,}\alpha)_\ast"]& \bEmp(f_{(\Sigma',O')})
\end{tikzcd}
\end{equation}
Given $e \in \Emp(\Sigma,O)$ and $(\sigma',s') \in \mathcal{E}_{O'}\Sigma'$, we have 
$$
\begin{aligned}
\left(\eE(\pi,\alpha)_{\ast}(\eta_{(\Sigma,O)}(e))\right)_{\sigma'}(\sigma',s')
&=\left((\pi,r(\alpha))_{\ast}(\eta_{(\Sigma,O)}(e))\right)_{\sigma'}(\sigma',s')\\
&=
r(\alpha)_{\ast}(\pi_\ast(\eta_{(\Sigma,O)}(e)))_{\sigma'}(\sigma',s')\\
&=D_R(r(\alpha)|_{(f_{\eE_O\circ \pi})
^{-1}(\sigma')})(\pi_\ast(\eta_{(\Sigma,O)}(e))_{\sigma'})(\sigma',s')\\
&= \sum_{r(\alpha)(\sigma',s)=(\sigma',s')} \left(\pi_{\ast}(\eta_{(\Sigma,O)}(e))\right)_{\sigma'}(\sigma',s)
\\
&= \sum_{\alpha_{\sigma'}(s)=s'} (\hat N \eta_{(\Sigma,O)}(e))_{\pi(\sigma')}(\pi(\sigma'),s)
\\
&= \sum_{\alpha_{\sigma'}(s)=s'} (\eta_{(\Sigma,O)}(e))_{\overline{\pi}(\sigma')}(\overline{\pi}(\sigma'),s)\\
&=
\sum_{\alpha_{\sigma'}(s)=s'}e_{\overline{\pi}(\sigma')}(s).
\end{aligned}
$$
%
On the other hand,
$$
\begin{aligned}
\left(\eta_{(\Sigma',O')}((\pi,\alpha)_\ast e)\right)_{\sigma'}(\sigma',s') &=((\pi,\alpha)_{\ast}e)_{\sigma'}(s')\\
&=D_R(\alpha_{\sigma'})(e_{\overline{\pi}(\sigma')})(s')\\
&=
\sum_{\alpha_{\sigma'}(s)=s'}e_{\overline{\pi}(\sigma')}(s).
\end{aligned}
$$
\end{proof}

Combining Proposition \ref{pro:fully-faithful-bundle} and \ref{pro:natural-iso-bundle} we obtain our main result in this section.

\Thm{\label{thm:main-bundle}
Sending a scenario $(\Sigma,O)$ to the bundle scenario $f_{(\Sigma,O)}:\eE_O\Sigma \to \Sigma$ specifies a fully faithful functor $\eE:\catScen \to \catbScen$. Moreover, there is a natural isomorphism $\eta:\Emp \to \bEmp\circ \eE$.  
}



\section{The category of simplicial bundle scenarios} \label{sec:category-simplicial-scenarios}

{The theory of} simplicial distributions is  introduced in \cite{okay2022simplicial} as a generalization of empirical models to scenarios consisting of a space of measurements and outcomes.  
In this section we extend simplicial distributions to simplicial bundle scenarios. 
First we introduce simplicial bundle scenarios and morphisms between them to obtain the category $\catsScen$ of simplicial bundle scenarios. Simplicial distributions are defined in this generality as a functor $\sDist:\catsScen \to \catSet$. Therefore in addition to extending the notion to bundles we also describe how to push-forward a simplicial distribution along a morphism.
Theorem \ref{thm:main-simplicial} is our main result that prove the nerve functor $N$ in Diagram (\ref{dia:-eta-zeta}) is fully faithful and $\zeta$ is a natural isomorphism.
In Section \ref{sec:simplicial-sets} we present the technical results concerning simplicial sets.

\begin{defn}\label{def:simplicial-DV-LS-BS}
{\rm
A map $f:E\to X$ of simplicial sets is called 
\begin{itemize}
\item surjective if it is surjective in each degree $n\geq 0$,
\item locally surjective if it has the right-lifting property with respect to $d^i:\Delta[n-1]\to \Delta[n]$ for $n\geq 1$, and $0\leq i \leq n$, i.e., {the diagonal map exists in the following commutative diagram}
$$
\begin{tikzcd}[column sep=huge,row sep=large]
\Delta[n-1] \arrow[d,"d^i"] \arrow[r] & E\arrow[d,"f"] \\
\Delta[n] \arrow[ru,dashed] \arrow[r] & X
\end{tikzcd}
$$

\item discrete over vertices if it has the right-lifting property with respect to $s^i:\Delta[n]\to \Delta[n-1]$ for $n\geq 1$, and $0\leq i \leq n-1$, {i.e., the diagonal map exists} 
$$
\begin{tikzcd}[column sep=huge,row sep=large]
\Delta[n] \arrow[d,"s^j"] \arrow[r] & E\arrow[d,"f"] \\
\Delta[n-1] \arrow[ru,dashed] \arrow[r] & X
\end{tikzcd}
$$
\end{itemize}
 A simplicial (bundle) scenario is a map $f:E\to X$ of simplicial sets that is surjective, locally surjective and discrete over vertices. 
}
\end{defn}


{ 
Let $N:s\catComp\to \catsSet$  denote the nerve functor that sends a simplicial complex $\Sigma$ to the simplicial set $N\Sigma$ whose $n$-simplices are given by
$$
(N\Sigma)_n = \set{(\sigma_1,\cdots,\sigma_n)\in \Sigma^n:\, \cup_{i=1}^n \sigma_i \in \Sigma )}.
$$
For the simplicial structure maps see Definition \ref{def:N}. 
The nerve construction can be used to obtain a simplicial scenario from a bundle scenario of simplicial complexes.
}

%
%



\Pro{\label{pro:nerve-pres-bundle-scen}
If $f$ is a bundle scenario then $Nf$ is a simplicial scenario.}
\Proof{
This follows from Lemma \ref{lem:nerve-S-LS} and \ref{lem:nerve-pre-dv}.}

\subsection{Morphisms of simplicial scenarios}

Defining morphisms between simplicial scenarios is easier than defining morphisms between bundle scenarios. In Section \ref{sec:morphisms-bundle-scen}, we introduced morphisms between bundle scenarios in two steps by defining type I and II morphisms. For simplicial scenarios, we can define morphisms in one step. Therefore this approach, while generalizing bundles of simplicial complexes, has a cleaner {description of morphisms.}

\begin{defn}\label{def:sScen}
{\rm
The category
 $\catsScen$ of simplicial (bundle) scenarios consists of  objects given by simplicial (bundle) scenarios and morphisms $f\to f'$  given by pairs $(\pi,\alpha)$ of simplicial set maps making the following diagram commute
\begin{equation}\label{dia:morphisms}
\begin{tikzcd}[column sep=huge,row sep=large]
E \arrow[d,"f"] &  \pi^*(E)  \arrow[l] \arrow[d,"f^\pi"] \arrow[r,"\alpha"] & E' \arrow[d,"f'"] \\
X & \arrow[l,"\pi"] X' \arrow[r,equal] & X' 
\end{tikzcd}
\end{equation}
where $\pi^*(E)$ is the pull-back along $\pi$.
}
\end{defn}


\begin{rem}
The map $f^\pi$ in Diagram (\ref{dia:morphisms}) is also a simplicial scenario since pulling back a map preserves {surjectivity, local surjectivity, and being discrete over vertices.
The latter two properties are described in terms of the liftings of certain diagrams, and liftings still exist under pull-backs.
}
\end{rem} 

The composition of $(\pi_1,\alpha_1):f\to f'$ and $(\pi_2,\alpha_2):f'\to f''$ is defined by 
\begin{equation}\label{eq:composition-sScen}
(\pi_2,\alpha_2)\circ (\pi_1,\alpha_1) = (\pi_1\circ \pi_2, \alpha_2\circ \pi_2^*( \alpha_1))
\end{equation}
where 
$$\pi_2^*(\alpha_1) : \pi_2^*(\pi_1^*(E)) \to \pi_2^*(E')$$
 is the pull-back of $\alpha_1$ along $\pi_2$.
This gives a well-defined category: The identity morphism of an object $f:E \to X$ is given by $(\Id_X,\Id_E)$ and the composition rule is associative as we prove next.
 
\Lem{\label{lem:sScen-is-well-defined}
The composition given by Equation (\ref{eq:composition-sScen}) is associative. 
} 
\Proof{For $(\pi_1,\alpha_1):f \to f'$, $(\pi_2,\alpha_2):f' \to f''$, and $(\pi_3,\alpha_3):f'' \to f'''$, we have
$$
\begin{aligned}
(\pi_3,\alpha_3)\circ \left((\pi_2,\alpha_2)\circ (\pi_1,\alpha_1)\right)&=
(\pi_3,\alpha_3)\circ (\pi_1 \circ \pi_2,\alpha_2 \circ \pi_2^{\ast}( \alpha_1))\\
&=\left((\pi_1 \circ \pi_2)\circ \pi_3 ,\alpha_3 \circ {\pi_3}^{\ast}\left(\alpha_2 \circ \pi_2^{\ast}( \alpha_1)\right)\right)
\\
&=\left(\pi_1 \circ (\pi_2\circ \pi_3) ,\alpha_3 \circ {\pi_3}^{\ast}(\alpha_2) \circ {\pi_3}^{\ast}( \pi_2^{\ast}( \alpha_1))\right)\\
&=\left( \pi_1 \circ (\pi_2\circ \pi_3),\alpha_3 \circ {\pi_3}^{\ast}(\alpha_2) \circ (\pi_2 \circ \pi_3)^{\ast}(\alpha_1)\right)\\
&=(\pi_2 \circ \pi_3,\alpha_3 \circ \pi_3^{\ast}(\alpha_2))\circ (\pi_1,\alpha_1))\\
&=
\left((\pi_3,\alpha_3)\circ (\pi_2,\alpha_2)\right)\circ (\pi_1,\alpha_1).
\end{aligned}
$$
}



%

%

Next we turn to the nerve functor $N:\catsComp \to \catsSet$.

\begin{defn}\label{def:T-pi}
Let $\pi:\Sigma'\to \hat N\Sigma$ be a simplicial complex map and
$\overline{\pi}$ denote the functor $\catC_{\Sigma'}\to \catC_{\Sigma}$ associated to $\pi$.
We define 
$$
T(\pi): N\Sigma'\to N\Sigma
$$
by
$$
T(\pi)_n(\sigma_1,\cdots,\sigma_n) =
(\overline{\pi}(\sigma_1),\cdots,\overline{\pi}(\sigma_n)) 
$$
where $(\sigma_1,\cdots,\sigma_n)\in (N\Sigma')_n$.
\end{defn}



%

Given a morphism $(\pi,\alpha):f\to f'$ between two bundle scenarios $f:\Gamma\to \Sigma$ and $f':\Gamma'\to \Sigma'$ we define a morphism of $\catsScen$ by the diagram:
\begin{equation}\label{eq:morphism-sScen}
\begin{tikzcd}[column sep=huge,row sep=large]
N\Gamma \arrow[d,"N f"] & N(\pi^*(\hat N \Gamma)) \arrow[l] \arrow[d,"Nf^\pi"] \arrow[r,"N\alpha"] & N\Gamma' \arrow[d,"Nf'"] \\
N\Sigma & N\Sigma' \arrow[l,"T(\pi)"] \arrow[r,equal] &  N\Sigma' 
\end{tikzcd}
\end{equation} 
This is a valid morphism of $\catsScen$ since the left square is a pull-back diagram (Lemma \ref{lem:tildemupull}), i.e., $N(\pi^*(\hat N \Gamma))\cong T(\pi)^*(N{\Gamma})$.


\begin{pro}\label{pro:fully-faithful-simplicial}
The assignments $f \mapsto Nf$  and $(\pi,\alpha) \mapsto (T(\pi),N\alpha)$ specify a fully faithful functor  
$
N :\catbScen \to \catsScen.
$ 
\end{pro}  
\Proof{For a bundle scenario $f: \Gamma \to \Sigma$ the simplicial map $Nf$ is a simplicial scenario by Proposition \ref{pro:nerve-pres-bundle-scen}. In addition, by Lemma \ref{lem:TFun} we have that  
$N(\delta_{\Sigma},\Id_{\Gamma})=
(T(\delta_{\Sigma}),
N(\Id_{\Gamma}))=
(\Id_{N\Sigma},\Id_{N\Gamma})$.
Given two morphisms
$(\pi,\alpha): f \to f'$ and 
$(\pi',\alpha'): f' \to f''$ of $\catsScen$, by Lemmas \ref{lem:TFun} and \ref{lem: ReplaabNN}, we have 
$$
\begin{aligned}
N((\pi',\alpha')\circ(\pi,\alpha))&=
N(\pi \diamond \pi', \alpha' \circ (\pi')^{\ast}(\hat N\alpha)) \\
&=
(T(\pi \diamond \pi'), N\left(\alpha' \circ (\pi')^{\ast}(\hat N\alpha)\right))
\\ &=(T(\pi) \circ T(\pi'), N(\alpha') \circ N\left((\pi')^{\ast}(\hat N\alpha)\right))\\
&=
\left(T(\pi) \circ T(\pi'), N(\alpha') \circ T(\pi')^\ast(N\alpha)\right)\\
&=(T(\pi'),N\alpha')\circ (T(\pi),N\alpha)\\
&=N(\pi',\alpha')\circ 
N(\pi,\alpha). 
\end{aligned}
$$
This proves that $N$ is a functor. By Lemma \ref{lem:pi-Npi} and \ref{lem:Nalph-alpha} this functor is fully faithful.
}

\subsection{Simplicial distributions on bundles}\label{sec:simp-dist-revisited}

 
\begin{defn}\label{def:simp-dist-revisited}
{\rm
Let $f:E\to X$ be a simplicial set map.
A {\it simplicial distribution} on $f$
is a simplicial set map $p:X\to D_R(E)$ that makes the following diagram commute:
\begin{equation}\label{dia:simp-dist-revisited}
\begin{tikzcd}[column sep=huge,row sep=large]
& D_R(E) \arrow[d,"{D_R(f)}"] \\
X \arrow[ur,"p"] \arrow[r,"\delta"'] & D_R(X)
\end{tikzcd}
\end{equation}
We write $\sDist(f)$ 
for the set of simplicial 
distributions on $f$.
}
\end{defn}
 
 \Pro{\label{pro:sDist-characterization}
A map $p:X\to D_R(E)$ belongs to $\sDist(f)$ if and only if
the support of $p_n(x)$  is contained in $f_n^{-1}(x)$ for all $x\in X_n$.
}
\Proof{
The map $p$ belongs to 
$\sDist(f)$ if and only if for every $x\in X_n$ we have
$$
\sum_{e \in E_n}p_n(x)(e)\delta^{f_n(e)}=\delta^x.
$$
Since $R$ is  a zero-sum-free semiring, this condition holds if and only if $p_n(x)(e)=0$ for every  $x \in X_n$ and $e\in E_n$ satisfying $f_n(e)\neq x$.
}
 
Given  simplicial sets $X$ and $Y$ we define
$$
m: D_R(X)\times D_R(Y) \to D_R(X\times Y)
$$
by sending $(p,q)\in D_R(X_n)\times D_R(Y_n)$ to the distribution $p\cdot q$ defined by
\begin{equation}\label{eq:p-dot-q}
p\cdot q(\sigma,\theta) = p(\sigma) q(\theta)
\end{equation}
for $\sigma\in X_n$ and $\theta\in Y_n$.
%
It is straightforward to verify that Equation (\ref{eq:p-dot-q}) respects the face and the degeneracy maps. {For instance, for the face maps we have}
$$
\begin{aligned}
D_R(d_i^X\times d_i^Y)(m_{{n+1}}(p,q))(\sigma,\theta)&=
\sum_{d^X_i(\sigma')=\sigma, \,
d^Y_i(\theta')=\theta} m_{{n+1}}(p,q)(\sigma',\theta')\\
&=
\sum_{d^X_i(\sigma')=\sigma , \,
d^Y_i(\theta')=\theta}p(\sigma')q(\theta') \\
&=\sum_{d^X_i(\sigma')=\sigma}p(\sigma') \sum_{d^Y_i(\theta')=\theta}q(\theta') \\
&=D_R(d_i^X)(p)(\sigma) \cdot D_R(d_i^Y)(q)(\theta) \\
&=m_{{n}}(D_R(d_i^X)(p),D_R(d_i^Y)(q))(\sigma,\theta).
\end{aligned}
$$ 
The canonical map that goes in the other direction
$$
D_R(\pr_1)\times D_R(\pr_2):D_R(X\times Y) \to D_R(X)\times D_R(Y)
$$
splits $m$, {i.e.,} $(D_R(\pr_1)\times D_R(\pr_2) )\circ m=\idy$.

\Rem{\label{rem:f-XY} 
Given a pair $(X,Y)$ of simplicial sets consider the projection map 
$$f_{(X,Y)}:X\times Y\to X$$
Simplicial distributions on $f_{(X,Y)}$ coincide with simplicial distributions on $(X,Y)$ in the sense of \cite{okay2022simplicial}, that is, with simplicial set maps $X\to D_R(Y)$. 
To see this correspondence,
let $p$ be a simplicial distribution in the sense of {Definition \ref{def:simp-dist-revisited}.}
For every $\sigma\in X_n$ we have that
\begin{equation}\label{eq:equiv-simp-dist}
{
\sum_{\theta\in Y_n} p_n(\sigma)(-,\theta) = \delta^\sigma.
}
\end{equation}
This implies that $\supp({p_n(\sigma)})$ consists of pairs $(\sigma,\theta)$ where $\theta\in \supp({p_n(\sigma)(\sigma,-)})$. Therefore we can identify  ${p_n(\sigma)}$ with the distribution ${p_n(\sigma)}(\sigma,-)\in D_R(Y_n)$.  Given $p$, the desired simplicial set map is the composite 
$$X\xrightarrow{p} D_R(X\times Y) \xrightarrow{D_R(\pr_2)}D_R(Y)$$ 
Conversely, a simplicial set map $q:X\to D_R(Y)$ can be lifted to $D_R(X\times Y)$ by the following composite:
$$
X\xrightarrow{\idy\times q} X\times  D_R(Y) \xrightarrow{\delta \times \idy} D_R(X)\times D_R(Y) \xrightarrow{m} D_R(X\times Y),
$$ 
{which belongs to $\sDist(f_{(X,Y)})$ by Proposition \ref{pro:sDist-characterization}.}
Note that $f_{(X,Y)}$ fails to be a simplicial scenario (Definition \ref{def:simplicial-DV-LS-BS}): In general, it is surjective and locally surjective, but not discrete over vertices.
}

\subsubsection{Push-forward simplicial distributions}

{
In this section we   define the push-forward of a simplicial distribution along a morphism {of $\catsScen$}. As a first step we define push-forward along $\pi:X'\to X$ by introducing  a distribution on  the middle scenario $f^\pi$ in Diagram (\ref{dia:morphisms}). 
}

\begin{defn}\label{def:push-forward-pi}
{\rm 
Given a simplicial distribution $p$ on $f:E\to X$ and a simplicial set map $\pi: X' \to X$, we define a simplicial distribution  on $f^{\pi}$ by {the composition
$$
\pi_*(p): X' \xrightarrow{p\circ \pi \times \delta_{X'}} D_R(E)\times D_R(X') \xrightarrow{m} D_R(E\times X')
$$ 
whose image lands in the simplicial subset $D_R(\pi^*E)$.
}
}
\end{defn}

Let us unravel this construction. 
Consider a simplicial distribution $p:X\to D_R(E)$ on $f$. 
The first observation which makes this construction work is that 
$$D_R(f)\circ p \circ \pi =\delta_X \circ \pi=D_R(\pi)\circ \delta_{X'}.$$ 
{Therefore $p\circ\pi\times \delta_{X'}$ lands inside the pull-back $(D_R\pi)^*(D_RE)$.}
The second observation is that for $x' \in X'_n$ and $(e,x)\in E_n \times X_n$ such that $f_n(e) \neq \pi_n(x)$,  Proposition \ref{pro:sDist-characterization} implies that  
$p_n(\pi_n(x'))(e)\cdot \delta^{x'}(x)=0\cdot \delta^{x'}(x)=0$. Therefore
$m \circ (p\circ \pi \times \delta_{X'})$ lands inside the space of distributions $D_R(\pi^*E)$ on the pull-back. 
In general, for $x' \in X'_n$, we have the following formula 
\begin{equation}\label{PiFormulaaa}
(\pi_\ast p)_n(x')=p_n(\pi_n(x'))\cdot \delta^{x'},
\end{equation}
which shows that $\pi_\ast p(x') \in (f^{\pi})^{-1}(x')$. Therefore  $\pi_{\ast}(p) \in \sDist(f^\pi)$ by Proposition \ref{pro:sDist-characterization}.

\begin{defn}\label{def:push-forward}
Given a simplicial distribution $p$ on $f:E\to X$ and a morphism $(\pi,\alpha):f\to f'$ of simplicial scenarios we define the {\it push-forward} distribution,  a simplicial distribution on $f'$, by the following composite:
\begin{equation}\label{dia:push-forward}
{
(\pi,\alpha)_* p : X' \xrightarrow{\pi_{\ast} (p)} D_R(\pi^*E) \xrightarrow{D_R(\alpha)} D_R(E')
}
\end{equation}  
\end{defn}

{To prove that the push-forward construction is functorial we need two preliminary results.}
\begin{lem}\label{lem:PiComppp}
Let $p$ be a simplicial distribution on $f:E\to X$.
For simplicial set maps $\pi:X' \to X$ and $\pi':X'' \to X'$, we have $(\pi \circ \pi')_\ast p={\pi'}_\ast(\pi_\ast p)$. 
\end{lem}

\Proof{
Given $x\in X_n''$ Equation (\ref{PiFormulaaa}) implies that
$$
(\pi'_\ast(\pi_\ast p))_n(x)=(\pi_\ast p)_n(\pi'_n(x))\cdot \delta^x=p_n\left(
\pi_n(\pi'_n(x))\right)\cdot \delta^{\pi'_n(x)}\cdot \delta^x.
$$
Therefore if we apply this map to $(e,\pi'_n(x),x) \in ((\pi')^\ast\pi^\ast(E))_n $, we obtain $p_n\left(
\pi_n(\pi'_n(x))\right)(e)$.
On the other hand, 
$$
(\pi \circ \pi')_\ast p(x)=p_n(\pi_n\circ\pi'_n(x))\cdot \delta^x.
$$
If we apply this map to the corresponding element $(e,x)\in (\pi \circ \pi')^\ast (E)_n$, we obtain the same result.
}

\begin{lem}\label{ReplaasDist}
Given a simplicial set map $\pi: X' \to X$,
a commutative diagram of simplicial set maps
$$
\begin{tikzcd}
E  
 \arrow[rd,"f"'] \arrow[rr,"\alpha"]&& E'
 \arrow[dl,"g"] \\
& X  & 
\end{tikzcd}
$$
and a distribution $p \in \sDist(f)$, we have 
$$
\pi_{\ast}(D_R(\alpha)\circ p)=D_R(\pi^{\ast}(\alpha))\circ \pi_{\ast}p.
$$
\end{lem}
\Proof{
Given $x' \in X'_n$ and $(e',x') \in \pi^{\ast}(E')_n$, using Equation (\ref{PiFormulaaa}) we have 
$$
\begin{aligned}
\pi_{\ast}(D_R(\alpha)\circ p)_n(x')(e',x')&=
(D_R(\alpha)\circ p)_n
\left(\pi_n(x')\right)(e')
\\
&=D_R(\alpha_n)\left(p_n(
\pi_n(x'))\right)(e')\\
&=
\sum_{e\,:\,\, \alpha_n(e)=e'}p_n\left(\pi_n(x')\right)(e).
\end{aligned}
$$
On the other hand,
$$
\begin{aligned}
D_R(\pi_{\ast}(\alpha)_n)\left((\pi_\ast p)_n(x')\right)(e',x')&=
\sum_{\pi^\ast
(\alpha)_n(e,x')=(e',x')}(\pi_\ast p)_n(x')(e,x')
\\
&=
\sum_{e\,:\,\, \alpha_n(e)=e'}p_n\left(\pi_n(x')\right)(e).
\end{aligned}
$$

}
\Pro{\label{pro:sDist-functor}
The push-forward construction in Definition \ref{def:push-forward} gives a functor
$$\sDist: \catsScen \to \catSet.$$
}
\Proof{
For a simplicial scenario $f: E \to X$ and $p \in \sDist(f)$, it is clear that  
$(\Id_{X},\Id_{E})_\ast p=p$. 
Applying  Lemma \ref{lem:PiComppp} and \ref{ReplaasDist}
to the morphisms
$(\pi,\alpha): f \to f'$ and 
$(\pi',\alpha'): f' \to f''$ of $\catsScen$  we obtain
$$
\begin{aligned}
(\pi',\alpha')_\ast\left((\pi,\alpha)_{\ast}p\right)
&=D_R(\alpha') \circ \pi'_\ast((\pi,\alpha)_\ast p)\\
&=D_R(\alpha') \circ \pi'_\ast(D_R(\alpha) \circ \pi_{\ast} p)
\\
&=
D_R(\alpha') \circ D_R((\pi')^{\ast}(\alpha))\circ \pi'(\pi_{\ast}p)\\
&=D_R\left(\alpha' \circ (\pi')^{\ast}(\alpha)\right)\circ \left((\pi \circ \pi')_{\ast}p\right)\\
&=\left(\pi \circ \pi', \alpha'\circ (\pi')^*( \alpha)\right)_\ast p
\\
&=((\pi',\alpha')\circ 
(\pi,\alpha))_\ast p.
\end{aligned}
$$
}
%


\subsection{The natural isomorphism}

{
In this section we construct a natural isomorphism $\zeta:\bEmp \to \sDist \circ N$. 
The first step is to construct a simplicial distribution on the nerve of a bundle scenario from a given empirical model on the bundle.
}

\Def{\label{def:Np}
Let $f:\Gamma \to \Sigma$ be a bundle scenario and $p \in \bEmp(f)$ be an empirical model. We define a simplicial distribution $Np$ on the simplicial  scenario $Nf :  N \Gamma \to N \Sigma$ as follows: 
\begin{equation}\label{eq:Np}
(Np)_n(\sigma_1,\cdots,\sigma_n)(\gamma_1,\cdots,\gamma_n)=\begin{cases}p_{\cup_{i=1}^n\sigma_i}(\cup_{i=1}^n\gamma_i) & (\gamma_1,\cdots,\gamma
_n)\in (N f)_n^{-1}(\sigma_1,\cdots,\sigma
_n) \\
0 & \text{otherwise,}
\end{cases}
\end{equation}
where
$(\sigma_1,\cdots,\sigma
_n)\in (N \Sigma)_n$.
}

The fact that Equation (\ref{eq:Np}) gives a well-defined simplicial distribution is proved in Lemma \ref{lem:N-p}.

%
\Pro{\label{pro:natural-iso-simplicial} 
For a bundle scenario $f:\Gamma\to \Sigma$ and 
$p \in \bEmp(f)$,  defining $\zeta_f(p)=Np$ gives a natural isomorphism 
$\zeta:\bEmp \to \sDist \circ N$.
}
\begin{proof}
Given $p,q \in \bEmp(f)$ such that 
$
\zeta_f(p)=\zeta_f(q)$, we have 
$$
p_\sigma(\gamma)=
(Np)_1(\sigma)(\gamma)=(Nq)_1(\sigma)(\gamma)=q_\sigma(\gamma).$$ 
Therefore $\zeta_f$ is injective. 
To prove surjectivity, let $\tilde{p}\in \sDist(Nf)$. 
For every $\sigma \in \Sigma$ we define $p_{\sigma}=\tilde{p}_1(\sigma)$.
To see that $p \in \bEmp(f)$, let $\sigma' \subset \sigma$ and $\gamma'\in f^{-1}(\sigma')$. Then by {Equation (\ref{eq:Improppp})}  and Lemma \ref{lem:PinNF}
  we have 
$$
\begin{aligned}
p_\sigma|_{\sigma'}(\gamma')&= \sum_{\gamma'\subset \gamma\in f^{-1}(\sigma)}
p_{\sigma}(\gamma)\\
&=\sum_{\gamma'\subset \gamma\in f^{-1}(\sigma)}
\tilde{p}_1({\sigma})(\gamma)\\
&=
\sum_{\tau \in f^{-1}(\sigma-\sigma'):\, \gamma' \cup \tau \in \Gamma}\tilde{p}_1({\sigma' \cup (\sigma - \sigma')})(\gamma' \cup \tau)\\
&= \sum_{\tau \,:\,\,(\gamma',\tau)\in (Nf)^{-1}(\sigma',\sigma-\sigma')} \tilde{p}_2(\sigma',\sigma-\sigma')(\gamma',\tau) \\
&=D_R(d_2)(\tilde{p}_2\left(\sigma',\sigma-\sigma')\right)(\gamma')\\
&=\tilde{p_1}\left(d_2(\sigma',\sigma-\sigma')\right)(\gamma')
\\
&=\tilde{p_1}(\sigma')(\gamma')
=p_{\sigma'}(\gamma').
\end{aligned}
$$
Again, using Lemma \ref{lem:PinNF} we obtain that $\zeta_f(p)=\tilde{p}$. 

Now, we prove the naturality of $\zeta$. Given a morphism $(\pi,\alpha): f\to 
f'$  of $\catbScen$, we prove that the following diagram commutes
\begin{equation}
\begin{tikzcd}[column sep=huge,row sep=large]
\bEmp(f) \arrow[r,"(\pi{,}\alpha)_\ast"] \arrow[d,"\zeta_f"'] & \bEmp(f') 
\arrow[d,"\zeta_{f'}"] \\
 \sDist(Nf) \arrow[r]  
 \arrow[r,"(T(\pi){,}N\alpha)_\ast"]& \sDist(Nf')
\end{tikzcd}
\end{equation}
For this let $p \in \bEmp(f)$. By Lemma \ref{lem:Npiaast} and \ref{lem:Nalphaast} we have 
$$
\begin{aligned}
\zeta_{f'}((\pi,\alpha)_\ast p)&=N(\alpha_\ast
(\pi_{\ast}p))
=
D_R(N\alpha)\circ N(\pi_\ast p)\\
&=D_R(N\alpha) \circ T(\pi)_\ast(Np)\\
&=(T(\pi),N\alpha)
_\ast(Np)\\
&=(T(\pi),N\alpha)
_\ast(\zeta_f(p)).
\end{aligned}
$$
\end{proof}

Combining Proposition \ref{pro:fully-faithful-simplicial} and \ref{pro:natural-iso-simplicial} we obtain our main result of Section \ref{sec:category-simplicial-scenarios}.

\Thm{\label{thm:main-simplicial}
Sending a bundle scenario $f:\Gamma\to \Sigma$ to the simplicial scenario $Nf:N\Gamma\to N\Sigma$ specifies a fully faithful functor $N:\catbScen \to \catsScen$. Moreover, there is a natural isomorphism $\zeta:\bEmp \to \sDist\circ N$. 
}

\section{Contextuality for bundle scenarios}\label{sec:convexity-and-contextuality}

In Sections  \ref{sec:bScen} and \ref{sec:category-simplicial-scenarios} we have introduced {distributions on} scenarios of bundle scenarios based on simplicial complexes and simplicial sets, respectively.
In this section we use the theory of convex categories \cite{kharoof2022simplicial} to introduce contextuality for bundle scenarios in two versions. 
Again using this theory we obtain the upgraded Diagram (\ref{dia:upgraded}) together with the natural isomorphisms $\tilde\eta$ and $\tilde\zeta$ relating empirical models and simplicial distributions. 
Our main result is Theorem \ref{thm:contextuality} relating different versions of contextuality via these natural isomorphisms.

\subsection{Convexity}
{Convexity can be studied in an abstract way for arbitrary semirings using the notion of an algebra over a monad. The relevant monad in this case is the distribution functor $D_R:\catSet\to \catSet$.  In Section \ref{sec:ConvCat} we recall basic properties of convexity in this abstract language.
}
We begin by observing that the empirical model and simplicial distribution functors defined from the three kinds of categories of scenarios into 
the category of sets
actually land in the category $\catConv_R$ of $R$-convex sets. 

{
The set of simplicial distributions on $f_{(X,Y)}: X\times Y\to X$ can be identified with the set  of simplicial set maps from $X\to D_R(Y)$ as discussed in Remark \ref{rem:f-XY}.
In \cite{kharoof2022simplicial} it is shown that $U=\catsSet(X,D_RY)$ is an $R$-convex set with the structure map
$\nu=\nu^U: D_R (U)\to U$  
defined as follows: 
\begin{equation}\label{eq:nu}
	\nu(Q)_n(x)=\sum_{p \in U}Q(p)\,p_n(x). 
\end{equation}
where $Q \in D_R(U)$ and $x\in X_n$. 
Next, we generalize this observation.
}

\Pro{\label{SDisconv1}
Given a simplicial {set map} $f:E \to X$, the set $\sDist(f)$ of simplicial distributions on $f$ is an $R$-convex set. 
}
\begin{proof} 
The set $\sDist(f)$ is contained in $U=\catsSet(X,D_R(E))$.
It suffices to show that for every $Q \in D_R(\sDist(f))$, the 
distribution $\nu(Q)$ lies in $\sDist(f)$.
Given $x \in X_n$ and $e \in E_n$ such that $f_n(e)\neq x$, {using Equation (\ref{eq:nu})} and Proposition \ref{pro:sDist-characterization} we have 
$$
\nu(Q)_n(x)(e)=\sum_{p \in U }Q(p)\,p_n(x)(e)
=\sum_{p \in \sDist(f) }Q(p)\,p_n(x)(e)=0.
$$
Again, by  Proposition \ref{pro:sDist-characterization} this implies 
$\nu(Q) \in \sDist(f)$.
\end{proof}

{
Using this result we can also equip the set of empirical models on bundle scenarios with an $R$-convex set structure using the natural isomorphism $\zeta$. Furthermore, the set of empirical models on scenarios also inherit an $R$-convex set structure via the other natural isomorphism $\eta$. We summarize this observation:
}
 
\begin{cor}\label{EmpConv}
Let $(\Sigma,O)$ be a scenario and $f:\Gamma\to \Sigma$ be a bundle scenario.
\begin{enumerate}
\item The set $\Emp(\Sigma,O)$ of empirical models on  $(\Sigma,O)$ is an $R$-convex set and the isomorphism 
$$\eta_{(\Sigma,O)} : \Emp(\Sigma,O) \to \bEmp({f_{(\Sigma,O)}})$$
 of Proposition \ref{pro:natural-iso-bundle} is an $R$-convex map.

\item The set $\bEmp(f)$ of empirical models on  $f$ is an $R$-convex set and the isomorphism 
$$\zeta_f : \bEmp(f) \to \sDist(Nf)$$ of Proposition \ref{pro:natural-iso-simplicial} is an $R$-convex map.
\end{enumerate}
\end{cor}

\Pro{\label{SDisconv2} 
Given a morphism $(\pi,\alpha):f \to f'$ of $\catsScen$,  the induced map $
(\pi,\alpha)_\ast:\sDist(f) \to \sDist(f')$ between the sets of simplicial distributions
is a morphism of $\catConv_R$.
}
\begin{proof}
Since $(\pi,\alpha)$ factors as $(\Id,\alpha)\circ(\pi,\Id)$ as a consequence of Equation (\ref{eq:composition-sScen}), it suffices to show that $(\Id,\alpha)_\ast$ and $(\pi,\Id)_\ast$ are morphisms of $\catConv_R$. 
The map 
$D_R(\alpha): D_R(\pi^{\ast}(E)) \to D_R(E')$ is a morphism of $\catsConv_R$, thus by  Proposition \ref{pro:2-15}
 the map $D_R(\alpha)_\ast: \catsSet(X,D_R\left(\pi^{\ast}(E)\right)) \to \catsSet(X,D_R(E'))$ is   a morphism of $\catConv_R$. The map $(\Id,\alpha)_\ast$ is obtained by the restriction of $D_R(\alpha)_\ast$ to $\sDist(f^\pi)$, hence it is a morphism of $\catConv_R$. In order to prove that $(\pi,\Id)_\ast=\pi_\ast$ belongs to $\catConv_R$ we show that the following diagram commutes:
$$
\begin{tikzcd}[column sep=huge,row sep=large]
D_R(\sDist(f)) 
\arrow[rr,"D_R(\pi_{\ast})"]
 \arrow[d,"\nu^{\sDist(f)}"] && D_R(\sDist(f^\pi)) 
 \arrow[d,
 "\nu^{\sDist(f^\pi)}"] \\
\sDist(f) 
 \arrow[rr,"\pi_\ast"] && \sDist(f^\pi)
\end{tikzcd}
$$
Given $Q \in D_R(\sDist(f))$, $x'\in X'_n$ and $(e,x')\in \pi^\ast(E)_n$, we have
$$ 
\begin{aligned}
\nu(D_R(\pi_\ast)(Q))_n(x')(e,x')&=
\sum_{q \in \sDist(f^\pi)} D_R(\pi_\ast)(Q)(q)q_n(x')(e,x')\\
&=
\sum_{q\in \sDist(f^\pi)} \;\sum_{p:\,\pi_\ast p=q}Q(p)(\pi_\ast p)_n(x')(e,x')
\\
&=\sum_{p \in \sDist(f)} Q(p)p_n(\pi_n(x'))(e)
\\
&=\nu(Q)_n(\pi_n(x'))(e)\\
&=\pi_\ast(\nu(Q))_n(x')(e,x').
\end{aligned}
$$
\end{proof}

This result implies that  the functor $\sDist$ lands in the category of convex sets:
$$
\sDist: \catsScen \to \catConv_R.
$$
{Similarly, t}he functors $\Emp$ and $\bEmp$ land in $\catConv_R$. In more details, we have the following:
\begin{enumerate}
\item Given a morphism $(\pi,\alpha):(\Sigma,O) \to (\Sigma',O')$ of $\catScen$,  the induced map $
(\pi,\alpha)_\ast:\Emp(\Sigma,O) \to \Emp(\Sigma',O')$
is a morphism of $\catConv_R$.

\item Given a morphism $(\pi,\alpha):f \to f'$ of $\catbScen$,  the induced map $
(\pi,\alpha)_\ast:\bEmp(f) \to \bEmp(f')$
is a morphism of $\catConv_R$.
\end{enumerate}
{The first statement can be proved as follows:} By Diagram (\ref{piEe}) we have $(\pi,\alpha)_\ast=\eta_{(\Sigma',O')}^{-1}\circ \mathcal{E}(\pi,\alpha)_\ast \circ \eta_{(\Sigma,O)}$. Therefore by Corollary \ref{EmpConv} and Proposition \ref{SDisconv2} we obtain that $(\pi,\alpha)_\ast$ is an $R$-convex map. The second part is similar.
%


%
{
In Proposition \ref{pro:ConvisConv} we show that $\catConv_R$ is an $R$-convex category. This fact allows us to lift the empirical model and simplicial distribution functors to the free $R$-convex categories ({Definition \ref{def:Free CatFun}}) on the associated category of scenarios using}
the transposes of the functors $\Emp$, $\bEmp$, and $\sDist$ with respect to the adjunction $\Cat \dashv \catConvCat_R$ {to obtain the following functors}:
$$
\begin{aligned}
\widetilde{\Emp}: D_R(\catScen) &\to \catConv_R\\
\widetilde{\bEmp}: D_R(\catbScen) &\to \catConv_R\\
\widetilde{\sDist}: D_R(\catsScen) &\to \catConv_R.
\end{aligned}
$$

%
\begin{cor}\label{cor:upgrade-convex} 
The free convex categories of scenarios are related by the fully faithful functors
$$
\begin{aligned}
D_R(\eE):D_R(\catScen) &\to D_R(\catbScen) \\
D_R(N):D_R(\catbScen) &\to D_R(\catsScen).
\end{aligned}
$$ 
Moreover, the natural isomorphisms $\eta$ and $\zeta$ lift to the  natural isomorphisms 
$$
\begin{aligned}
\tilde \eta:\widetilde{\Emp} &\to \widetilde{\bEmp}\circ D_R(\eE)\\
\tilde \zeta:\widetilde{\bEmp} &\to \widetilde{\sDist} \circ D_R(N).
\end{aligned}
$$ 
\end{cor}
\Proof{The first one follows from the fact that $\eE$ and $N$ are fully faithful (Theorem \ref{thm:main-bundle} and Theorem \ref{thm:main-simplicial}), and $D_R$ sends an isomorphism of sets to an isomorphism. 
Note that transposes for the functors $\bEmp\circ \eE$ and $\sDist\circ N$ with respect to the adjunction $\Cat \dashv \catConvCat_R$ are $\widetilde{\bEmp}\circ D_R(\eE)$ and $\widetilde{\sDist} \circ D_R(N)$; respectively. Therefore, we obtain the second part directly from Theorem \ref{thm:main-bundle}, Theorem \ref{thm:main-simplicial}, and Proposition \ref{pro:Nattrans}.
}

\subsection{Contextuality}
In \cite{okay2022simplicial} the notion of contextuality is defined for simplicial distributions on pairs $(X,Y)$ of simplicial sets.
In this section we extend this definition to simplicial distributions on simplicial bundle scenarios.
First we introduce deterministic distributions in the bundle picture and then define a comparison map that  sends a probabilistic mixture of deterministic distributions to a simplicial distribution. 
Analogously we define contextuality for empirical models on scenarios $(\Sigma,O)$ and bundles scenarios $f:\Gamma\to \Sigma$ of simplicial complexes.
The natural isomorphisms $\eta$ and $\zeta$ relating empirical models and simplicial distributions on there categories of scenarios can be used to compare these {different} notions of contextuality.

\Def{\label{def:sSect}
Let $f:E \to X$ be a simplicial set map.
The set of sections of $f$, i.e., simplicial set maps $s:X\to E$ such that $f\circ s =\idy_X$, will be denoted by $\sSect(f)$.
{A \emph{deterministic distribution} is a simplicial distribution of the form
$$
\delta^s:X\xrightarrow{s} E\xrightarrow{\delta_E} D_R(E)
$$
where $s$ is a section of $f$.}
}

Sending a section $s$ to the associated deterministic distribution $\delta^s$ specifies an injective map 
\begin{equation}\label{eq:delta-f}
\delta_f: \sSect(f) \to \sDist(f)
\end{equation} 

\Lem{\label{lem:det-preserved}
Given a morphism $(\pi,\alpha):f \to f'$ of $\catsScen$, the induced map $(\pi,\alpha)_\ast$ sends a deterministic distribution in $\sDist(f)$ to a deterministic distribution in $\sDist(f')$.  
}
\begin{proof}
Given $s \in \sSect(f)$,  the following diagram commutes:
$$
\begin{tikzcd}[column sep=huge,row sep=large]
E \arrow[d,"{f}"'] & X 
\arrow[ld,"\idy_{X}"] \arrow[l,"{s}"'] &
X' \arrow[l,"\pi"'] \arrow[d,"\idy_{X'}"] \\
X && \arrow[ll,"{\pi}"]  X'
\end{tikzcd} 
$$
and
induces a map $(s \circ\pi)\times \Id_{X'} : X' \to \pi^\ast(E)$, which gives a section of $f^{\pi}$. In addition, for $x' \in X'_n$ and $(e,x') \in \pi^\ast(E)_n$, we have 
$$
\begin{aligned}
(\delta^{(s\circ \pi)\times\Id_{X'}})_n(x')(e,x')&=\delta^{(s_n\circ \pi_n(x'),x')}(e,x')\\
&=
\delta^{(s_n(\pi_n(x'))}(e)\\
&=(\delta^s)_n
(\pi_n(x'))(e)\\
&=\pi_\ast(\delta^s)_n(x')(e,x').
\end{aligned}
$$
Therefore $\pi_\ast(\delta^s)=
\delta^{(s \circ \pi) \times \Id_{X'}}$.
Now, for a section $\tilde{s}$ of 
$f^\pi$, we have
$$
D_R(\alpha)\circ \delta^{\tilde{s}}=
D_R(\alpha)\circ \delta_{\pi^\ast(E)} \circ \tilde{s} = \delta_{E'} \circ \alpha \circ \tilde{s}= \delta^{\alpha \circ \tilde{s}}.
$$
This implies that 
\begin{equation}\label{eq:seccc}
(\pi,\alpha)_\ast \delta^{s}=\delta^{\alpha \circ \left((s \circ \pi) \times \Id_{X'} \right)}.
\end{equation}
\end{proof}

\Pro{
{
Sending a scenario to the set of sections of that scenario gives a functor 
$$
\sSect: \catsScen \to \catSet
$$
and 
}
$\delta: \sSect \to \sDist$ defined by $\delta_f$ in (\ref{eq:delta-f}) at each simplicial scenario $f$, is a natural transformation. 
}
\Proof{{
Given a morphism $(\pi,\alpha):f\to f'$ and a section $s$ of $f$ the induced section 
$${(\pi,\alpha)_*(s): X' \xrightarrow{(s\circ \pi)\times {\idy_{X'}} } \pi^*(E) \xrightarrow{\alpha} E'}$$
By Equation (\ref{eq:seccc}) and since $\sDist$ is a functor 
this assignment gives a functor. Moreover, again  the same equation implies that
$\delta: \sSect \to \sDist$ is a natural transformation.
}
}

Let $f:E \to X$ be a simplicial set map.
We define  
\begin{equation}\label{eq:Theta-f}
\Theta_f: D_R(\sSect(f)) \to \sDist(f)
\end{equation} 
to be the transpose of 
the morphism $\delta_f :\sSect(f) \to \sDist(f)$ of
$\catConv_R$ with respect to the adjunction 
$\catSet \dashv \catConv_R$.
%
More explicitly, for $d=\sum_{s} {d(s)} s \in D_R(\sSect(f))$, we have  
$$
\Theta_f(d)_n(x)(e) = \sum_s{d(s)} \delta^{s_n(x)}(e), 
$$
where $x\in X_n$ and $e\in E_n$.
The naturality of $\delta$ implies that 
\begin{equation}\label{eq:Theta}
\Theta: D_R \circ \sSect \to \sDist
\end{equation}
{defined by $\Theta_f$ in (\ref{eq:Theta-f}) at each object of $\catsScen$}
 is a natural transformation.
\begin{defn}\label{def:Context}
A simplicial distribution $p \in \sDist(f)$ is called {\it non-contextual} if it lies in the image of $\Theta_f$. Otherwise, it is called {\it contextual}.
\end{defn}
Note that by Remark \ref{rem:f-XY} this definition specializes to 
 Definition 3.10 in \cite{okay2022simplicial}  when the scenario is of the form $f_{(X,Y)}:X\times Y\to X$.
%

As in the simplicial case, we define the functors $\Sect$ and $\bSect$ on $\catScen$ and $\catbScen$, respectively:
\begin{enumerate}
\item We define 
$$\Sect: \catScen \to \catSet$$
by sending a scenario $(\Sigma,O)$ to the set $\mathcal{E}_O(\Sigma_0)$. There is a natural transformation 
$$
\delta: \Sect \to \Emp
$$
whose component at $(\Sigma,O)$ is given by 
the map $\delta_{(\Sigma,O)}:\Sect(\Sigma,O) \to \Emp(\Sigma,O)$ which sends a section $s\in \mathcal{E}_O(\Sigma_0)$ to the empirical model 
$\{\delta^{s|_{\sigma}}\}_{\sigma \in \Sigma}$.

\item We define 
$$\bSect: \catScen \to \catSet$$
by sending a bundle scenario $f:\Gamma\to \Sigma$ to the set of sections of $f$, i.e., simplicial complex maps $s:\Sigma\to \Gamma$ such that $f\circ s = \idy_\Sigma$.
There is a natural transformation 
$$
\delta: \bSect \to \bEmp
$$
whose component at $f$ is given by 
the map $\delta_{f}:\bSect(f) \to \bEmp(f)$ which sends a section $s\in \bSect(f)$ to the empirical model $\{\delta^{s(\sigma)}\}_{\sigma \in \Sigma}$.
\end{enumerate}
%
%
{
Moreover, we can define natural transformations 
$$
\begin{aligned}
\Theta:D_R\circ \Sect &\to \Emp\\
\Theta:D_R\circ \bSect &\to \bEmp
\end{aligned}
$$
whose components $\Theta_{(\Sigma,O)}$
and
$\Theta_f$ are given by the transpose of $\delta_{(\Sigma,O)}$ and $\delta_f$  with respect to the adjunction $\catSet \dashv \catConv_R$, respectively. 
}

\begin{defn}\label{def:Contextbbbb}
Let $(\Sigma,O)$ be a scenario and $f:\Gamma\to \Sigma$ be a bundle scenario.
\begin{enumerate}
\item An empirical model $e \in \Emp(\Sigma,O)$ on the scenario $(\Sigma,O)$ is called {\it non-contextual} if it lies in the image of $\Theta_{(\Sigma,O)}$. 
\item An empirical model $p \in \bEmp(f)$ on the bundle scenario $f$ is called {\it non-contextual} if it lies in the image of $\Theta_f$.
\end{enumerate}
Otherwise, the empirical model is called {\it contextual} in both cases.
\end{defn}
Note that  part (1) of Definition \ref{def:Contextbbbb} for contextuality agrees with the original definition given in \cite{abramsky2011sheaf}. {Our main result of Section \ref{sec:convexity-and-contextuality} is the compatibility of the notion of contextuality for empirical models and distributions on the three kinds of scenarios.
} 

\begin{thm}\label{thm:contextuality}
Let $(\Sigma,O)$ be a scenario and $f:\Gamma\to \Sigma$ be a bundle scenario.
\begin{enumerate}
\item An empirical model  $e \in \Emp(\Sigma,O)$ is contextual if and only if 
 $\eta_{(\Sigma,O)}(e) \in \bEmp(f_{(\Sigma,O)})$ 
 is contextual.
\item  An empirical model $p \in \bEmp(f)$ is contextual if and only if 
 $\zeta_{f}(p)\in \sDist(Nf)$ 
 is contextual.
\end{enumerate}
\end{thm} 
\begin{proof}
We will prove part (2). The proof of part (1) is similar.  Sending $s \in \bSect(f)$ to $Ns \in \sSect(Nf)$ gives a function $\bSect(f)\to \sSect(Nf)$ that makes the following diagram commute:
\begin{equation}\label{Diag:Deltaaa}
\begin{tikzcd}[column sep=huge,row sep=large]
\bSect(f) \arrow[r,"\delta_f"] \arrow[d,""'] & \bEmp(f) 
\arrow[d,"\cong"',"\zeta_{f}"] \\
 \sSect(Nf) \arrow[r]  
 \arrow[r,"\delta_{Nf}"]& \sDist(Nf)
\end{tikzcd}
\end{equation}
We know that $\delta_{Ng}$ is injective for any morphism $g$ of $\catbScen$. Using this fact, the naturality of $\zeta$ and both of horizontal $\delta$ maps, we see that the left vertical map is natural.  
In addition, it's clear that this map is injective. 
For a section $\tilde{s} \in \sSect(Nf)$ and $\sigma\in \Sigma$ we have $f(\tilde{s}_1(\sigma))=(Nf)_1(\tilde{s}_1(\sigma))=\sigma$. Since  $f$ is a discrete over 
vertices, we obtain that $\tilde{s}_1(\sigma)$ is a vertex in $\Gamma$. 
Thus $\tilde{s}_1$ is a simplicial complex map. More precisely, it belongs to $\bSect(f)$, and by part (1) of Proposition \ref{pro:fNtoN} we obtain that $N\tilde{s}_1=\tilde{s}$. Therefore  
 Diagram (\ref{Diag:Deltaaa})
induces the following commutative diagram:
\begin{equation}
\begin{tikzcd}[column sep=huge,row sep=large]
D_R(\bSect(f)) \arrow[r,"\Theta_f"] \arrow[d,"\cong"'] & \bEmp(f) 
\arrow[d,"\cong"',"\zeta_{f}"] \\
 D_R(\sSect(Nf)) \arrow[r]  
 \arrow[r,"\Theta_{Nf}"]& \sDist(Nf)
\end{tikzcd}
\end{equation}
\end{proof}



\bibliography{bib}

\begin{thebibliography}{10}

\bibitem{abramsky2011sheaf}
S.~Abramsky and A.~Brandenburger, ``The sheaf-theoretic structure of
  non-locality and contextuality,'' {\em New Journal of Physics}, vol.~13,
  no.~11, p.~113036, 2011.

\bibitem{okay2022simplicial}
C.~Okay, A.~Kharoof, and S.~Ipek, ``Simplicial quantum contextuality,'' {\em
  arXiv preprint arXiv:2204.06648}, 2022.

\bibitem{karvonen2018categories}
M.~Karvonen, ``Categories of empirical models,'' in {\em 15th International
  Conference on Quantum Physics and Logic (QPL 2018)} (P.~Selinger and
  G.~Chiribella, eds.), vol.~287 of {\em Electronic Proceedings in Theoretical
  Computer Science}, pp.~239--252, 2019.

\bibitem{barbosa2023closing}
R.~S. Barbosa, M.~Karvonen, and S.~Mansfield, ``Closing {B}ell: Boxing black
  box simulations in the resource theory of contextuality,'' in {\em Samson
  Abramsky on Logic and Structure in Computer Science and Beyond}
  (A.~Palmigiano and M.~Sadrzadeh, eds.), vol.~25 of {\em Outstanding
  Contributions to Logic}, Springer, 2023.

\bibitem{abramsky2015contextuality}
S.~Abramsky, R.~S. Barbosa, K.~Kishida, R.~Lal, and S.~Mansfield,
  ``Contextuality, cohomology and paradox,'' in {\em 24th EACSL Annual
  Conference on Computer Science Logic (CSL 2015)} (S.~Kreutzer, ed.), vol.~41
  of {\em Leibniz International Proceedings in Informatics (LIPIcs)},
  pp.~211--228, Schloss Dagstuhl--Leibniz-Zentrum fuer Informatik, 2015.

\bibitem{beer2018contextuality}
K.~Beer and T.~J. Osborne, ``Contextuality and bundle diagrams,'' {\em Physical
  Review A}, vol.~98, no.~5, p.~052124, 2018.

\bibitem{terra2019measures}
M.~Terra~Cunha, ``On measures and measurements: a fibre bundle approach to
  contextuality,'' {\em Philosophical Transactions of the Royal Society A},
  vol.~377, no.~2157, p.~20190146, 2019.

\bibitem{abramsky2019comonadic}
S.~Abramsky, R.~S. Barbosa, M.~Karvonen, and S.~Mansfield, ``A comonadic view
  of simulation and quantum resources,'' in {\em 34th Annual ACM/IEEE Symposium
  on Logic in Computer Science (LiCS 2019)}, pp.~1--12, IEEE, 2019.

\bibitem{karvonen2021neither}
M.~Karvonen, ``Neither contextuality nor nonlocality admits catalysts,'' {\em
  Physical Review Letters}, vol.~127, p.~160402, 2021.

\bibitem{kharoof2022simplicial}
A.~Kharoof and C.~Okay, ``Simplicial distributions, convex categories and
  contextuality,'' {\em arXiv preprint arXiv:2211.00571}, 2022.

\bibitem{okay2022mermin}
C.~Okay, H.~Y. Chung, and S.~Ipek, ``Mermin polytopes in quantum computation
  and foundations,'' {\em arXiv preprint arXiv:2210.10186}, 2022.

\bibitem{kharoof2023topological}
A.~Kharoof, S.~Ipek, and C.~Okay, ``Topological methods for studying
  contextuality: {$N$}-cycle scenarios and beyond,'' {\em arXiv preprint
  arXiv:2306.01459}, 2023.

\bibitem{raussendorf2016cohomological}
R.~Raussendorf, ``Cohomological framework for contextual quantum
  computations,'' {\em Quantum Information and Computation}, vol.~19,
  no.~13{\&}14, pp.~1141--1170, 2019.

\bibitem{raussendorf2023putting}
R.~S. Barbosa, M.~Karvonen, and S.~Mansfield, ``Putting paradoxes to work:
  contextuality in measurement-based quantum computation,'' in {\em Samson
  Abramsky on Logic and Structure in Computer Science and Beyond}
  (A.~Palmigiano and M.~Sadrzadeh, eds.), vol.~25 of {\em Outstanding
  Contributions to Logic}, Springer, 2023.

\bibitem{riehl2017category}
E.~Riehl, {\em Category theory in context}.
\newblock Courier Dover Publications, 2017.

\bibitem{jacobs2010convexity}
B.~Jacobs, ``Convexity, duality and effects,'' in {\em IFIP International
  Conference on Theoretical Computer Science}, pp.~1--19, Springer, 2010.

\bibitem{mac2013categories}
S.~Mac~Lane, {\em Categories for the working mathematician}, vol.~5.
\newblock Springer Science \& Business Media, 2013.

\bibitem{goerss2009simplicial}
P.~G. Goerss and J.~F. Jardine, {\em Simplicial homotopy theory}.
\newblock Springer Science \& Business Media, 2009.

\end{thebibliography}
\bibliographystyle{ieeetr}

\appendix


\section{Simplicial complexes}\label{sec:simplicial-complexes}

{
Simplicial complexes are one type of combinatorial model we use to represent spaces in this paper. This section introduces the simplicial complex that represents the standard topological $n$-simplex. These simplices can be related by maps of simplicial complexes that come from a certain category, called the simplex category. Using these morphisms, we provide alternative characterizations of local surjectivity and being discrete over vertices.
}
 
Let $\Delta:\catSet \to \catsComp$ denote the functor that sends a set $U$ to the simplicial complex $\Delta^U$ given by the power set of $U$, i.e., $(\Delta^U)_0=U$ and simplicial are all nonempty subsets of $U$. 
A set map $g:U'\to U$ induces a map of simplicial complexes, which will be denoted by $g:\Delta^{U'}\to \Delta^U$, defined by $g(\sigma')=\set{g(x'):\, x'\in \sigma'}$ for $\sigma'\subset U'$. {Next we introduce an important subcategory known as the simplex category and consider the restriction of the functor on this category.}
The {\it simplex category} $\catDelta$ {is defined as follows}  \cite{goerss2009simplicial}:
\begin{itemize}
\item Objects are given by $[n]=\set{0,1,\cdots,n}$ where $n\geq 0$.
\item A morphism $\theta:[m]\to [n]$ is an order preserving function, that is,
$$
\theta(i)\leq \theta(j),\;\; \forall i\leq j \in [m].
$$ 
\end{itemize}
There are two kinds of distinguished maps: 
$$
d^i:[n-1] \to [n] \;\;\;\text{ and }\;\;\; s^j:[n+1] \to [n],
$$
where the face maps $d^i$  skips $i$ in the image, and the degeneracy maps $s^j$   has a double preimage at $j$. An arbitrary morphism $\theta$ can be written as a composite of face and degeneracy maps.

We can extend this category to the augmented simplex category $\catDelta_+$ by adding the initial object $[-1]=\emptyset$.
We will write $\Delta^n$ for the simplicial complex $\Delta^{[n]}$.
The assignment $[n]\mapsto \Delta^n$ gives a functor
$$
\Delta: \catDelta_+ \to s\catComp.
$$ 

\begin{defn}\label{def:right-lifting}
{\rm
Let $\catC$ be a category. 
We say that $f$ has the {\it right lifting property} with respect to $g$ if for every commutative diagram in $\catC$ consisting of the solid arrows
$$
\begin{tikzcd}[column sep=huge,row sep=large]
A \arrow[r] \arrow[d,"g"'] & X \arrow[d,"f"] \\
B \arrow[r] \arrow[dashed,ru] & Y
\end{tikzcd}
$$
the dashed arrow exists making both triangles commute.
}
\end{defn}

\begin{pro}\label{pro:equivalent-LS} 
For a map $f: \Gamma \to \Sigma$ in $s\catComp$ of simplicial complexes the following conditions are equivalent: 
\begin{enumerate}
\item $f$ is locally surjective (Definition \ref{def:S-LS-DV-BS}).
\item $f$ has the right lifting property with respect to every injective map 
$\Delta^m \hookrightarrow \Delta^n$ where $m\geq 0$. 
\item $f$ has the right lifting property with respect to every $d^n: \Delta^{n-1} \hookrightarrow \Delta^n$ where $n\geq 1$. 
\end{enumerate}
\end{pro}
\begin{proof}
It is obvious that  (2) implies (3). 
To prove that (3) implies $(1)${, for} $\gamma=\{v_0,v_1,\cdots,v_m\} \in \Gamma$ and $\tau=\{f(v_0),f(v_1),\cdots,f(v_m),u_{m+1},\cdots,u_n\} \in \St(f(\gamma))$, we define 
$g: \Delta^m \to \Gamma$ by $g(i)=v_i$ for every $0\leq i \leq m$, and $k: \Delta^n \to \Sigma$ by 
$$
k(i)=
\begin{cases}
f(v_i) & \text{if}~ 0 \leq i \leq m \\
u_i & \text{if}~ i >m. 
\end{cases}
$$
Let $d: \Delta^m \to \Delta^n$ be the composition $ d^n\circ \cdots\circ d^{m+2}\circ d^{m+1}$ (which sends $i \in {(\Delta^m)_0}$ to $i \in {(\Delta^n)_0}$). Then the following diagram commutes:
 \begin{equation}\label{DDiag1}
\begin{tikzcd}[column sep=huge,row sep=large]
\Delta^m  \arrow[r,"g"]
 \arrow[d,hook, "d"'] & \Gamma 
 \arrow[d,"f"] \\
\Delta^n \arrow[r,"k"] & \Sigma 
\end{tikzcd}
\end{equation}
By applying part (3)  $n-m$ times we obtain a map $h:\Delta^n \to \Gamma$ that satisfies $h \circ d =g$ and $f \circ h=k$. Thus $
\gamma=\im(g) \subset 
\im(h)$, which means that $\im(h) \in \St(\gamma)$. Also we obtain that $f(\im (h))= \im (k) = \tau$.
 
Finally, we prove that  (1) implies (2). Consider the following diagram  
\begin{equation}\label{DDiag2}
\begin{tikzcd}[column sep=huge,row sep=large]
\Delta^m  \arrow[r,"g"]
 \arrow[d,hook,"l"'] & \Gamma 
 \arrow[d,"f"] \\
\Delta^n \arrow[r,"k"] & \Sigma 
\end{tikzcd}
\end{equation} 
Let {$\gamma=\im (g)$ and $\tau=\im (k)$}. By {the commutativity of} Diagram (\ref{DDiag2}) we {have} that $f(\gamma)\subseteq \tau$, hence 
$\tau \in \St(f(\gamma))$.  By condition (1) {there exists} $\sigma \in \St(\gamma)$ such that $f(\sigma)=\tau$.
We will denote by $i_{m+1},\cdots, i_n$ the vertices in $\Delta^n$ that do not belong to $\im (l)$. 
Since $f(\sigma)=\tau=\im(k)$, there exists $z_{m+1},\cdots,z_n \in \sigma$ such that $f(z_j)=k(i_j)$ for every $m+1 \leq j \leq n$. We define $h: \Delta^n \to \Gamma$ by 
$$
h(r)=
\begin{cases}
g(i)   &  r=l(i) \\
z_j    &  r=i_j.
\end{cases}
$$
{Then} $h$ is well-defined since 
$\Image(h)=\{g(0),\cdots,g(m),z_{m+1},\cdots, z_n\}=\gamma \cup \{z_{m+1},\cdots, z_n\} \subseteq \sigma$, thus $\Image (h) \in \Gamma$. In addition, by the definition of $h$, we have 
$h(l(i))=g(i)$ for   $0 \leq i \leq m$, and 
$$
f(h(r))=
\begin{cases}
f(g(i))=k(l(i))=k(r)   &  r=l(i) \\
f(z_j)=k(i_j)=k(r)    &  r=i_j.
\end{cases}
$$
Therefore $h$ is a lifting for Diagram (\ref{DDiag2}).
\end{proof}

\begin{pro}\label{pro:equivalent-DV}
For a map $f: \Gamma \to \Sigma$ in $s\catComp$ of simplicial complexes the following conditions are equivalent: 
\begin{enumerate}
\item $f$ is discrete over vertices (Definition \ref{def:S-LS-DV-BS}).
\item $f$ has the right lifting property with respect to every surjective map 
$\Delta^n \twoheadrightarrow \Delta^m$ where $m\geq 0$. 
\item $f$ has the right lifting property with respect to  $s^0:\Delta^{1} \twoheadrightarrow \Delta^{0}$. 
\end{enumerate}
\end{pro}
\begin{proof}
{It is obvious that} (2) implies (3).
To prove that (3) implies (1){, for}  $x,y \in \Gamma_0$ such that $f(x)=f(y)$ suppose that $\{x,y\}\in \Gamma$. We define  $g: \Delta^1 \to \Gamma$ by $g(0)=x$, $g(1)=y$, and $k:\Delta^0 \to \Sigma$ by $k(0)=f(x)$.
{We} obtain the following commutative diagram
\begin{equation}\label{dia:ghf}
\begin{tikzcd}[column sep=huge,row sep=large]
\Delta^1  \arrow[r,"g"]
 \arrow[d,twoheadrightarrow,"s^0"'] & \Gamma 
 \arrow[d,"f"] \\
\Delta^0 \arrow[r,"k"] & \Sigma 
\end{tikzcd}
\end{equation}
By condition (3) there exists $h:\Delta^0 \to \Gamma$  that gives a lifting for Diagram (\ref{dia:ghf}). In particular, 
$$
x=g(0)=h(s^0(0))=
h(s^0(1))=g(1)=y.
$$
Finally, we prove that  $(1)$ implies $(2)
$. 
We have the following commutative diagram  
\begin{equation}\label{DDiag3}
\begin{tikzcd}[column sep=huge,row sep=large]
\Delta^n  \arrow[r,"g"]
 \arrow[d,twoheadrightarrow,"s"'] & \Gamma 
 \arrow[d,"f"] \\
\Delta^m \arrow[r,"k"] & \Sigma 
\end{tikzcd}
\end{equation}
If $s(i_1)=s(i_2)$ then by Diagram (\ref{DDiag3}) we have 
$$
f(g(i_1))=k(s(i_1))=k(s(i_2))
=f(g(i_2)).
$$
On the other hand, $\{g(i_1),g(i_2)\} \subset \im (g)$ and thus $\{g(i_1),g(i_2)\} \in \Gamma$. By condition (1), we conclude that $g(i_1)=g(i_2)$. We define $h: \Delta^m \to \Gamma$ by 
$h(i)=g(s^{-1}(i))$. 
As a map between the vertices, $h$ is well-defined since $s$ is surjective and we proved that the preimage of $i$ under $s$ maps to a single point under $g$. 
In addition, the map $h$ is a {morphism of} $\catsComp$ since $\im (h) = \im (g) \in \Gamma$.
It remains to prove that $h$ is a lifting for the Diagram (\ref{DDiag3}).
Given $i \in (\Delta^n)_0 $, by the definition of $h$, we have 
$h(s(i))= g(i)$.
Now given $i \in (\Delta^m)_0$, we have 
$$
f(h(i))=f(g(s^{-1}(i)))=
k(s(s^{-1}(i)))=k(i).
$$
\end{proof}

\Cor{\label{cor:equivalent-bundle}
A simplicial complex map $f:\Gamma\to \Sigma$ is a bundle scenario (Definition \ref{def:S-LS-DV-BS}) if and only if $f$ has the right lifting property with respect to $\theta:\Delta^m \to \Delta^n$ for all morphisms $\theta$ of the augmented simplex category $\catDelta_+$.
}

\subsection{Pull-backs {of} bundle scenarios}\label{subsec: Pullback}

{The key observation in this section is that pull-backs along bundle scenarios behave particularly nice and they can be described easily. In addition, we show that pull-back of a bundle scenario is also a bundle scenario.
}

Let $f:\Gamma\to \Sigma$ be a bundle scenario, and let $\pi:\Sigma\to \Sigma'$ be a simplicial complex map.
The pull-back diagram associated to $f$ and $\pi$  consists of a commutative diagram of simplicial complex maps
\begin{equation}\label{dia:pull-back}
\begin{tikzcd}[column sep=huge,row sep=large]
\Gamma  
 \arrow[d,"f"'] & \pi^{\ast}(\Gamma) 
 \arrow[d,"\tilde f"] \arrow[l,"\tilde \pi"'] \\
\Sigma  & \Sigma' \arrow[l,"\pi"]
\end{tikzcd}
\end{equation}
where $\pi^*(\Gamma)$ is the {\it pull-back}: A simplicial complex with   vertex set consisting of pairs $(x,y)\in \Gamma_0 \times  \Sigma'_0$ such that $f(x)=\pi(y)$ 
and simplex set
$$
\pi^\ast\Gamma = \set{ \set{(x_1,y_1),\cdots,(x_k,y_k)} \subset (\pi^\ast\Gamma)_0:\, \set{x_1,\cdots,x_k}\in \Gamma,\; \set{y_1,\cdots,y_k}\in \Sigma' }.
$$    
The maps $\tilde f$ and $\tilde \pi$ are given by projection.

\Pro{\label{pro:pull-back}
Diagram (\ref{dia:pull-back}) is a pull-back in the category of simplicial complexes.}
\Proof{
Let $\gamma \in \Gamma$ and $\sigma' \in \Sigma'$ be such that $f(\gamma)=\pi(\sigma')=\{z_1,\cdots,z_n\}$ where $z_i \neq z_j$ for $i \neq j$. Since $f$ is discrete over vertices there exists distinct vertices $x_1,\cdots,x_n$ in $\Gamma$ such that 
$\gamma=\{x_1,\cdots,x_n\}$ and $f(x_i)=z_i$. On the other hand, for  
$1 \leq i \leq n$, we {write} $\{y_{i1},\cdots,y_{im_i}\}$ {for} the simplex $\pi^{-1}(z_i) \cap \sigma'$. Then we have $\sigma'=\cup_{i=1}^n\{y_{i1},\cdots,y_{im_i}\}$. We conclude that 
$\tau=\cup_{i=1}^n\{(x_i,y_{i1}),\cdots,(x_i,y_{im_i})\}$ is the {unique} simplex in $\pi^{\ast}(\Gamma)$ such that $\tilde{f}(\tau)=\sigma'$ and $\tilde{\pi}(\tau)=\gamma$. 
} 

\begin{pro}\label{pro:pull-back-bundle}
The map $\tilde{f}$ in Diagram (\ref{dia:pull-back}) is a bundle scenario. 
\end{pro}
\Proof{
Let $\sigma'=\{y_1,\cdots,y_n\} \in \Sigma'$. Since $f$ is surjective, there exists 
$\gamma \in \Gamma$ such that $f(\gamma)=\pi(\sigma')$. 
For   $1 \leq i \leq n$, we choose $x_i \in \gamma$ such that $f(x_i)=\pi(y_i)$. Then 
$\{(x_1,y_1),\cdots,(x_n,y_n)\} \in \pi^{\ast}(\Gamma)$ and $\tilde{f}(\{(x_1,y_1),\cdots,(x_n,y_n)\})=\sigma'$. This shows that $\tilde{f}$ is surjective. To prove that  $\tilde{f}$ is locally surjective, let $\{(x_1,y_1),\cdots,(x_k,y_k)\} \in \pi^{\ast}(\Gamma)$ and 
$\sigma'=\{y_1,\cdots,y_k,y_{k+1},\cdots,y_n\} \in \St(\{y_1,\cdots,y_k\})$. Then $\pi(\sigma') \in \St(\pi(\{y_1,\cdots,y_k\}))$. Note that $f(\{x_1,\dots,x_k\})=\{y_1,\dots,y_k\}$. Since $f$ is locally surjective,  there exists $\gamma \in \St(\{x_1,\cdots,x_k\})$ such that 
$f(\gamma)=\pi(\sigma')$.  
Thus for every $k+1 \leq i\leq n$ there exists $x_i \in \gamma$ such that $f(x_i)=\pi(y_i)$. 
We conclude that $\{(x_1,y_1),\cdots,(x_n,y_n)\} \in \pi^{\ast}(\Gamma)$
and $\tilde{f}(\{(x_1,y_1),\cdots,(x_n,y_n)\})=\sigma'$.
To prove that $\tilde{f}$ is discrete over vertices,  consider $(x_1,y_1),(x_2,y_2) \in \pi^{\ast}(\Gamma)_0$ such that $\tilde{f}(x_1,y_1)=\tilde{f}(x_2,y_2)$, that is $y_1=y_2$. If $\{(x_1,y_1),(x_2,y_2)\} \in \pi^{\ast}(\Gamma)$
then $\{x_1,x_2\} \in \Gamma$. But $f(x_1)=\pi(y_1)
=\pi(y_2)=f(x_2)$. Thus we conclude that $x_1=x_2$ since $f$ is discrete over vertices. Therefore $(x_1,y_1)=(x_2,y_2)$.
}

\subsection{{Bundles of nerve complexes}}\label{sec:simp-rel-kl-mor}

{The nerve complex functor $\hat N:\catsComp \to \catsComp$ given in Definition \ref{def:hatN} is used in the description of morphisms of $\catbScen$. We show that this functor preserves bundle scenarios.
}


\Pro{\label{pro:N-bundle-sce}
If $f:\Gamma\to \Sigma$ is a bundle scenario then $\hat N f:\hat N\Gamma\to \hat N \Sigma$ is also a bundle scenario. 
} 
\Proof{
Consider $\{\sigma_1,\cdots,\sigma_n\} \in \hat N \Sigma$. We have $\cup_{i=1}^n \sigma_i \in \Sigma$. Surjectivity of $f$ implies that 
there exists $\gamma \in \Gamma$ such that $f(\gamma)=\cup_{i=1}^n\sigma_i$. 
Thus we have $\gamma_1,\cdots,\gamma_n \subseteq \gamma$ such that $f(\gamma_i)=\sigma_i$. 
This shows that $\{\gamma_1,\cdots,\gamma_n \}\in \hat N \Gamma$ and 
$\hat N f(\{\gamma_1,\cdots,\gamma_n \})=\{\sigma_1,\cdots,\sigma_n \}$. 
Now, we prove that $\hat N f$ is discrete over vertices. 
Consider distinct vertices $\gamma_1,\gamma_2 $  in $\hat N \Gamma$ such that 
$\hat N f(\gamma_1)=\hat N f(\gamma_2)$. 
This means that $f(\gamma_1)=f(\gamma_2)$. Since $f$ is discrete over vertices we get that $|\gamma_1|=|\gamma_2|$. 
On the other hand, $f(\gamma_1 \cup \gamma_2)=f(\gamma_1) \cup f(\gamma_2)=f(\gamma_1)$ and $|\gamma_1 \cup \gamma_2| > |\gamma_1|$. We conclude that $\gamma_1 \cup \gamma_2 \notin \Gamma$, hence $\{\gamma_1,\gamma_2\} \notin  \hat N \Gamma$.    
Finally, we prove that $\hat N f$ is locally surjective. Given $\{\gamma_1,\cdots,\gamma_n\} \in \hat N \Gamma$, we want to prove the surjectivity of the restriction of $f$ on $\St(\{\gamma_1,\cdots,\gamma_n\})$:
%
$$
f|_{\{\gamma_1,\cdots,\gamma_n\}}:  \St(\{\gamma_1,\cdots,\gamma_n\}) \to \St(\{f(\gamma_1),\cdots,f(\gamma_n)\}).
$$
For $\{f(\gamma_1),\cdots,f(\gamma_n),\sigma_1,\cdots,\sigma_m\}\in \St(\{f(\gamma_1),\cdots,f(\gamma_n)\})$, the union 
$$\cup_{i=1}^{n}f(\gamma_i) \; \bigcup \; \cup_{i=1}^{m}\sigma_i  $$
belongs to $\Sigma$, which implies that the union
$$f(\cup_{i=1}^{n}\gamma_i) \bigcup \cup_{i=1}^{m}\sigma_i $$ 
belongs to $\St(f(\cup_{i=1}^{n}\gamma_i))$.
Since $f$ is locally surjective, there exists $\cup_{i=1}^{n}\gamma_i \subset \gamma \in \Gamma$ such that 
$f(\gamma)=f(\cup_{i=1}^{n}\gamma_i)\; \bigcup \; \cup_{i=1}^{m}\sigma_i$. 
In particular, there exists $\tau_1,\cdots,\tau_m \subset \gamma$ such that $f(\tau_i)=\sigma_i$. Thus $\{\gamma_1,\cdots,\gamma_n,\tau_1,\cdots,\tau_m\}\in \St(\{\gamma_1,\cdots,\gamma_n\})$ and 
$$
\hat N f(\{\gamma_1,\cdots,\gamma_n,\tau_1,\cdots,\tau_m\})=\{f(\gamma_1),\cdots,f(\gamma_n),\sigma_1,\cdots,\sigma_m\}.$$
}

\subsection{{Lemmas: Category of bundle scenarios}}
\label{sec:lem-category-bundle-scen} 

{In $\catbScen$ there are two types of morphisms. In this section we prove two results about type I morphisms. First one of these is used in the definition of composition of type I morphisms. The second result constitutes part of the identity morphism of an object in $\catbScen$. Finally we show that the composition rule in this category is associative.
}
 
\begin{lem}\label{lem:pull-hat-N}
For a type I morphism given in Diagram (\ref{dia:typeI}),
the composition of the following squares is a pull-back
%
\begin{equation}\label{DDiag6}
\begin{tikzcd}[column sep=huge,row sep=large]
\hat N \Gamma  
 \arrow[d,"\hat N f"'] & \arrow[l,"\mu_{\Gamma}"'] \hat N^2 \Gamma  
 \arrow[d,"{\hat N}^2 f"'] & \hat N\pi^{\ast}(\hat N \Gamma) 
 \arrow[d,"\hat N f^\pi"] \arrow[l,"\hat N l"']\\
\hat N\Sigma  & \arrow[l,"\mu_{\Sigma}"] 
\hat N^2 \Sigma  & \hat N \Sigma' \arrow[l,"\hat N \pi"]
\end{tikzcd}
\end{equation} 
\end{lem}
\Proof{ 
By  
Proposition \ref{pro:pull-back} and Proposition \ref{pro:N-bundle-sce} $(\mu_{\Sigma}\circ \hat N \pi)^\ast(\hat N \Gamma)$ is the pull-back of ${\mu}_{\Sigma}\circ \hat N\pi$ along $\hat N f$. 
Given a vertex $(\gamma,\sigma')$ in $(\mu_{\Sigma}\circ \hat N \pi)^\ast(\hat N \Gamma)$, we have $f(\gamma)=\overline{\pi}(\sigma')$. Let $\sigma'=\{y_{1},\cdots,y_n\}$. Then $f(\gamma)=\cup_{j=1}^{n}\pi(y_{j})$. 
By Remark \ref{Improppp} there is a unique simplex $\set{\gamma_1\cdots,\gamma_n} \in {\hat N\Gamma}$
such that 
$$
\cup_{j=1}^{n}\gamma_{j}=\gamma \,\, \text{and} \,\,f(\gamma_{j})=\pi(y_{j})
\,\, , \,\forall 1\leq j \leq n.
$$
{Let} $S(\gamma,\sigma')$ {denote}
the vertex 
$ 
\{(\gamma_{1},y_{1}),\cdots,(\gamma_{n},y_{n})\}$
in $\hat N(\pi^{\ast}(\hat N \Gamma))
$.
{Using this we can construct}
a simplicial complex map $S:(\mu_{\Sigma}\circ \hat N \pi)^\ast(\hat N \Gamma) \to \hat N(\pi^{\ast}(\hat N \Gamma))$.
In addition, we have
$$ 
\begin{aligned}
\hat N f^{\pi} \left(S (\gamma,\sigma') \right)
=(\hat N f^{\pi})(\{(\gamma_{1},y_{1}),\cdots,(\gamma_{n},y_{n})\})=\sigma'
\end{aligned}
$$
and
$$
\begin{aligned}
{\mu}_\Gamma & \circ (\hat N l) \circ S  \left(\gamma,\sigma'\right)\\
&={\mu}_\Gamma  \circ (\hat N l)(\{(\gamma_{1},y_{1}),\cdots,(\gamma_{n},y_{n})\})
\\
&={\mu}_\Gamma(\{\gamma_{1},\cdots,\gamma_{n}\})\\
&=\cup_{j=1}^{n}\gamma_{j}
=\gamma.
\end{aligned}
$$
Finally, one can check that the induced map from
$\hat N(\pi^{\ast}(\hat N \Gamma))$ to $(\mu_{\Sigma}\circ \hat N \pi)^\ast(\hat N \Gamma)$ is the inverse of $S$. 
Therefore the composition of the  squares in Diagram (\ref{DDiag6}) is a pull-back. 
}

\begin{lem}\label{DDDeltsqu}
Let $f:\Gamma \to \Sigma$ be a simplicial complex map that is discrete over vertices. Then the following diagram
\begin{equation}\label{DDiag9}
\begin{tikzcd}[column sep=huge,row sep=large]
\hat N\Gamma  
 \arrow[d,"\hat N f"'] & \Gamma 
 \arrow[d,"f"] \arrow[l,hook,"\delta_{\Gamma}"']\\
\hat N\Sigma  & \Sigma \arrow[l,hook, "\delta_{\Sigma}"]
\end{tikzcd}
\end{equation}
is a pullback square.
\end{lem}
\Proof{Given  simplices $\{\gamma_1,\cdots,\gamma_n\} \in \hat N \Gamma$ and $\sigma \in \Sigma$ such that $\{f(\gamma_1),\cdots,f(\gamma_n)\}=\delta_{\Sigma}(\sigma)$,  the images $f(\gamma_1),\cdots,f(\gamma_n)$ are the vertices of $\sigma$. 
The map $f$ is discrete over vertices, thus every $\gamma_i$ is a vertex in $\Gamma$. We conclude that $\gamma = \cup_{i=1}^n \gamma_i$ is the unique simplex in $\Gamma$ with 
$\delta_{\Sigma}(\gamma)=\{\gamma_1,\cdots,\gamma_n\}$ and   $f(\gamma)=\sigma$.}

\Lem{\label{lem:bScen-is-well-defined}
The composition given by Equation (\ref{eq:composition-bundle}) is associative. 
} 
\Proof{
Given a type II morphism $\alpha:f \to f'$ and  relations 
$\pi: \Sigma' \to \hat N \Sigma$, $\pi': \Sigma'' \to \hat N \Sigma'$, we have
\begin{equation}\label{eq:FFForComp}
(\pi \kleisli \pi')^\ast(\hat N \alpha)
=(\pi')^\ast(\hat N \pi^{\ast}(\hat N \alpha)).
\end{equation}
{In Lemma \ref{lem:pull-hat-N} we proved that 
$(\mu_{\Sigma}\circ\hat N \pi)^\ast(\hat N \Gamma)=\hat N \pi^{\ast}(\hat N \Gamma)$. Using this we conclude that}   
$(\mu_{\Sigma}\circ\hat N \pi)^\ast(\hat N \alpha)=\hat N \pi^{\ast}(\hat N \alpha)$. {Then we have} 
$(\pi \kleisli \pi')^\ast(\hat N \alpha)=(\mu_{\Sigma}\circ \hat N \pi \circ \pi')^\ast(\hat N \alpha)=
(\pi')^\ast((\mu_{\Sigma}\circ\hat N \pi)^\ast(\hat N \alpha))=(\pi')^\ast(\hat N \pi^{\ast}(\hat N \alpha))$.
Given $(\pi_1,\alpha_1):f \to f'$, $(\pi_2,\alpha_2):f' \to f''$, and $(\pi_3,\alpha_3):f'' \to f'''$, we have 
$$
\begin{aligned}
(\pi_3,\alpha_3)\circ \left((\pi_2,\alpha_2)\circ (\pi_1,\alpha_1)\right)&=
(\pi_3,\alpha_3)\circ (\pi_1 \kleisli \pi_2,\alpha_2 \circ \pi_2^{\ast}(\hat N \alpha_1))\\
&=\left((\pi_1 \kleisli \pi_2)\kleisli \pi_3 ,\alpha_3 \circ {\pi_3}^{\ast}(\hat N\left(\alpha_2 \circ \pi_2^{\ast}(\hat N \alpha_1)\right))\right)
\\
&=\left((\pi_1 \kleisli \pi_2)\kleisli \pi_3 ,\alpha_3 \circ {\pi_3}^{\ast}(\hat N \alpha_2 \circ \hat N \pi_2^{\ast}(\hat N \alpha_1))\right)
\\
&=\left((\pi_1 \kleisli \pi_2)\kleisli \pi_3 ,\alpha_3 \circ {\pi_3}^{\ast}(\hat N \alpha_2) \circ {\pi_3}^{\ast}(\hat N \pi_2^{\ast}(\hat N \alpha_1))\right).
\end{aligned}
$$
On the other hand, 
$$
\begin{aligned}
\left((\pi_3,\alpha_3)\circ (\pi_2,\alpha_2)\right)\circ (\pi_1,\alpha_1)&=(\pi_2 \kleisli \pi_3,\alpha_3 \circ \pi_3^{\ast}(\hat N \alpha_2))\circ (\pi_1,\alpha_1)\\
&=\left(\pi_1 \kleisli (\pi_2\kleisli \pi_3) ,\alpha_3 \circ {\pi_3}^{\ast}(\hat N \alpha_2) \circ (\pi_2 \kleisli \pi_3)^{\ast}(\hat N \alpha_1)\right).
\end{aligned}
$$
By Equation (\ref{eq:FFForComp}) and the associativity of the Kleisli composition we obtain the equality.
}

\subsection{{Lemmas: Embedding scenarios into bundle scenarios}}
\label{sec:lem-embedding-into-bundle-scen}
 
{In this section we prove key lemmas that are used in the construction of the functor $\eE:\catScen \to \catbScen$ and showing that it is a fully-faithful embedding.} 
 
\begin{lem}\label{lem:pullback} 
The following diagram is a pull-back square
\begin{equation}\label{Dia: EoSig}
\begin{tikzcd}[column sep=large,row sep=large]
\hat N (\eE_O\Sigma) \arrow[d,"{\hat N f_{(\Sigma,O)}}"'] & (\eE_O\circ \pi)\Sigma' \arrow[d,"f_{{\eE}_O\circ \pi}"] \arrow[l,"l"] \\
\hat N\Sigma & \Sigma' \arrow[l,"\pi"]
\end{tikzcd}
\end{equation}
where $l(x',s)= (\pi(x'),s)$.  

\end{lem} 
\Proof{
First, we prove that the map $l$ is a valid morphism in $\catsComp$. Note that for $s \in \mathcal{E}_O(\pi(x'))$, the pair 
$(\pi(x'),s)$ is a simplex in $\mathcal{E}_O \Sigma$, that is, a vertex in $\hat N (\mathcal{E}_O \Sigma)$. 
For a simplex $(\sigma',s)\in (\mathcal{E}_O\circ \pi)\Sigma'$, we have   
$l(\sigma',s)=
\set{l(x',s|_{\pi(x')}):\, x'\in \sigma'}=\set{(\pi(x'),s|_{\pi(x')}):\, x'\in \sigma'}$. 
{Then the union
$$ 
\cup_{x'\in \sigma'}(\pi(x'),s|_{\pi(x')}) =
(\cup_{x'\in \sigma}\pi(x'),s) 
=(\overline{\pi}
(\sigma'),s) 
$$
is a simplex in $\mathcal{E}_O \Sigma$.}

{Next} we prove that Diagram (\ref{Dia: EoSig}) is a pullback square. 
It is obvious that the square commutes. 
Given $\{(\sigma_1,s_1),\cdots,(\sigma_n,s_n)\} \in \hat N(\mathcal{E}_O \Sigma)$ and $\sigma' \in \Sigma'$ such that 
$\hat N f_{(\Sigma,O)}((\sigma_1,s_1),\cdots,(\sigma_n,s_n))=\pi(\sigma')$. Since $\cup_{i=1}^{n}(\sigma_i,s_i) \in \eE_O\Sigma$ we have a unique 
$s \in \mathcal{E}_O(\cup_{i=1}^n \sigma_i)$ such that $s|_{\sigma_i}=s_i$. In addition, we have  
$\pi(\sigma')=\{\sigma_1,\cdots,\sigma_n\}$, and thus $\overline{\pi}(\sigma')=\cup_{i=1}^n \sigma_i$. Therefore $(\sigma',s)$ is the unique element in $(\mathcal{E}_O \circ \pi)\Sigma'$ with  $f_{\mathcal{E}_O \circ \pi}(\sigma',s)=\sigma'$
and $l(\sigma',s)=\{(\sigma_1,s_1),\cdots,(\sigma_n,s_n)\}$.
}

\begin{lem}\label{lem:ralpr} 
For morphisms
$(\pi,\alpha): (\Sigma,O) \to (\Sigma',O')$ and 
$(\pi',\alpha'): (\Sigma',O') \to (\Sigma'',O'')$ of $\catScen$ we have
$$
r(\alpha' \circ (\alpha \ast \Id_{\overline{\pi'}}))
=r(\alpha')\circ(\pi')^\ast
(\hat N r(\alpha))
$$
{where $r$ is defined in Diagram (\ref{diag:pi-r-alpha}).}
\end{lem}
\Proof{
Note that $\alpha' \circ (\alpha \ast \Id_{\overline{\pi'}})$ is a natural transformation $\eE_O \circ \overline{\pi \kleisli \pi'}  \to \eE_{O''}$. Therefore we have a map of simplicial complexes
$r(\alpha' \circ (\alpha \ast \Id_{\overline{\pi'}})):(\eE_O \circ (\pi \kleisli \pi'))\Sigma'' \to \eE_{O''}\Sigma''$.
Given $\sigma'' \in \Sigma''$ and $s\in \eE_O(\overline{\pi \kleisli \pi'}(\sigma''))$, we have
$$
r(\alpha'\circ (\alpha\ast \Id_{\overline{\pi'}}))(\sigma'',s)=(\sigma'',(\alpha'\circ (\alpha\ast \Id_{\overline{\pi'}})
)_{\sigma''}(s))=(\sigma'',\alpha'_{\sigma''}(\alpha_{\overline{\pi'}(\sigma'')}(s))).
$$
Observe that by Lemma \ref{lem:pullback} 
$$(\pi')^\ast(\hat N((\eE_O \circ \pi)\Sigma'))=(\pi')^\ast(\hat N \pi^\ast(\hat N(\eE_O \Sigma)) ).$$ 
Using Lemma \ref{lem:pull-hat-N} we obtain that 
$(\pi')^\ast(\hat N \pi^\ast(\hat N(\eE_O \Sigma)) )=(\pi \kleisli \pi')^\ast(\hat N(\eE_O \Sigma))$, which again by Lemma \ref{lem:pullback} is equal to 
$(\eE_O \circ (\pi \kleisli \pi'))\Sigma'' $. Thus we can see $(\sigma'',s)$ as an element of $(\pi')^\ast(\hat N((\eE_O \circ \pi)\Sigma'))$ and by applying $(\pi')^\ast(\hat N r(\alpha))$ to this element  we obtain $(\sigma'',\alpha_{\overline{\pi'}}(s))$ in 
$(\pi')^\ast(\hat N \eE_{O'}\Sigma')$. According to Lemma \ref{lem:pullback} we can identify this element with $(\eE_{O'}\circ \pi') \Sigma''$. Finally, by applying $r(\alpha')$ on $(\sigma'',\alpha_{\overline{\pi'}}(s))$ we obtain 
$(\sigma'',\alpha'_{\sigma''}(\alpha_{\overline{\pi'}(\sigma'')}(s)))$.
}

\begin{lem}\label{lem:assemble-alpha}
Let $\pi:\Sigma' \to \hat N \Sigma$ be a simplicial complex map. Any family of maps $\set{\alpha_x:\mathcal{E}_{O}(\pi(x)) \to O'_{x}}_{x \in \Sigma'_0}$ can be assembled into a natural transformation $\alpha: \mathcal{E}_O \circ \overline{\pi} \to \mathcal{E}_{O'}$.
\end{lem}
\begin{proof}
A natural transformation $\alpha :\mathcal{E}_O \circ \overline{\pi} \to \mathcal{E}_{O'}$ is a collection of maps 
$$
\set{\alpha_{\sigma}: \mathcal{E}_O(\overline{\pi}(\sigma)) \to  \mathcal{E}_{O'}(\sigma)}_{\sigma \in \Sigma'}
$$
such that for   $\sigma_1 \subset \sigma_2$ the following diagram commutes
\begin{equation}
\begin{tikzcd}[column sep=huge,row sep=large]
\mathcal{E}_O(\overline{\pi}(\sigma_2)) 
 \arrow[r,"\alpha_
 {\sigma_2}"] \arrow[d] & \mathcal{E}_{O'}(\sigma_2) \arrow[d] \\
 \mathcal{E}_O(\overline{\pi}(\sigma_1)) \arrow[r]  
 \arrow[r,"\alpha_
 {\sigma_1}"]& \mathcal{E}_{O'}(\sigma_1) 
\end{tikzcd}
\end{equation}
In particular, for $\tilde{x}\in \sigma$, the following diagram commutes
\begin{equation}
\begin{tikzcd}[column sep=huge,row sep=large]
\prod_{x \in \overline{\pi}(\sigma)}O_x 
 \arrow[r,"\alpha_
 {\sigma}"] \arrow[d] &  \prod_{x \in \sigma}O'_x \arrow[d] \\
\prod_{x \in \pi(\tilde{x})}O_x  \arrow[r]  
 \arrow[r,"\alpha_
 {\tilde{x}}"]& O'_{\tilde{x}} 
\end{tikzcd}
\end{equation}
{This} means that for $s \in \prod_{x \in \overline{\pi}(\sigma)}O_x $, we have  
$\alpha_{\sigma}(s)(\tilde{x})=\alpha_
{\tilde{x}}(s|_{\pi(\tilde{x})})$. Thus $\alpha$ is defined by the maps $\{\alpha_x\}_{x \in \Sigma'_0}$. On the other hand, for a collection of maps $\{\alpha_x:\prod_{y\in\pi(x)}O_y \to O'_x\}_{x \in \Sigma'_0}$, one can construct a natural {transformation} $\alpha :\mathcal{E}_O\circ \overline{\pi} \to \mathcal{E}_{O'}$ 
by defining $\alpha_{\sigma}(s)=\{\alpha_x(s|_{\pi(x)})\}_{x\in \sigma}$ for 
$\sigma \in \Sigma'$ and $s \in \prod_{x \in \overline{\pi}(\sigma)}O_x$.
\end{proof}

\begin{lem}\label{lem:assemble-typeII}
Let $\pi:\Sigma' \to \hat N \Sigma$ be a simplicial complex map.  Any family of maps 
$$\set{\gamma_x :\{x\} \times \mathcal{E}_{O}(\pi(x)) \to \{x\} \times O'_{x}}_{x \in \Sigma'_0}$$ 
can be assembled into a type II morphism
$\gamma: (\mathcal{E}_O \circ \pi)\Sigma' \to \mathcal{E}_{O'}\Sigma'$.
\end{lem}
\Proof{
A simplicial complex map $\gamma: (\mathcal{E}_O \circ \pi)\Sigma' \to \mathcal{E}_{O'}\Sigma'$ is determined by its restriction to the vertices of $(\mathcal{E}_O\circ \pi)\Sigma'$. Hence it is defined by the maps $\gamma|_{\{x\}\times\mathcal{E}_O(\pi(x))}$, where $x \in \Sigma'_0$. 
On the other hand, if we have a collection of maps  $\set{\gamma_x :\{x\} \times \mathcal{E}_{O}(\pi(x)) \to \{x\} \times O'_{x}}_{x \in \Sigma'_0}$ then we can construct a simplicial complex map $\gamma : (\mathcal{E}_O \circ \pi) \Sigma' \to \mathcal{E}_{O'}\Sigma'$ by defining $\gamma(x,s)=\gamma_x(x,s)$ for every vertex $(x,s) \in (\mathcal{E}_O\circ \pi)\Sigma'$. 
This is because, for a simplex $(\sigma,s)$ in $\mathcal{E}\circ \pi(\Sigma')$, where $\sigma=\{x_1,\cdots,x_n\}$, we have that
$$
\begin{aligned}
\gamma(\sigma,s) &=
{\gamma}(\{(x_1,s|_{\pi(x_1)}),\cdots,(x_n,s|_{\pi(x_n)})\})\\
&=\{(\gamma_{x_1}(s|_{\pi(x_1)}),\cdots,\gamma_{x_n}(s|_{\pi(x_n)})\}
\in \{\sigma\}\times \prod_{x \in \sigma}O'_x.
\end{aligned}
$$ 
Therefore {$\gamma(\sigma,s)$} is a simplex in $\mathcal{E}_{O'}\Sigma'$. In addition, $f_{(\Sigma',O')} \circ \gamma = f_{\mathcal{E}_O \circ \pi}$. 
}

\subsection{{Lemmas: Empirical model functor on bundle scenarios}}
\label{sec:lem-empirical-bundle-scen}
 
{In this section we provide important results for our study of push-forward empirical models. First we show that our construction $p\mapsto \hat Np$ given in Definition \ref{def:hatNp} produces a well-defined empirical model. Then we prove results concerning push-forwards along type I and type II morphisms.}

\begin{lem}\label{lem:hat-N-p}
{The family} $\hat Np$ {of distributions constructed} in Definition \ref{def:hatNp} is a well-defined empirical model on $\hat Nf$.
\end{lem}
\Proof{
First, we prove that $(\hat N p)_{\{\sigma_1\cdots\sigma
_n\}}$ is a well-defined distribution on $(\hat N f)^{-1}(\{\sigma_1\cdots\sigma
_n\})$. We have
$$
\begin{aligned}
\sum_{\{\gamma_1\cdots\gamma
_n\}\in (\hat N f)^{-1}(\{\sigma_1\cdots\sigma
_n\})} (\hat N p)_{\{\sigma_1\cdots\sigma
_n\}}(\{\gamma_1,\cdots,\gamma_n\}) &= \sum_{\{\gamma_1,\cdots,
\gamma
_n\} :\,f(\gamma_i)
=\sigma_i }p_{\cup_{i}\sigma_i}(\cup_{i}\gamma_i) \\
&=\sum_{\gamma :\,f(\gamma) \cup_{i}\sigma_i }p_{\cup_{i}\sigma_i}(\gamma)\\
&=
\sum_{\gamma \in f^{-1}(\cup_{i}\sigma_i) }p_{\cup_{i}\sigma_i}(\gamma)=1
\end{aligned}
$$
%
%
since $f$ is discrete over vertices.
Now, we prove the compatibility:
$$
\begin{aligned}
(\hat N p)_{\{\sigma_1,\cdots,\sigma
_k,\sigma_{k+1},\cdots,\sigma_n\}|_{ 
\{\sigma_1,\cdots,\sigma
_k\}}}(\{\gamma_1,\cdots,\gamma_k\})
&=\sum_{\{\gamma_{k+1},\cdots,\gamma_n\}: \,f(\gamma_i)=\sigma_i}
(\hat N p)_{\{\sigma_1,\cdots,\sigma_n\}}(\{\gamma_1,\cdots,\gamma_n\})\\
&=\sum_{\{\gamma_{k+1},\cdots,\gamma_n\}: \,f(\gamma_i)=\sigma_i}
p_{\cup_{i=1}^n\sigma_i(\cup_{i=1}^n\gamma_i)}\\
&=\sum_{\cup_{i=1}^k \gamma_i \subseteq \gamma \in f^{-1}(\cup_{i=1}^n\sigma_i)}
p_{\cup_{i=1}^n\sigma_i}(\gamma)\\
&=p_{\cup_{i=1}^n\sigma_i}|_{\cup_{i=1}^k\sigma_i}(\cup_{i=1}^k\gamma_i)\\
&=p_{\cup_{i=1}^k\sigma_i}(\cup_{i=1}^k\gamma_i)\\
&=(\hat N p)_{\{\sigma_1,\cdots,\sigma
_k\}}(\{\gamma_1,\cdots,\gamma_k\}).
\end{aligned}
$$
}

\begin{lem}\label{lem:push-typeI}
The push-forward {empirical model} $\pi_{\ast}(p)$ in Definition \ref{def:push-along-T1} belongs to $\bEmp(f^\pi)$. Moreover, for $\pi':\Sigma'' \to \hat N\Sigma'$, we have $(\pi \kleisli \pi')_\ast p={\pi'}_\ast(\pi_\ast p)$. 
\end{lem}
\Proof{ 
{Let $\sigma' \in \Sigma'$ and $\gamma' \in (f^{\pi})^{-1}(\sigma')$.} Observe that $\pi(\sigma')=\pi
(f^{\pi}(\gamma'))=\hat N f(l(\gamma'))$. Thus $l(\gamma')\in (\hat N f)^{-1}(\pi(\sigma'))$. In addition, 
$$
\sum_{\gamma' \in (f^{\pi})^{-1}(\sigma')}(\pi_{\ast}p)_{\sigma'}(\gamma')=
\sum_{\gamma' \in (f^{\pi})^{-1}(\sigma')} (\hat N p)_{\pi(\sigma')}(l(\gamma')).
$$
By 
Proposition \ref{pro:pull-back} this is equal to 
$\sum_{\gamma \in (\hat Nf)^{-1}(\pi(\sigma'))}(\hat N p)_{\pi(\sigma')}(\gamma)$, which is equal to $1$ according to Lemma \ref{lem:hat-N-p}. 
Now, given $\sigma'_1 \subseteq  \sigma'$ in $\Sigma'$ and $\gamma'_1 \in 
(f^{\pi})^{-1}(\sigma'_1)$, we have  
%
%
$$
\begin{aligned}
(\pi_{\ast}p)_{\sigma'}|_{\sigma'_1}(\gamma_1')
&=\sum_{\gamma_1' \subseteq \gamma' \in (f^{\pi})^{-1}(\sigma')}(\pi_{\ast}p)_{\sigma'}(\gamma') \\
&=\sum_{\gamma_1' \subseteq \gamma' \in (f^{\pi})^{-1}(\sigma')}(\hat N p)_{\pi(\sigma')}(l(\gamma'))\\
&=
\sum_{l(\gamma'_1) \subseteq \gamma \in (\hat N f)^{-1}(\pi(\sigma'))}(\hat N p)_{\pi(\sigma')}(\gamma)\\
&= (\hat N p)_{\pi(\sigma')}|_{\pi(\sigma_1')}(l(\gamma_1'))\\
&=
 (\hat N p)_{\pi(\sigma_1')}(l(\gamma'_1))=
(\pi_{\ast}p)_{\sigma'_1}(\gamma'_1).
\end{aligned}
$$
Now, given $\sigma \in \Sigma''$ and $\gamma \in \left((f^{\pi})^{\pi'}\right)^{-1}(\sigma)$, we have
$$
\begin{aligned}
\left({\pi'}_{\ast}(\pi_\ast p)\right)_{\sigma}(\gamma) &=
(\pi_{\ast}p)_{\overline{\pi'}(\sigma)}(\overline{l}_{f^{\pi}}(\gamma))\\
&=
p_{\overline{\pi}(\overline{\pi'}(\sigma))}\left(\overline{l}_f(\overline{l}_{f^{\pi}}(\gamma))\right)\\
&=p_{\overline{\pi \kleisli \pi'}(\sigma)}
(\overline{l}_{f,\pi \kleisli \pi'}(\gamma))\\
&=\left((\pi \kleisli \pi')_\ast p\right)_{\sigma}(\gamma).
\end{aligned}
$$
}

\begin{lem}\label{lem:push-typeII}
The push-forward {empirical model}  $\alpha_{\ast}p$ in Definition \ref{def:push-along-T2} belongs to $\bEmp(g)$. Moreover, given another type II morphism $\alpha': g \to h$, we have $(\alpha' \circ \alpha)_\ast p={\alpha'}_\ast(\alpha_\ast p)$.
\end{lem}
\Proof{
Because of the commutativity of Diagram $(\ref{AAlpDiag})$, the image of $\alpha|_{f^{-1}(\sigma)}$ lies in $g^{-1}(\sigma)$.
Thus $(\alpha_{\ast}p)_{\sigma} \in D_R(g^{-1}(\sigma))$. Now, we prove the compatibility. Given 
$\sigma_1 \subseteq \sigma$ in $\Sigma$. Using the fact that 
$r_{\sigma,\sigma_1}\circ \alpha|_{f^{-1}(\sigma)}=\alpha|_{f^{-1}(\sigma_1)}\circ r_{\sigma,\sigma_1}$ 
we obtain  
$$
\begin{aligned}
(\alpha_{\ast}p)_{\sigma}|_{\sigma_1}
&=
D_R(r_{\sigma,\sigma_1})(D_R(\alpha|_{f^{-1}(\sigma)})(p_{\sigma}))\\
&=
D_R(r_{\sigma,\sigma_1}\circ \alpha|_{f^{-1}(\sigma)})(p_{\sigma})\\
&=
D_R(\alpha|_{f^{-1}(\sigma_1)}\circ r_{\sigma,\sigma_1})(p_{\sigma})\\
&=D_R(\alpha|_{f^{-1}(\sigma_1)}) (D_R\left(r_{\sigma,\sigma_1})(p_{\sigma})\right)
\\
&=D_R(\alpha|_{f^{-1}(\sigma_1)}) (p_{\sigma}|_{\sigma_1})
\\
&=D_R(\alpha|_{f^{-1}(\sigma_1)}) (p_{\sigma_1})= (\alpha_{\ast}p)_{\sigma_1}.
\end{aligned}
$$
For the second part of the lemma, we have
$$
\begin{aligned}
(\alpha'\circ \alpha)_{\ast}(p)_{\sigma}
&=
D_R(\alpha' \circ \alpha |_{f^{-1}(\sigma)})(p_\sigma)\\
&=D_R(\alpha' |_{g^{-1}(\sigma)}\circ \alpha|_{f^{-1}(\sigma)})(p_\sigma) \\
&=D_R(\alpha' |_{g^{-1}(\sigma)})\left(D_R(\alpha |_{f^{-1}(\sigma)})(p_\sigma)\right)\\
&=(\alpha'_{\ast}
(\alpha_\ast p))_{\sigma}.
\end{aligned}
$$
}

\begin{lem}\label{ReplaabEmp}
Given a type II morphism $\alpha:f \to g$, a simplicial complex map 
$\pi: \Sigma' \to \hat N \Sigma$, and an {empirical model} $p \in bEmp(f)$, we have 
$$
\pi_{\ast}(\alpha_{\ast}p)=\left(\pi^{\ast}(\hat N \alpha)\right)_{\ast}(\pi_{\ast}p).
$$
\end{lem}
\Proof{
Given $\sigma' \in \Sigma'$ and $\tilde{\tau} \in (g^{\pi})^{-1}(\sigma')$, we have
$$
\begin{aligned}
\pi_{\ast}(\alpha_\ast p)_{\sigma'}(\tilde{\tau})
&=(\alpha_{\ast} p)_{\overline{\pi}(\sigma')}(\overline{l}_{g}(\tilde{\tau})) \\
&=D_R(\alpha|_{f^{-1}(\overline{\pi}(\sigma'))})(p_{\overline{\pi}(\sigma')})(\overline{l}_{g}(\tilde{\tau}))\\
&=\sum_{\gamma:\, \alpha(\gamma)=\overline{l}_{g}(\tilde{\tau})}p_{\overline{\pi}(\sigma')}(\gamma).
\end{aligned}
$$
On the other hand, 
$$
\begin{aligned}
(\left(\pi^{\ast}(\hat N \alpha)\right)_{\ast}(\pi_{\ast}p))_{\sigma'}(\tilde{\tau})&=
D_R(\pi^{\ast}(\hat N \alpha)|_{(f^{\pi})^{-1}(\sigma')})\left((\pi_{\ast}p)_{\sigma'}\right)(\tilde{\tau})\\
&=\sum_{\tau:\,\pi^{\ast}(\hat N \alpha)(\tau)=\tilde{\tau}}
(\pi_{\ast}p)_{\sigma'}(\tau)\\
&=\sum_{\tau:\,\pi^{\ast}(\hat N \alpha)(\tau)=\tilde{\tau}}p_{\overline{\pi}(\sigma')}(\overline{l}_f(\tau)).
\end{aligned}
$$
Observe that if $\tau \in \pi^{\ast}(\hat N \Gamma)$ such that $\pi^{\ast}(\hat N \alpha)(\tau)=\tilde{\tau}$ then $\overline{l}_f(\tau) \in \Gamma$ and $\alpha(\overline{l}_f(\tau))=\overline{l}_{g}(\pi^{\ast}(\hat N \alpha)(\tau))=
\overline{l}_{g}(\tilde{\tau}))$. In fact, using the universal property of pullbacks, one can see that the assignment $\tau \mapsto \overline{l}_f(\tau)$ gives a one-to-one corresponding between $\{\tau:\,\pi^{\ast}(\hat N \alpha)(\tau)=\tilde{\tau}\}$ and $\{\gamma:\, \alpha(\gamma)=\overline{l}_{g}(\tilde{\tau})\}$.
}

\section{Simplicial sets}\label{sec:simplicial-sets}
\label{sec:simplicial-sets}

{Simplicial sets are combinatorial models of spaces which have better expressive power than simplicial complexes. In this section we give basic definitions and introduce the simplicial set representing the standard topological simplex.}

\Def{\label{def:simp-set}
{\rm
A {\it simplicial set} is a functor $X:\catDelta^\op \to \catSet$. Explicitly, a simplicial set consists of 
the following data:
\begin{itemize}
\item A sequence of sets  $X_0,X_1,\cdots,X_n,\cdots$ for $n\geq 0$ where
each $X_n$ represents the set of $n$-simplices.
\item Face maps
$$
d_i:X_n \to X_{n-1}\;\;\;\text{ for } 0\leq i\leq n
$$
representing the faces of a given simplex.
\item Degeneracy maps
$$
s_j:X_n \to X_{n+1}\;\;\;\text{ for } 0\leq j\leq n
$$
representing the degenerate simplices.
\end{itemize}
The face and the degeneracy maps are subject to the  simplicial identities \cite{goerss2009simplicial}.
A map $f:X\to Z$ of simplicial sets consists of a sequence of functions $f_n:X_n\to Z_n$, where $n\geq 0$, compatible with the face and the degeneracy  maps in the sense that
$$
d_i f_n(\sigma) = f_{n-1}(d_i \sigma)\;\; \text{ and } \;\;s_j f_n(\sigma) = f_{n+1}(s_j \sigma)
$$
for all $0\leq i,j\leq n$ and $\sigma\in X_n$.
}
}


\Ex{\label{ex:d-simplex}
The $d$-simplex $\Delta[d]$ is the simplicial set whose set $n$-simplices are given by 
$$
\set{\sigma^{a_0\cdots a_n}:\, 0\leq a_0\leq \cdots \leq a_d\leq n,\,a_i \in \mathbb{Z}}.
$$
The face maps deletes an index: $d_i(\sigma^{a_0\cdots a_n})= \sigma^{a_0\cdots a_{i-1}a_{i+1}\cdots a_n}$, and the degeneracy maps copies an index: $s_j(\sigma^{a_0\cdots a_n})= \sigma^{a_0\cdots a_{j}a_{j}\cdots a_n}$. The simplex $\sigma^{01\cdots d}$ is the generating simplex of the simplicial set in the sense that any other simplex can be obtained by applying a sequence of face and degeneracy maps to this simplex.
}

\subsection{The nerve {space}}
\label{sec:nerve-space}

{There is a simplicial set version of the nerve complex $\hat N (\Sigma)$ construction which we simply denote by $N(\Sigma)$.
In this section we introduce the nerve space associated to a simplicial complex and study maps between two such simplicial sets.}

\Def{\label{def:N}
{\rm
Let $ N:s\catComp \to \catsSet$ denote the functor that sends a simplicial complex $\Sigma$ to a simplicial set $N\Sigma$, called the {\it nerve space} of $\Sigma$, whose $n$-simplices are given by
$$
(N\Sigma)_n = \set{(\sigma_1,\sigma_2,\cdots,\sigma_n)\in \Sigma^n:\, \cup_{i=1}^n \sigma_i \in \Sigma}.
$$
where $(N\Sigma)_0=\set{()}$.
The simplicial structure maps are given by
$$
d_i(\sigma_1,\sigma_2,\cdots,\sigma_n) = \left \lbrace
\begin{array}{ll}
(\sigma_2,\cdots,\sigma_n) & i=0 \\
(\sigma_1,\cdots,\sigma_i\cup \sigma_{i+1},\cdots,\sigma_n) & 0<i<n \\
(\sigma_1,\cdots,\sigma_{n-1}) & i=n
\end{array}
\right. 
$$
and
$$
s_j(\sigma_1,\sigma_2,\cdots,\sigma_n) = (\sigma_1,\cdots,\sigma_{j},\emptyset,\sigma_{j+1},\cdots,\sigma_n)\;\;\; 0\leq j\leq n. 
$$
A simplicial complex map $f:\Gamma\to \Sigma$ induces a simplicial set map between the nerves $Nf: N\Gamma \to N\Sigma$ defined in degree $n$ by 
$$
(Nf)_n(\gamma_1,\gamma_2,\cdots,\gamma_n) = (f(\gamma_1),f(\gamma_2),\cdots,f(\gamma_n)).
$$
}}

\begin{pro}\label{pro:fNtoN}
Given simplicial complexes 
$\Gamma$ and $\Sigma$, a simplicial set map 
$f: N\Gamma \to N \Sigma$ between the nerves satisfies the following properties:
\begin{enumerate}
\item  
$
f_n(\gamma_1,\cdots,\gamma_n)=(f_1(\gamma_1),\cdots,f_1(\gamma_n))
$
for every $(\gamma_1,\cdots,\gamma_n)\in (N\Gamma)_n$.
\item $f_{1}(\gamma)=\cup_{i=1}^n f_1({x_i})$ 
for every $\gamma=\{x_1,\cdots,x_n\} \in \Gamma$.
\end{enumerate} 
\end{pro}
\Proof{
Part 1: Given $1\leq i \leq n$, let $d^\Gamma: (N\Gamma)_n \to (N\Gamma)_{1}$ denote the composition of $i-1$ times $d_0$ with $(n-i)$ times $d_n$, i.e., $d^\Gamma(\gamma_1,\cdots,\gamma_n)=\gamma_i$. Similarly, we define $d^\Sigma$. Since $f$ respects the simplicial structure we have $d^{\Sigma}(f_n(\gamma_1,\cdots,\gamma_n))=f_1(d^{\Gamma}(\gamma_1,\cdots,\gamma_n))=f_1(\gamma_i)$.

Part 2: Given $\gamma=\{x_1,\cdots,x_n\}\in \Gamma$, we have $(\{x_1\},\cdots,\{x_n\})\in (N\Gamma)_n$. Using part 1 and the fact that $f$ respects the simplicial structure we obtain that 
$$
\begin{aligned}
f_1(\gamma)&=f_1(\{x_1,\cdots,x_n\})\\
&=f_1\left(d_1\circ \cdots \circ d_1(\{x_1\},\cdots,\{x_n\})\right)\\
&=d_1\circ \cdots \circ d_1\left(f_n(\{x_1\},\cdots,\{x_n\})\right)\\
&=d_1\circ \cdots \circ d_1(f_1(x_1),\cdots,f_1(x_n))\\
&=\cup_{i=1}^n f_1(x_i).
\end{aligned}
$$}

\subsection{{Lemmas: Category of simplicial scenarios}}
\label{sec:lem-categoryof-simplicial-scen}

{Similar to the $\hat N$ construction its simplicial set version $N$ also preserves bundle scenarios. In this section we prove the two parts that go into the proof, that is, this construction preserves both local surjectivity and being discrete over vertices.} 

\Lem{\label{lem:nerve-S-LS} 
If $f:\Gamma \to \Sigma$ is (locally) surjective   then $Nf:N\Gamma \to N\Sigma$ is (locally) surjective. 
}
\begin{proof} 
Assume that $f$ is surjective and let $(\sigma_1,\cdots,\sigma_n)\in (N \Sigma)_n$. There exists $\gamma \in \Gamma$ such that $f(\gamma)=\cup_{i=1}^n\sigma_i \in \Sigma$ and we can choose $\gamma_1,\cdots,\gamma_n \subset \gamma$ such that 
$f(\gamma_i)=\sigma_i$. 
This means that $(\gamma_1,\cdots,\gamma_n)\in (N\Gamma)_n$ and $(Nf)_n(\gamma_1,\cdots,\gamma_n)=(\sigma_1,\cdots,\sigma_n)$.

Let us write $\tau=\sigma^{01\cdots (n-1)}$ and $\sigma=\sigma^{01\cdots n}$ for the generating simplices of $\Delta[n-1]$ and $\Delta[n]$, respectively; see Example \ref{ex:d-simplex}.
Now, assume that $f$ is locally surjective and consider the following commutative diagram:
\begin{equation}
\label{DDiag20}
\begin{tikzcd}[column sep=huge,row sep=large]
\Delta[n-1]  \arrow[r,"\tilde \alpha"]
 \arrow[d,hook, "d^i"'] & N\Gamma 
 \arrow[d,"Nf"] \\
\Delta[n] \arrow[ru,dashed,"\beta"] \arrow[r,"\alpha"] & N\Sigma 
\end{tikzcd}
\end{equation}
where $\alpha_{\sigma}=(\sigma_1{,}\cdots{,}\sigma_n)$ 
and $\tilde \alpha_{\tau}=(\gamma_1{,}\cdots{,}\gamma_{n-1})$.
If $0<i<n$ then 
$$
(\sigma_1,\cdots,\sigma_i\cup \sigma_{i+1},\cdots,\sigma_n)=(f(\gamma_1),\cdots,f(\gamma_{n-1})).
$$
where
$f(\gamma_i)=\sigma_i\cup \sigma_{i+1}$.
Let 
$\tilde{\gamma}_i,\tilde{\gamma}_{i+1}\subset \gamma_i$ be such 
that $f(\tilde{\gamma}_i)=\sigma_i$, $f(\tilde{\gamma}_{i+1})=\sigma_{i+1}$ and
$\tilde{\gamma}_i \cup \tilde{\gamma}_{i+1}= \gamma_i$. Thus we have the $n$-simplex $(\gamma_1,\cdots,
\tilde{\gamma}_i,\tilde{\gamma}_{i+1},\cdots,\gamma_{n-1}) \in (N\Gamma)_n$. 
Defining $\beta_{\sigma} = (\gamma_1,\cdots,
\tilde{\gamma}_i,\tilde{\gamma}_{i+1},\cdots,\gamma_{n-1})$ gives a lifting for Diagram (\ref{DDiag20}):
$$
(\beta\circ d^i)_{\tau}=(\gamma_1,\cdots,
\tilde{\gamma}_i\cup \tilde{\gamma}_{i+1},\cdots,\gamma_{n-1})=(\gamma_1,\cdots,\gamma_{n-1})
$$
and
$$
(Nf\circ \beta)_{\sigma} = (f(\gamma_1),\cdots,
f(\tilde{\gamma}_i),f(\tilde{\gamma}_{i+1}),\cdots,f(\gamma_{n-1}))=
(\sigma_1,\cdots,\sigma_n).
$$
Next, let $i=n$. We have 
$$
(\sigma_1,\cdots,\sigma_{n-1})=(f(\gamma_1),\cdots,f(\gamma_{n-1})).
$$
In particular, $f(\cup_{i=1}^{n-1}\gamma_i)=\cup_{i=1}^{n-1}\sigma_i$ and $\cup_{i=1}^{n}\sigma_i\in \St(\cup_{i=1}^{n-1}\sigma_i)$. 
Note that since $f$ is  locally surjective, there exists $\gamma \in \St(\cup_{i=1}^{n-1}\gamma_i)$ such that $f(\gamma)=\cup_{i=1}^{n}\sigma_i$. Therefore, we have $\gamma_n \subseteq \gamma$ with $f(\gamma_n)=\sigma_n$ and defining $\beta_\sigma=(\gamma_1,\cdots,\gamma_{n-1},\gamma_n) $ is a lifting for Diagram (\ref{DDiag20}). A similar argument works for the case of $i=0$. 
\end{proof}


\Lem{\label{lem:nerve-pre-dv}  
The map $Nf:N\Gamma \to N\Sigma$ is discrete over vertices for every simplicial complex map $f:\Gamma \to \Sigma$. 
}
\begin{proof}
As in the proof of Lemma \ref{lem:nerve-S-LS} let us write $\tau$ and $\sigma$ for the generating simplices of the $(n-1)$-simplex and the $n$-simplex.
Consider the following commutative diagram
\begin{equation}
\label{DDiag21}
\begin{tikzcd}[column sep=huge,row sep=large]
\Delta[n]  \arrow[r,"\tilde\alpha"]
 \arrow[d,twoheadrightarrow, "s^i"'] & N\Gamma 
 \arrow[d,"Nf"] \\
\Delta[n-1] \arrow[r,"\alpha"] \arrow[ru,dashed,"\beta"] & N\Sigma 
\end{tikzcd}
\end{equation}
where $\alpha_\tau = (\sigma_1{,}\cdots{,}\sigma_{n-1})$ and $\tilde \alpha_\sigma=(\gamma_1{,}\cdots{,}\gamma_{n})$, which are related by
$$
(f(\gamma_1),\cdots,f(\gamma_n))= (\sigma_1,\cdots,\sigma_{i},\emptyset,\sigma_{i+1},\cdots,\sigma_{n-1}). 
$$
In particular, $f(\gamma_i)=\emptyset$ and hence $\gamma_i=\emptyset$. 
We will prove that $\beta_\tau = (\gamma_1,\cdots,\gamma_{i-1},\gamma_{i+1},\cdots,\gamma_n)$ 
gives a lifting for  Diagram (\ref{DDiag21}): We have 
$$
(\beta \circ s^i)_\sigma =(\gamma_1,\cdots,\emptyset,\cdots,\gamma_n)
=(\gamma_1,\cdots,\gamma_i,\cdots,\gamma_n)
$$
and
$$
(Nf\circ \beta)_\tau=(f(\gamma_1),\cdots,f(\gamma_{i-1}),f(\gamma_{i+1}),\cdots,f(\gamma_n))  
=(\sigma_1,\cdots,\sigma_{n-1}) .
$$  
\end{proof}

\subsection{{Lemmas: Embedding bundle scenarios into simplicial scenarios}}
\label{sec:lem-embedding-into-simplicial-scen}

{
The nerve space construction $N$, which preserves bundle scenarios, in fact gives a functor $N:\catbScen \to \catsScen$. To be able to prove this we need to show that applying $N$ to a morphism between two bundle scenarios produces a morphism between the associated simplicial bundle scenarios. Our main tool is a natural transformation $N\hat N\to N$.
}

\begin{defn}\label{def:tilde-mu}
Given a simplicial complex $\Sigma$, we define a natural map $\tilde{\mu}_{\Sigma}:N\hat N \Sigma \to N \Sigma$ by 
%
%
%
$$
(\tilde{\mu}_{\Sigma})_n
(\{\sigma_{11},\cdots,\sigma_{1m_1}\},\cdots,
\{\sigma_{n1},\cdots,\sigma_{nm_n}\})=(\cup_{i=1}^{m_1}\sigma_{1i},\cdots,\cup_{i=1}^{m_n}\sigma_{ni})
$$
where $(\{\sigma_{11},\cdots,\sigma_{1m_1}\},\cdots,
\{\sigma_{n1},\cdots,\sigma_{nm_n}\}) \in (N \hat N \Sigma)_n$.
\end{defn}

The map $T(\pi):N\Sigma'\to N\Sigma$ given in  Definition \ref{def:T-pi} can be factored as
$$
T(\pi): N\Sigma' \xrightarrow{N\pi} N\hat N\Sigma \xrightarrow{\tilde\mu_\Sigma} N\Sigma 
$$ 
where $\tilde{\mu}_{\Sigma}$ is as in Definition \ref{def:tilde-mu}.
{Next we prove a sequence of results to be used in showing that the nerve space construction gives us a well-defined functor.}

\begin{lem}\label{lem:tildemupull} 
For a type I morphism given in Diagram (\ref{dia:typeI})
the composition of the following squares is a pull-back
\begin{equation}\label{DDiag7}
\begin{tikzcd}[column sep=huge,row sep=large]
 N \Gamma  
 \arrow[d,"N f"'] & \arrow[l,"\tilde{\mu}_{\Gamma}"'] N\hat N \Gamma  
 \arrow[d,"N \hat N f"'] &  N\pi^{\ast}(\hat N \Gamma) 
 \arrow[d," N f^\pi"] \arrow[l,"N l"']\\
 N\Sigma  & \arrow[l,"\tilde{\mu}_{\Sigma}"] 
N \hat N \Sigma  & N \Sigma' \arrow[l,"N \pi"]
\end{tikzcd}
\end{equation}  
In particular, ${T(\pi)^*(N\Gamma)}\cong N\pi^*(\hat N\Gamma)$.
\end{lem}
\Proof{
Let $N \Gamma \times_{N \Sigma} N\Sigma'$ be the  pull-back of the composite $\tilde{\mu}_{\Sigma}\circ N\pi$ along $Nf$. 
Given a simplex $\left((\gamma_1,\cdots,\gamma_n),(\sigma'_1,\cdots,\sigma'_n)\right)$ in $(N \Gamma \times_{N \Sigma} N\Sigma')_n$, we have
$f(\gamma_i)=
\overline{\pi}(\sigma'_i)$ for every $1\leq i \leq n$. 
Let $\sigma'_i=\{y_{i1},\cdots,y_{im_{i}}\}$. Then we have $f(\gamma_i)=\cup_{j=1}^{m_i}\pi(y_{ij})$. Recall that $f$ is discrete over vertices. Thus by Remark \ref{Improppp} there is unique $\gamma_{i1},\cdots, \gamma_{im_i}\subseteq \gamma_i$ such that 
$$
\cup_{j=1}^{m_i}\gamma_{ij}=\gamma_i, \,\, \text{and} \,\,f(\gamma_{ij})=\pi(y_{ij})
\,\, \forall 1\leq j \leq n.
$$
We conclude that $\{(\gamma_{i1},y_{i1}),\cdots,(\gamma_{im_i},y_{im_i})\} \in \pi^\ast(\hat N \Gamma)$ {(see Section \ref{subsec: Pullback})}. We define 
$$
S_n\left((\gamma_1,\cdots,\gamma_n),(\sigma'_1,\cdots,\sigma'_n)\right)
$$ 
to be
$
(\{(\gamma_{11},y_{11}),\cdots,(\gamma_{1m_1},y_{1m_1})\},\cdots,\{(\gamma_{n1},y_{n1}),\cdots,(\gamma_{n m_n},y_{n m_n})\})
\in N(\pi^{\ast}(\hat N \Gamma)).
$
Using Remark \ref{Improppp} one can see that the  construction above give us a simplicial set map 
$$S: N \Gamma \times_{N \Sigma} N\Sigma' \to N\pi^{\ast}(\hat N \Gamma).$$ 
In addition, we have
$$ 
\begin{aligned}
(Nf^{\pi})_n &\left(S_n \left((\gamma_1,\cdots,\gamma_n),(\sigma'_1,\cdots,\sigma'_n)\right) \right)\\
&=(Nf^{\pi})_n(\{(\gamma_{11},y_{11}),\cdots,(\gamma_{1m_1},y_{1m_1})\},\cdots,\{(\gamma_{n1},y_{n1}),\cdots,(\gamma_{n m_n},y_{n m_n})\})\\
&=(\{y_{11},\cdots,y_{1m_1}\},\cdots,\{y_{n1},\cdots,y_{nm_n}\})\\
&=(\sigma'_1,\cdots,\sigma'_n)
\end{aligned}
$$
and
$$
\begin{aligned}
(\tilde{\mu}_{\Gamma})_n & \circ (Nl)_n \circ S_n  \left((\gamma_1,\cdots,\gamma_n),(\sigma'_1,\cdots,\sigma'_n)\right)\\
&=(\tilde{\mu}_{\Gamma})_n \circ (Nl)_n(\{(\gamma_{11},y_{11}),\cdots,(\gamma_{1m_1},y_{1m_1})\},\cdots,\{(\gamma_{n1},y_{n1}),\cdots,(\gamma_{n m_n},y_{n m_n})\})
\\
&=(\tilde{\mu}_{\Gamma})_n(\{\gamma_{11},\cdots,\gamma_{1m_1}\},\cdots,\{\gamma_{n1},\cdots,\gamma_{nm_n}\})\\
&=(\cup_{j=1}^{m_1}\gamma_{1j},\cdots,\cup_{j=1}^{m_n}\gamma_{nj})\\
&=(\gamma_1,\cdots,\gamma_n).
\end{aligned}
$$
Finally, one can check that the induced map from
$N\pi^{\ast}(\hat N \Gamma)$ to $N \Gamma \times_{N \Sigma} N\Sigma'$ is the inverse of $S$. This gives the desired result.
}

\begin{lem}\label{lem:TFun}
We have
$T(\delta_{\Sigma})=\Id_{N \Sigma}$ and $T(\pi \kleisli \pi')=T(\pi)\circ T(\pi')$.
\end{lem}
\Proof{Given $(\sigma_1,\cdots,\sigma_n) \in (N \Sigma)_n$, we have 
$$
T(\delta_{\Sigma})_n(\sigma_1,\cdots,\sigma_n)=(\overline{\delta_{\Sigma}}(\sigma_1),\cdots,\overline{\delta_{\Sigma}}(\sigma_n))=(\sigma_1,\cdots,\sigma_n).
$$
For $(\sigma_1,\cdots,\sigma_n) \in (N \Sigma'')_n$, we have
$$
\begin{aligned}
T(\pi \kleisli \pi')_n(\sigma_1,\cdots,\sigma_n)&=\left(\overline{\pi \kleisli \pi'}(\sigma_1),\cdots,\overline{\pi \kleisli \pi'}(\sigma_n)\right)\\
&=(\overline{\pi} \circ \overline{\pi'}(\sigma_1),\cdots,\overline{\pi} \circ \overline{\pi'}(\sigma_n))\\
&=T(\pi)_n\left( T(\pi')_n(\sigma_1,\cdots,\sigma_n)\right).
\end{aligned}
$$
} 
\begin{lem}\label{lem: ReplaabNN}
Given a type II morphism $\alpha:f \to g$, and a simplicial complex map 
$\pi: \Sigma' \to \hat N \Sigma$, we have 
$$
N(\pi^\ast(\hat N \alpha))=
T(\pi)^\ast(N\alpha).
$$
\end{lem}
\Proof{The target of the map $T(\pi)^\ast(N\alpha)$ is $N\pi^\ast(\hat N \Gamma')$, hence it is uniquely determined by the projections to $N\Gamma'$ and $N\Sigma'$ (see Lemma \ref{lem:tildemupull}). One can {check} that the map $N(\pi^{\ast}(\hat N \alpha))$ has the same projection to $N\Gamma'$ and $N\Sigma'$.}
%

 
\begin{lem}\label{lem:pi-Npi}
Given a simplicial set map $f: N\Sigma' \to N \Sigma$,   there exits a unique map $\pi:\Sigma' \to \hat N\Sigma$ such that $T(\pi)=f$.
\end{lem}
\Proof{ 
We define a map $\pi:\Sigma' \to \hat N\Sigma$ of simplicial complexes by 
$\pi(x)=f_1(x)$ where $x\in \Sigma'_0$. 
For $\sigma'=\{x_1,\cdots,x_n\} \in \Sigma'$, we have 
$\pi(\sigma')=\{f(x_1),\cdots,f(x_n)\}$. By Part (2) of Proposition \ref{pro:fNtoN}, $\cup_{i=1}^n f_1(x_i)=f_1(\sigma') \in \Sigma$. Therefore $\pi(\sigma')$ is a simplex of $\hat N \Sigma$.  
Now, we prove that $T(\pi)=f$. Given $(\{x_{11},\cdots,x_{1m_1}\}, \cdots,\{x_{n1},\cdots,x_{nm_n}\}) \in (N\Sigma')_n$, using Proposition \ref{pro:fNtoN}, we obtain
$$
\begin{aligned}
&T(\pi)_n (\{x_{11},\cdots,x_{1m_1}\}, \cdots,\{x_{n1},\cdots,x_{nm_n}\})\\ 
&=(\tilde{\mu}_{\Sigma})_n\left(N(\pi)_n(\{x_{11},\cdots,x_{1m_1}\}, \cdots,\{x_{n1},\cdots,x_{nm_n}\})\right) \\
&=(\tilde{\mu}_{\Sigma})_n\left(\pi(\{x_{11},\cdots,x_{1m_1}\}), \cdots,\pi(\{x_{n1},\cdots,x_{nm_n}\})\right)
\\
&=(\tilde{\mu}_{\Sigma})_n\left(\{f_1(x_{11}),\cdots,f_1(x_{1m_1})\}, \cdots,\{f_1(x_{n1}),\cdots,f_1(x_{nm_n})\}\right)\\
&=\left(\cup_
{i=1}^{m_1}f_1(x_{1i}), \cdots,\cup_
{i=1}^{m_n}f_1(x_{ni})\right)\\
&=\left(f_1(\{x_{11},\cdots,x_{1m_1}\}), \cdots,f_1(\{x_{n1},\cdots,x_{nm_n}\})\right)
\\
&=
f_n(\{x_{11},\cdots,x_{1m_1}\}, \cdots,\{x_{n1},\cdots,x_{nm_n}\}).
\end{aligned}
$$
The uniqueness follows from the observation that any simplicial complex map $\pi: \Sigma \to \hat N \Sigma'$ with the property that 
$T(\pi)=f$ satisfies
$$
\pi(x)=\overline{\pi}(x)=
T(\pi)_1(x)=f_1(x)
$$ 
for every vertex $x \in \Sigma'_0$.   
}

\begin{lem}\label{lem:Nalph-alpha}
Given bundle scenarios $f: \Gamma \to \Sigma$, $g: \Gamma' \to \Sigma$, and the following commutative diagram of simplicial sets
%
\begin{equation}
\begin{tikzcd}
N\Gamma  
 \arrow[rd,"Nf"'] \arrow[rr,"\tilde{\alpha}"]&& N\Gamma'
 \arrow[dl,"Ng"] \\
& N\Sigma  & 
\end{tikzcd}
\end{equation}
there exists a unique type II morphism $\alpha: f \to g$ such that 
$N\alpha =\tilde{\alpha}$.  
\end{lem}
\Proof{ Observe that for $x \in \Gamma_0$ we have 
$g(\tilde{\alpha}_1(x))=f(x)\in \Sigma_0$. Since $g$ is discrete over vertices, we obtain that $\tilde{\alpha}_1(x)$ is a vertex in $\Gamma'$. We define $\alpha: \Gamma \to \Gamma'$ to be $\alpha(x)=\tilde{\alpha}_1(x)$ for every $x \in \Gamma_0$. By part $2$ of Proposition \ref{pro:fNtoN} we conclude that $\alpha(\gamma)=\tilde{\alpha}_1(\gamma)\in \Gamma'$ for every simplex $\gamma\in \Gamma$. 
%
%
%
Therefore $\alpha$ is a well-defined simplicial complex map. Given $(\gamma_1,\cdots,\gamma_n) \in (N\Gamma)_n$, 
by part $1$ of Proposition \ref{pro:fNtoN} we have  
$$
\tilde{\alpha}_n(\gamma_1,\cdots,\gamma_n)=(\tilde{\alpha}_1(\gamma_1),\cdots,{\alpha}_1(\gamma_n))=({\alpha}(\gamma_1),\cdots,{\alpha}(\gamma_n))
=(N\alpha)_n(\gamma_1,\cdots,\gamma_n).$$
If $N(\alpha_1)=N(\alpha_2)$, then it is clear that $\alpha_1=\alpha_2$. Therefore the map $\alpha$ is unique with the property that $N\alpha=\tilde{\alpha}$.
}

\subsection{{Lemmas: Simplicial distributions functor}}
\label{sec:lem-simplicial-dist}

{In this section we prove results that allow us to push-forward a simplicial distribution along a morphism of $\catsScen$. We begin by showing that the construction $p\mapsto Np$ given in Definition \ref{def:Np} provides a well-defined simplicial distribution.
}

\begin{lem}\label{lem:N-p}
$Np$ is a well-defined simplicial distribution on $Nf$, i.e., it belongs to $\sDist(Nf)$.
\end{lem}
\Proof{
First, we prove that $(Np)_n(\sigma_1,\cdots,\sigma_n) \in D_R((Nf)^{-1}_n(\sigma_1,\cdots,\sigma_n))$. Using Remark \ref{Improppp} we have
$$
\begin{aligned}
\sum_{(\gamma_1,\cdots,\gamma_n)\in (Nf)^{-1}_n(\sigma_1,\cdots,\sigma_n)}(Np)_n(\sigma_1,\cdots,\sigma_n)(\gamma_1,\cdots,\gamma_n)&=
\sum_{(\gamma_1,\cdots,\gamma_n)\in (Nf)^{-1}_n(\sigma_1,\cdots,\sigma_n)}p_{\cup_{i=1}^n\sigma_i}(\cup_{i=1}^n\gamma_i)\\
&=\sum_{\gamma \in f^{-1}(\cup_{i=1}^n\sigma_i)}p_{\cup_{i=1}^n\sigma_i}(\gamma)=1.
\end{aligned}
$$
Now, we prove that $Np$ respect the simplicial structure. Given 
$(\sigma_1,\cdots,\sigma_n) \in (N \Sigma)_n$ and 
$(\tilde{\gamma}_1,\cdots,\tilde{\gamma}_{n-1}) \in
(N \Gamma)_{n-1}$,  for $0<i<n$ we have
$$
\begin{aligned}
D_R(d_i^{N\Gamma})& \left((Np)_n(\sigma_1,\cdots,\sigma_n) \right)(\tilde{\gamma}_1,\cdots,\tilde{\gamma}_{n-1})\\
&=\sum_{(\gamma_1,\cdots,\gamma_n)\,:\,\,d_i^{N\Gamma}(\gamma_1,\cdots,\gamma_n)=(\tilde{\gamma}_1,\cdots,\tilde{\gamma}_{n-1})}(Np)_n(\sigma_1,\cdots,\sigma_n) ({\gamma}_1,\cdots,{\gamma}_n) 
\\
&=\sum_{(\gamma_1,\cdots,\gamma_n) \in (Nf)^{-1}_n(\sigma_1 ,\cdots,\sigma_n)\,:\,\,  
  \gamma_1=\tilde{\gamma}_1,\cdots,\gamma_i \cup \gamma_{i+1}=\tilde{\gamma}_i,\cdots,\gamma_n=\tilde{\gamma}_{n-1}}p_{\cup_{i=1}^n\sigma_i}(\cup_{i=1}^n \gamma_i).
\end{aligned}
$$
By {Remark (\ref{Improppp})} there exists a unique $(\gamma_1,\cdots,\gamma_n) \in (Nf)^{-1}_n(\sigma_1 ,\cdots,\sigma_n)$ such that $\cup_{i=1}^n \gamma_i=\cup_{i=1}^{n-1} \tilde{\gamma}_i$. 
Therefore the sum above is equal to $p_{\cup_{i=1}^n\sigma_i}(\cup_{i=1}^{n-1} \tilde{\gamma}_i)$. On the other hand, we have 
$$
\begin{aligned}
(Np)_{n-1} & \left(d_i^{N\Sigma}(\sigma_1,\cdots,\sigma_n)\right)(\tilde{\gamma}_1,\cdots,\tilde{\gamma}_{n-1})\\
&=
(Np)_{n-1}(\sigma_1,\cdots,\sigma_i\cup\sigma_{i+1},\cdots,\sigma_n)(\tilde{\gamma}_1,\cdots,\tilde{\gamma}_{n-1})\\
&=p_{\cup_{i=1}^n\sigma_i}(\cup_{i=1}^{n-1} \tilde{\gamma}_i).
\end{aligned}
$$
For $i=n$,   using {Equation (\ref{eq:Improppp})} we have
$$
\begin{aligned}
D_R(d_n^{N\Gamma})& \left((Np)_n(\sigma_1,\cdots,\sigma_n) \right)(\tilde{\gamma}_1,\cdots,\tilde{\gamma}_{n-1})\\
&=\sum_{\gamma \in f^{-1}(\sigma_n)}(Np)_n(\sigma_1,\cdots,\sigma_n) (\tilde{\gamma}_1,\cdots,\tilde{\gamma}_{n-1},\gamma)\\
&=\sum_{\gamma \in f^{-1}(\sigma_n)} p_{\cup_{i=1}^n\sigma_i}(\cup_{i=1}^{n-1} \tilde{\gamma}_i \cup \gamma) \\
&=\sum_{\cup_{i=1}^{n-1} \tilde{\gamma}_i \subset \gamma' \in f^{-1}(\cup_{i=1}^{n} \sigma_i)} p_{\cup_{i=1}^n\sigma_i}(\gamma') \\
&=p_{\cup_{i=1}^n\sigma_i}|_{\cup_{i=1}^{n-1}\sigma_i}(\cup_{i=1}^{n-1} \tilde{\gamma}_i) \\
&=p_{\cup_{i=1}^{n-1}\sigma_i}(\cup_{i=1}^{n-1} \tilde{\gamma}_i)
\\
&=(Np)_{n-1}(\sigma_1,\cdots,\sigma_{n-1})(\tilde{\gamma}_1,\cdots,\tilde{\gamma}_{n-1})\\
&=(Np)_{n-1}\left(d^{N\Sigma}_n(\sigma_1,\cdots,\sigma_{n})\right)(\tilde{\gamma}_1,\cdots,\tilde{\gamma}_{n-1}).
\end{aligned}
$$
The $i=0$ case is similar. For the degeneracy maps,  one can see that for $0 \leq i \leq n$ both 
$$
D_R(s_i^{N\Gamma}) \left((Np)_n(\sigma_1,\cdots,\sigma_n) \right)(\tilde{\gamma}_1,\cdots,\tilde{\gamma}_{n+1})
 \,\,\, \text{and}\,\,\, (Np)_{n+1}  \left(s_i^{N\Sigma}(\sigma_1,\cdots,\sigma_n)\right)(\tilde{\gamma}_1,\cdots,\tilde{\gamma}_{n+1})$$ 
 are equal to 
$
p_{\cup_{i=1}^n \sigma_i}
(\cup_{i=1}^{n+1} \tilde{\gamma}_{i})
$ 
if $\tilde{\gamma}_i=\emptyset$; otherwise, both are zero.  
}

\begin{lem}\label{lem:PinNF}
Let $f: \Gamma \to \Sigma$ be a simplicial complex map that is discrete over vertices.
Given $p \in \sDist(Nf)$ and $(\sigma_1,\cdots,\sigma_n)\in (N\Sigma)_n$, we have
$$
p_n(\sigma_1,\cdots,\sigma_n)(\gamma_1,\cdots,\gamma_n)=p_1(\cup_{i=1}^n \sigma_i)(\cup_{i=1}^n \gamma_i)
$$
where  $(\gamma_1,\cdots,\gamma_n) \in
(Nf)^{-1}_n(\sigma_1,\cdots,\sigma_n)$.
\end{lem}
\begin{proof}
Since $p$ respects the simplicial structure, by Remark \ref{Improppp} we obtain 
$$
\begin{aligned}
p_1(\cup_{i=1}^n \sigma_i)(\cup_{i=1}^n \gamma_i)&=
p_1(d_1\circ\cdots\circ d_1(\sigma_1,\cdots,\sigma_n))(\cup_{i=1}^n \gamma_i)\\
&=
D_R(d_1\circ\cdots\circ d_1)(p_n(\sigma_1,\cdots,\sigma_n))(\cup_{i=1}^n \gamma_i)\\
&=\sum_{f(\tau_i)=
\sigma_i\,\,\text{and} \,\, \cup_{i=1}^n \tau_i=\cup_{i=1}^{n}\gamma_i}p_n(\sigma_1,\cdots,\sigma_n)(\tau_1,\cdots,\tau_n)\\
&=p_n(\sigma_1,\cdots,\sigma_n)(\gamma_1,\cdots,\gamma_n)  .
\end{aligned}
$$
\end{proof}

\begin{lem}\label{lem:Nalphaast}
 Let  $\alpha: f \to g$ be a type II morphism. For $p\in \bEmp(f)$, we have  
$N(\alpha_\ast p)= D_R(N\alpha) \circ Np$.

\end{lem}
\begin{proof}
Given $(\sigma_1,\cdots,\sigma_n) \in (N \Sigma)_n$ and $(\gamma'_1,\cdots,\gamma'_n) \in (Ng)_n^{-1}(\sigma_1,\cdots,\sigma_n)$, we have 
$$
\begin{aligned}
N(\alpha_\ast p)_n & (\sigma_1,\cdots,\sigma_n)
(\gamma'_1,\cdots,\gamma'_n)\\
&=
(\alpha_{\ast}p)_{\cup_{i=1}^n \sigma_i}(\cup_{i=1}^n \gamma'_i)\\
&=D_R(\alpha|_{f^{-1}(\cup_{i=1}^n \sigma_i)})
(p_{\cup_{i=1}^n \sigma_i})(\cup_{i=1}^n \gamma'_i)\\
&=\sum_{\gamma \in
f^{-1}(\cup_{i=1}^n \sigma_i) \,:\,\, \alpha(\gamma)=\cup_{i=1}^n \gamma'_i}p_{\cup_{i=1}^n \sigma_i}(\gamma).
\end{aligned}
$$
On the other hand, 
$$
\begin{aligned}
D_R(N\alpha)_n &  \left((Np)_n(\sigma_1,\cdots,\sigma_n)\right)
(\gamma'_1,\cdots,\gamma'_n)\\
&=
\sum_{(\gamma_1,\cdots,\gamma_n)\,:\,\,(N\alpha)_n(\gamma_1,\cdots,\gamma_n)=(\gamma'_1,\cdots,\gamma'_n)}(Np)_n(\sigma_1,\cdots,\sigma_n)
(\gamma_1,\cdots,\gamma_n)\\
&=\sum_{(\gamma_1,\cdots,\gamma_n)\,:\,\,
\gamma_i \in f^{-1}(\sigma_i)\, \text{and} \,\alpha(\gamma_i)=\gamma'_i}
p_{\cup_{i=1}^n \sigma_i}(\cup_{i=1}^n\gamma_i).
\end{aligned}
$$
Note that $\alpha$ is discrete over vertices since $f$ is discrete over vertices. Therefore
by Remark \ref{Improppp} we obtain an equality.
\end{proof}

\begin{lem}\label{lem:Npiaast}
Let $\pi : \Sigma' \to \hat N \Sigma$ be a simplicial complex map  and  $f:\Gamma \to \Sigma$ be a bundle scenario. For $p\in \bEmp(f)$, we have 
$N(\pi_\ast p)=T(\pi)_\ast(Np)$. 
\end{lem}
\begin{proof}
Note that $N(\pi_\ast p) \in sDist(Nf^{\pi})$, meanwhile, 
$T(\pi)_\ast(Np)\in sDist((Nf)^{T(\pi)})$. By Lemma \ref{lem:tildemupull} we can Identify 
$(Nf)^{T(\pi)}$ with $Nf^{\pi}$. 
Given 
$$(\sigma_1,\cdots,\sigma_n)\in (N\Sigma')_n \;\;\;\; \text{ and } \;\;\;\; (\gamma_1,\cdots,\gamma_n) \in (Nf^\pi)^{-1}(\sigma_1,\cdots,\sigma_n)$$ 
we have
$$
\begin{aligned}
N(\pi_\ast p)_n  (\sigma_1,\cdots,\sigma_n)(\gamma_1,\cdots,\gamma_n)
&=(\pi_\ast p)_{\cup_{i=1}^n \sigma_i}(\cup_{i=1}^n \gamma_i)\\
&=p_{\overline{\pi}(\cup_{i=1}^n \sigma_i)}(\overline{l}\left(\cup_{i=1}^n \gamma_i)\right)
\\
&=p_{\cup_{i=1}^n \overline{\pi}(\sigma_i)}\left(\cup_{i=1}^n \overline{l}(\gamma_i)\right)\\
&=(Np)_n\left(\overline{\pi}(\sigma_1),\cdots,\overline{\pi}(\sigma_n)\right)\left(\overline{l}(\gamma_1),\cdots,\overline{l}(\gamma_n)\right)\\
&=
(Np)_n\left((\tilde{\mu}_{\Sigma}\circ N \pi)_n(\sigma_1,\cdots,\sigma_n)\right)\left((\tilde{\mu}_{\Gamma}\circ N l)_n (\gamma_1,\cdots,\gamma_n)
\right)\\
&=\left(T(\pi)_\ast(Np)\right)_n(\sigma_1,\cdots,\sigma_n)(\gamma_1,\cdots,\gamma_n).
\end{aligned}
$$

\end{proof}

\section{Convex categories}\label{sec:ConvCat}
\label{sec:lem-conv-cat}

In this section we recall the notion of a convex category introduced in \cite{kharoof2022simplicial} and some basic properties of these objects. 
In the abstract setting convexity is defined using algebras over a monad. 
A {\it monad} on a category $\catC$ is a functor $T:\catC\to \catC$ together with natural transformations $\delta: \idy_\catC \to T$ and $\mu:T^2 \to T$ (satisfying certain conditions).
A {\it $T$-algebra} consists of an object $X$ of
$\catC$ together with a structure map given by a morphism 
$$\nu=\nu^X:T(X) \to X$$
 of $\catC$  
such that $\nu\circ\delta_X=\idy_X$ and the following diagram commutes  
\begin{equation}\label{diag:Talg}
		\begin{tikzcd} [column sep=huge,row sep=large] 
			T^2(X) \arrow[r,"T(\nu)"] \arrow[d,"\mu_X"'] & T(X) \arrow[d,"\nu"] \\
			T(X) \arrow[r,"\nu"] & X
		\end{tikzcd}
\end{equation}
A  {\it morphism of $T$-algebras} is a morphism $f:X \to Y$ of $\catC$ that commutes with the structure maps.
The  {\it category of $T$-algebras} will be denoted by $\catC^T$.
The object $T(X)$ together with the structure morphism $\mu_{X}$ is called a  {\it free $T$-algebra}. There is an adjunction $T:\catC \adjoint \catC^T:U$ where $T$ sends an object to the associated free $T$-algebra and $U$ is the forgetful functor. 
Under the bijection $\catC^T(T(X),Y)\cong \catC(X,Y)$ a morphism $f:T(X)\to Y$ is sent to $f\circ \delta_X$. Conversely, under this isomorphism, a morphism $g:X\to Y$ is sent to $\nu^Y\circ T(g)$. The {\it Kleisli category} of $T$, denoted by $\catC_T$, is the category whose objects are the same as the objects of $\catC$ and morphisms $X\to Y$ are given by $\catC(X,{T}(Y))$. Let   
\begin{equation}\label{eq:FT}
F_{T} : \catC \to\catC_{T}
\end{equation}
denote the functor defined as identity on objects and by sending a morphism $f: X \to Y$ to  the composite $F_T(f):X \xrightarrow{f} Y \xrightarrow{\delta_Y} TY$. See \cite[Section VI]{mac2013categories} and \cite[Chapter 5]{riehl2017category}.


Recall that the functor $D_R:\catSet\to \catSet$ is a monad. Algebras over this monad are called {\it $R$-convex sets}.
We will write $\catConv_R$ for the {\it category of $R$-convex sets}, i.e., for the category $\catSet^{D_R}$ of algebras. 
Applying the distribution functor level-wise one can extend to a functor $D_R:\catsSet \to \catsSet$, which turns out to be a monad as well.
The resulting algebras are called simplicial convex sets, and the category of these objects will be denoted by $\catsConv_R$.

\Pro{\label{pro:2-15} 
The functor $\catsSet(-,-) : \catsSet^{op} \times \catsSet \to \catSet$ restricts to a functor
$$
\catsSet(-,-): \catsSet^{op} \times s\catConv_R \to \catConv_R.
$$
}
\Proof{
See Proposition 2.15 in \cite{kharoof2022simplicial}. 
}

The simplicial set $D_R(Y)$ is an object of $\catsConv_R$.
Therefore $\catsSet(X,D_R(Y))$ is an $R$-convex set and the
 convex structure  is defined by  
\begin{equation}\label{eq:sSet-nu-eq10}
\nu^{\catsSet(X,D_R(Y))}(Q)_n(x)  = \sum_{p\in \catsSet(X,D_R(Y))} Q(p) p_n(x)
\end{equation}
where $Q\in D_R (\catsSet(X,D_R(Y)))$ and $x\in X_n$. 
Let $\Theta_{X,Y}$ denote the transpose
of $(\delta_Y)_\ast : \catsSet(X,Y) \to \catsSet(X,D_R(Y))$ in $\catConv_R$ with respect to the adjunction $\catSet\adjoint \catConv_R$.
More explicitly, we have
\begin{equation}\label{eq:Thetapi}
\Theta_{X,Y}=\nu^{\catsSet(X,D_R(Y))}\circ D_R((\delta_Y)_\ast).
\end{equation}

The starting point for passing to convex categories is the observation that the distribution monad $D_R:\catSet \to \catSet$ can be upgraded to a monad on the category of locally small categories $D_R:\catCat \to \catCat$. 
\Def{\label{def:convex-cat}
A $D_R$-algebra in the category of locally small categories $\catCat$ is called an {\it $R$-convex category}. We will write $\catConvCat_R$
for the {\it category of $R$-convex categories} (see  \cite[Section 3.1]{kharoof2022simplicial}
 for more details).
}

\Def{\label{def:Free CatFun}
For a locally small category $\catC$, the {\it free $R$-convex category} $D_R(\catC)$ is defined as follows:
\begin{itemize}
\item Its objects are the same objects as $\catC$.
\item Morphisms are given by $D_R(\catC)(X,Y)=D_R(\catC(X,Y))$.
\end{itemize}
For a functor $F:\catC \to \catD$ we define a functor $D_R(F):D_R(\catC)\to D_R(\catD)$ between the free $R$-convex categories by $X\mapsto F(X)$ on objects and by $F_{X,Y}\mapsto D_R(F_{X,Y})$ where $F_{X,Y}:\catC(X,Y)\to \catD(X,Y)$
}

\begin{pro}\label{pro:3-14}
The transpose of the functor $F_{D_R}: \catsSet \to \catsSet_{D_R}$  (see Equation (\ref{eq:FT}))
		with respect to the 
		adjunction $D_R: \catCat \adjoint \catConvCat_R: U$ 
		is the functor 
		$\Theta:  D_R(\catsSet) \to \catsSet_{D_R}$   defined as identity on the objects and as the map $\Theta_{X,Y}$ in Equation (\ref{eq:Thetapi})  
		on morphisms. 
\end{pro}
\begin{proof}
See \cite[Proposition 3.14]{kharoof2022simplicial}
\end{proof}

In the next two results we will use the version $\Theta:D_R(\catSet)\to \catSet_{D_R}$ for the category of sets.

\lem{\label{ThetaConvRes}
For $R$-convex sets $X,Y$ and $Q \in D_R(\catConv_R(X,Y))$, we have
\begin{equation}
\nu^Y \circ \Theta_{X,Y}(Q)\circ \nu^X=\nu^Y \circ \mu_Y \circ D_R(\Theta_{X,Y}(Q)).
\end{equation}
%
}
\begin{proof}
Given $p \in D_R(X)$, we have
$$
\Theta_{X,Y}(Q)(\nu^X(p))=
\sum_{\varphi \in \catConv_R(X,Y)} Q(\varphi)\delta^{\varphi(\nu^X(p))}
=\sum_{\varphi \in \catConv_R(X,Y)}Q(\varphi)\delta^{\nu^Y\left(D_R(\varphi)(p)\right)}.
$$
On the other hand 
$$
\begin{aligned}
\mu_{Y}\left(D_R(\Theta_{X,Y}(Q))(p)\right)&=\sum_{q \in D_R(Y)} \left(D_R(\Theta_{X,Y}(Q))(p)\right)(q)q\\
&=
\sum_{q \in D_R(Y)} \left(\sum_{x \in X:\, (\Theta_{X,Y}(Q))(x)=q}p(x)\right)q\\
&=\sum_{x\in X}p(x) 
(\Theta_{X,Y}(Q))(x) \\
&= \sum_{x\in X}p(x)
\sum_{\varphi \in \catConv_R(X,Y)} Q(\varphi)\delta^{\varphi(x)} \\
&= 
\sum_{\varphi \in \catConv_R(X,Y)} Q(\varphi)\sum_{x\in X}p(x)\delta^{\varphi(x)}\\
&= 
\sum_{\varphi \in \catConv_R(X,Y)} Q(\varphi)D_R(\varphi)(p).
\end{aligned}
$$    
We define $\tilde{Q} \in D_R(D_R(Y))$ by 
%
$
\tilde{Q}(q)=
\sum_{\varphi: \,D_R(\varphi)(p)=q }Q(\varphi)  
$ for $q \in D_R(Y)$. 
This is a well-defined distribution since
$$
\sum_{q\in D_R(Y)}\tilde{Q}(q)=
\sum_{q\in D_R(Y)} \;\sum_{\varphi: \,D_R(\varphi)(p)=q }Q(\varphi) 
=\sum_{\varphi \in \catConv_R(X,Y)}Q(\varphi)=1 .
$$
In addition, we have
$$
D_R(\nu^Y)(\tilde{Q})=\sum_{q \in D_R(Y)}\tilde{Q}(q)\delta^{\nu^Y(q)}=
\sum_{\varphi \in \catConv_R(X,Y)}Q(\varphi)\delta^{\nu^Y\left(D_R(\varphi)(p)\right)}
$$
and
$$
\mu_Y(\tilde{Q})=\sum_{q \in D_R(Y)}\tilde{Q}(q)q=
\sum_{\varphi \in \catConv_R(X,Y)}Q(\varphi)D_R(\varphi)(p).
$$
%
By Diagram (\ref{diag:Talg})  we obtain the result.  
\end{proof}

\Pro{\label{pro:ConvisConv}
The category $\catConv_R$ is an $R$-convex category.
}
\begin{proof}
Given $R$-convex sets $X$,$Y$, by  Proposition \ref{pro:2-15}   we obtain that $\catSet(X,Y)$ is an $R$-convex set.
Next, we prove that $\catConv_R(X,Y)$ is an $R$-convex subset of $\catSet(X,Y)$. 
Given $Q \in D_R(\catConv_R(X,Y))$, 
we will show that $\nu(Q)=\nu^Y \circ (\Theta_{X,Y}(Q))$ is an 
 $R$-convex map.
Using Lemma \ref{ThetaConvRes} and Diagram (\ref{diag:Talg}) we have
$$
\begin{aligned}
\nu(Q) \circ \nu^X
&=\nu^Y \circ \Theta_{X,Y}(Q)\circ \nu^X\\
&=\nu^Y \circ \mu_Y \circ D_R(\Theta_{X,Y}(Q))\\
&=\nu^Y \circ D_R(\nu^Y) \circ D_R(\Theta_{X,Y}(Q))\\
&=
\nu^Y \circ D_R(\nu(Q)).
\end{aligned}
$$ 
Now, given $P \in D_R(\catConv_R(X,Y))$ and
$Q \in D_R(\catConv_R(Y,Z))$,  
we will prove that $\nu(Q \ast P)=\nu(Q)\circ \nu(P)$. 
%
Using Proposition \ref{pro:3-14} and Lemma \ref{ThetaConvRes} we have
$$
\begin{aligned}
\nu(Q \ast P) 
&=\nu^Z \circ \left( \Theta_{X,Z}(Q \ast P)\right)\\
&= \nu^Z \circ
\left(\Theta_{Y,Z}(Q)\diamond \Theta_{X,Y}(P)\right) \\
&= \nu^Z \circ
\mu_Z \circ D_R(\Theta_{Y,Z}(Q))\circ \left( \Theta_{X,Y}(P)\right) \\
&=
\nu^Z \circ \Theta_{Y,Z}(Q)\circ \nu^Y \circ \left( \Theta_{X,Y}(P)\right)\\
&=\nu(Q)\circ \nu(P).
\end{aligned}
$$
%
%
\end{proof}

\Pro{\label{pro:Nattrans} 
Let $\catC$ be a (locally) small category and $\catE$ be an $R$-convex category.
Consider two functors $F,G:\mathbf{C} \to \mathbf{E}$
and a natural transformation $\eta: F \to G$.
Let $\tilde{F}$ and $\tilde{G}$ be the transpose of $F$ and $G$
with respect to the adjunction $\Cat \dashv \catConvCat_R$; respectively.
Then $\eta$ lifts to a natural transformation $\tilde \eta:\tilde F\to \tilde G$.
}
\begin{proof}
We define $\tilde \eta_X = \eta_X$ at an object $X$ of $D_R(\catC)$.
Given $Q \in D_R(\mathbf{C})(X,Y)$, we will prove that $\eta_Y \circ \tilde{F}_{X,Y}(Q)=\tilde{G}_{X,Y}(Q)\circ \eta_X$. Firstly, we have
$$
\tilde{F}_{X,Y}(Q)=\pi_{F(X),F(Y)}\left(D_R(F_{X,Y})(Q)\right)=\pi_{F(X),F(Y)}
\left(\sum_{f\in C(X,Y)} Q(f) \delta^{F_{X,Y}(f)}\right),
$$
see Remark 2.1 in \cite{kharoof2022simplicial} . Similarly, we have 
$$
\tilde{G}_{X,Y}(Q)=\pi_{G(X),G(Y)}
\left(\sum_{f\in C(X,Y)} Q(f) \delta^{G_{X,Y}(f)}\right).
$$
Now, by the functoriality of $\pi$ and Lemma 3.2 from \cite{kharoof2022simplicial}
we have
$$
\begin{aligned}
\eta_Y \circ \tilde{F}_{X,Y}(Q)&=\pi_{F(Y),G(Y)}(\delta^{\eta_Y})\circ \pi_{F(X),F(Y)}
\left(\sum_{f\in C(X,Y)} Q(f) \delta^{F_{X,Y}(f)}\right)\\
&=\pi_{F(X),G(Y)}\left(\delta^{\eta_Y}\ast 
(\sum_{f\in C(X,Y)} Q(f) \delta^{F_{X,Y}(f)})\right)\\
&=\pi_{F(X),G(Y)} 
\left(\sum_{f\in C(X,Y)} Q(f) \delta^{\eta_Y\circ F_{X,Y}(f)}\right).
\end{aligned}
$$ 
Similarly, we have 
$$
\tilde{G}_{X,Y}(Q) \circ \eta_X = \pi_{F(X),G(Y)} 
\left(\sum_{f\in C(X,Y)} Q(f) \delta^{ G_{X,Y}(f)\circ \eta_X } \right).
$$ 
The result follows from the naturality of $\eta$.
\end{proof}

\end{document}